\documentclass{article}
\usepackage[utf8]{inputenc}
\usepackage[margin=1in]{geometry}
\usepackage{amsmath}
\usepackage{amsthm}
\usepackage{amssymb}
\usepackage{graphicx}
\usepackage{siunitx}
\usepackage{graphics}
\usepackage{booktabs}
\usepackage{hyperref}
\usepackage{cite}
\usepackage{booktabs}
\usepackage{enumerate}
\usepackage{enumitem}
\usepackage{authblk}

\usepackage{algorithm}
\usepackage{algorithmicx}
\usepackage[noend]{algpseudocode}
\algrenewcommand\algorithmicrequire{\textbf{Precondition:}}
\newcommand*\Let[2]{\State #1 $\gets$ #2}

\include{refs.bib}

\newcommand{\bfx}{\ensuremath{\mathbf{x}}}
\newcommand{\bfy}{\ensuremath{\mathbf{y}}}
\newcommand{\bfz}{\ensuremath{\mathbf{z}}}
\newcommand{\bfh}{\ensuremath{\mathbf{h}}}
\newcommand{\bfxi}{\ensuremath{\boldsymbol{\xi}}}
\newcommand{\bfgamma}{\ensuremath{\boldsymbol{\gamma}}}
\newcommand{\bfnu}{\ensuremath{\boldsymbol{\nu}}}

\newcommand{\bfu}{\ensuremath{\mathbf{u}}}
\newcommand{\bfv}{\ensuremath{\mathbf{v}}}
\newcommand{\bfn}{\ensuremath{\mathbf{n}}}
\newcommand{\bfe}{\ensuremath{\mathbf{e}}}
\newcommand{\bff}{\ensuremath{\mathbf{f}}}

\newcommand{\bfg}{\ensuremath{\mathbf{g}}}

\newcommand{\bbr}{\ensuremath{\mathbb{R}}}
\newcommand{\jump}[1]{\ensuremath{[#1]}}
\newcommand{\avg}[1]{\ensuremath{\{#1\}}}
\newcommand{\subbar}[2]{\ensuremath{\left. #1 \right|_{#2}}}

\providecommand{\e}[1]{\ensuremath{\times 10^{#1}}}

\newcommand{\bigo}[1]{\mathcal{O}(#1)}

\DeclareMathOperator{\Tr}{Tr}

\newtheorem{proposition}{Proposition}

\title{Jump splicing schemes for elliptic interface problems and the incompressible Navier-Stokes equations}

\author[ ]{Ben Preskill\thanks{Corresponding author. E-mail address: bpreskill@berkeley.edu.}}
\author[ ]{James A. Sethian}
\affil[ ]{Department of Mathematics, University of California, Berkeley}

\date{May 2015}

\begin{document}
\maketitle

\begin{abstract}

We present a general framework for accurately evaluating finite difference operators in the presence of known discontinuities across an interface. Using these techniques, we develop simple-to-implement, second-order accurate methods for elliptic problems with interfacial discontinuities and for the incompressible Navier-Stokes equations with singular forces. To do this, we first establish an expression relating the derivatives being evaluated, the finite difference stencil, and a compact extrapolation of the jump conditions. By representing the interface with a level set function, we show that this extrapolation can be constructed using dimension- and coordinate-independent normal Taylor expansions with arbitrary order of accuracy. Our method is robust to non-smooth geometry, permits the use of symmetric positive-definite solvers for elliptic equations, and also works in 3D with only a change in finite difference stencil. We rigorously establish the convergence properties of the method and present extensive numerical results. In particular, we show that our method is second-order accurate for the incompressible Navier-Stokes equations with surface tension.
\end{abstract}

\section{Introduction}
\label{sec:intro}

Elliptic interface problems of the form
\begin{equation} \label{eqn:jump-elliptic}
\left\{ \setlength\arraycolsep{2pt} \begin{array}{rll}
-\nabla \cdot (\beta \nabla u) &= f & \text{on } \Omega \setminus \Gamma \\
u &= h &\text{on } \partial \Omega \\
\jump{u} &= g^0 & \text{across } \Gamma \\
\jump{\beta \partial_\bfn u} &= g^1 & \text{across } \Gamma
\end{array} \right.
\end{equation}
arise in a wide variety of applications in physics and engineering, including electrodynamics, fluid mechanics, heat transfer, and shape optimization. Here $\Omega \subset \bbr^d$ is a domain of interest, $\Gamma \subset \Omega$ is a smooth, closed, codimension-one interface, $\bfn$ is a unit normal to $\Gamma$, $\partial_\bfn u = \nabla u \cdot \bfn$, and we define the ``jump'' in $u$ as
\[ \jump{u}(\bfx) = u^+(\bfx) - u^-(\bfx), \]
where $u^\pm(\bfx) = \lim_{\epsilon \to 0^+} u(\bfx \pm \epsilon \bfn)$. Both $u$ and $\beta$ may be discontinuous across the interface, but are otherwise smooth.

Problems of the form \eqref{eqn:jump-elliptic} often occur in the discretization of time-dependent free interface problems. For example, elliptic interface problems must be solved when projection methods for the Navier-Stokes equations are applied in the context of singular forces on an interface, as in the case of surface tension or membrane elasticity.

One approach to solving \eqref{eqn:jump-elliptic} is through a finite element method acting on an unstructured mesh fitted to the interface $\Gamma$. However, when the interface is evolving, as in time-dependent problems with a free surface, remeshing has complications and stability drawbacks. As an alternative to remeshing, immersed boundary, immersed interface, and embedded boundary methods have been developed to solve \eqref{eqn:jump-elliptic} on unfitted meshes, and in particular on Cartesian grids.

An an alternative, in this paper we introduce the ``jump splice'', a general finite difference approach to approximating, with arbitrary order of accuracy, differential operators in the presence of discontinuities across an interface. We do so by extending jump conditions off of the interface and creating a normal Taylor expansion that fully captures the jump structure of the solution across the interface. This leads to an auxiliary set of equations that we can then solve with high accuracy to build the solution. Our approach has links to previous techniques developed to solve \eqref{eqn:jump-elliptic}, but the mathematical simplicity of the jump splice provides numerous advantages:
\begin{itemize}
\item{The approach has rigorous convergence estimates.}
\item{It can be used with arbitrary finite difference operators and arbitrary-order jump conditions.}
\item{It is straightforward to implement in both 2D and 3D.}
\item{The method makes use of coordinate-free normal derivatives and surface gradients.} 
\item{It avoids component-by-component dimensional reduction, and instead formulates the problem with respect to the
jump conditions and the implicitly defined geometry of the interface, independent of grid-interface orientation.}
\item{The method is not limited to achieving an $\bigo{h}$ truncation error near the interface.}
\end{itemize}
We use these techniques to solve elliptic interface problems and the singular force Navier-Stokes equations with second-order accuracy as well as perform quadrature on implicitly defined interfaces with fourth order accuracy. Much of our discussion here parallels the presentation in \cite{preskill:2015}, where more extensive results are shown. 

The remainder of the paper is structured as follows. In the next section, we review existing work on methods for elliptic interface problems and the singular force Navier-Stokes equations; in Section \ref{sec:splice} we develop the mathematical foundations for the jump splice and describe how to evaluate arbitrary finite difference operators in the presence of discontinuities; in Section \ref{sec:elliptic}, we describe how the jump splice leads to a simple method for solving elliptic equations and show extensive convergence results; in Section \ref{sec:integration}, we briefly discuss an application to integration on implicitly defined domains and show convergence results; and finally in Section \ref{sec:ns}, we develop a fully second-order method for the singular force Navier-Stokes equations based on jump splice methodology and show detailed convergence analysis for the case of surface tension.

\section{Previous Work}
\label{sec:previous-work}

Peskin's Immersed Boundary Method (IBM) \cite{peskin:1977, peskin:2002} is a first-order accurate finite difference approach to solving both \eqref{eqn:jump-elliptic} as well as the singular force Navier-Stokes equations. By using smooth approximations to the Dirac $\delta$ function, the IBM approximates jump conditions and singular forces defined on the interface with source terms defined on an underlying grid. The IBM is straightforward to implement, but does not sharply resolve discontinuities due to the use of a smoothing operation. See \cite{lai:2000, roma:1999, fadlun:2000, cortez:2000} for further development of the IBM, including a formally second-order accurate approach as well as use in complex 3D fluid flow. In \cite{tornberg:2004, tornberg:2004:2, engquist:2005}, Tornberg and Engquist generalize the IBM approach and allow for higher-order approximations of singular source terms. See also \cite{mittal:2005} for a review of IBM techniques.

A second-order finite difference approach to solving \eqref{eqn:jump-elliptic} is the Immersed Interface Method of LeVeque and Li \cite{leveque:1994}. Designed to solve elliptic interface problems without smoothing, the IIM uses coordinate-split Taylor expansions to integrate jump conditions into the finite difference stencil of the elliptic operator, thereby obtaining $\bigo{h}$ local truncation error in the vicinity of the interface. The IIM retains the standard 5-point stencil when $\beta$ is smooth, but leads to a non-symmetric system derived from a local constraint problem when $\beta$ is discontinuous across the interface. The IIM generally requires component-wise evaluation of derivatives of the jump conditions along the interface, which can lead to subtle implementation details, particularly in 3D. The works \cite{li:2001:2, deng:2003, chen:2008, wiegmann:2000, berthelsen:2004, li:1998, adams:2005} describe further development of the IIM for elliptic problems. The IIM has also been used extensively for solving the Stokes and Navier-Stokes equations in the presence of singular forces \cite{leveque:1997, tan:2008, le:2006, xu:2008, xu:2006:2, li:2001, leveque:2003}. A comprehensive overview of the IIM can be found in \cite{li:2006}.

Another finite difference approach introduces fictitious degrees of freedom on a Cartesian grid with values determined by the jump conditions through extrapolation; see, for example, the Ghost Fluid Method (GFM) \cite{liu:2000}. The GFM as formulated in \cite{liu:2000} achieves a fully symmetric linear discretization, even for the case of discontinuous $\beta$, but is limited to first-order accuracy. Other approaches based on fictitious points have been employed to achieve higher order accuracy, though typically at the cost of ease of implementation or symmetry of the stencil. For example, the Matched Interface and Boundary (MIB) method \cite{zhou:2006} determines fictitious values by matching one-sided discretizations of the jump conditions with high-order extrapolations of the solution. The MIB stencil is determined by local geometry, which results in a non-symmetric linear problem. In \cite{zhou:2006:2}, the MIB is extended to handle interfaces with high curvature and in \cite{yu:2007}, the MIB is adapted to 3D. The MIB has also been used to solve the Navier-Stokes equations with singular forces \cite{zhou:2012}. Another approach, the Coupling Interface Method (CIM) \cite{chern:2007}, uses a second-order extrapolation everywhere but at exceptional points, where a first-order approximation is used instead. Due to the use of one-sided finite difference stencils, the CIM likewise leads to a non-symmetric linear problem. See \cite{shu:2014} for recent development of the CIM. More recently, second-order accuracy with a symmetric linear system in the general case has been achieved in \cite{bedrossian:2010} with the use of a variational method to define the stencil combined with a Lagrange multiplier approach to enforce the jump conditions. These techniques have recently been extended to 3D in \cite{hellrung:2012} and applied to Stokes flow in \cite{assencio:2013}. Higher-order accuracy on Poisson problems has also been recently obtained for a correction function method similar to the GFM \cite{marques:2011}.

There are also a number of finite element method (FEM) approaches to solving \eqref{eqn:jump-elliptic}; see, for example, the extended finite element method (XFEM) \cite{belytschko:2001, moes:1999, ji:2004, vaughan:2006}. The XFEM adds additional discontinuous basis elements to the standard finite element basis, along with additional degrees of freedom, in order to capture the discontinuous structure of the solution. Recently, a high-order XFEM method using a discontinuous-Galerkin approach has been developed \cite{brandstetter:2015}. XFEM has also been used to solve the Navier-Stokes equations with surface tension \cite{gross:2007}. Other FEM methods that introduce additional degrees of freedom include \cite{dolbow:2009, hansbo:2002, hansbo:2004, massjung:2012}. In these methods, as with XFEM, the solution spaces do not typically allow the interface conditions to be exactly satisfied, so linear constraints are added in the form of Lagrange multipliers or penalty terms, either of which can incur significant computational cost. In contrast, other FEM approaches \cite{li:1998:2, li:2003, ji:2014, lin:2015, hou:2005} alter the basis functions to satisfy the interface constraints directly. Similarly, the Exact Subgrid Interface Correction Scheme (ESIC) \cite{huh:2008} and Simplified Exact Subgrid Interface Correction Scheme (SESIC) \cite{discacciati:2013} methods integrate the jump conditions into the formulation of the basis functions and provide a fast and simple approach, with a symmetric linear system, when $\jump{\beta} = 0$. FEM methods in general enjoy symmetric positive definitive discretizations, except with Lagrange multipliers wherein the discretization may be symmetric indefinite, but often suffer poorer conditioning, particularly when stabilization is used.

Finite volume methods for \eqref{eqn:jump-elliptic} have also been developed. For example, Oevermann and Klein \cite{oevermann:2006, oevermann:2009} present a second-order finite volume method for elliptic interface problems by solving local constraint equations, though still arrive at a non-symmetric system in the general case.

\section{The Jump Splice}
\label{sec:splice}

In this section, we develop a mathematically rigorous methodology for evaluating arbitrary finite difference stencils in the presence of known discontinuities specified across an interface. The result is a highly general framework for evaluating derivatives and solving differential equations with known jump conditions. We proceed as follows.
\begin{itemize}
\setlength\itemsep{0.2em}
\item We begin by motivating the theoretical considerations that lead to the jump splice in Section \ref{sec:splice:motivation}.
\item In Section \ref{sec:splice:splice}, we define the jump splice for arbitrary linear finite difference operators and prove Proposition \ref{prop:jump-splice}, the key result underlying our technique. We define the jump extrapolation, but we do not yet construct it.
\item Next, we show an intuitive approach, though not what we use in practice, to calculating the jump extrapolation in Section \ref{sec:splice:extrapolation}.
\item In Section \ref{sec:splice:accuracy}, we put precise limits on how accurately the jump extrapolation needs to be computed for the guarantees of Proposition \ref{prop:jump-splice} to hold.
\item We then describe a straightforward bootstrapping procedure for constructing the jump extrapolation in practice in Section \ref{sec:splice:calculation}.
\item Finally, in Section \ref{sec:splice:implementation}, which is essentially self-contained, we lay out the full algorithm for implementing the jump splice.
\item We briefly show numerical results in Section \ref{sec:splice:results}. We will present a more comprehensive convergence analysis in Section \ref{sec:elliptic}.
\end{itemize}

\subsection{Notation}
\label{sec:splice:notation}

In what follows, we will write $\phi : \Omega \to \bbr$ for the signed distance function corresponding to the interface $\Gamma$. We use the convention that $\phi > 0$ in the interior of the region bounded by $\Gamma$ and take $\bfn = \nabla \phi$ as the inward-pointing unit normal. We also write $\Omega^+$ and $\Omega^-$ for the interior and exterior of the region bounded by $\Gamma$, respectively. See \cite{osher:1988, sethian:1999, osher:2002} for detailed discussion of signed distance and level set functions and their development.

For a function $u : \Omega \to \bbr$, we define the surface gradient as
\begin{equation}
\nabla_s u = \nabla u - (\partial_\bfn u) \bfn,
\end{equation}
where $\partial_\bfn u = \nabla u \cdot \bfn$ is the normal derivative. We also define the surface Laplacian as
\begin{equation}
\Delta_s u = \nabla_s \cdot (\nabla_s u),
\end{equation}
where 
\begin{equation}
\nabla_s \cdot \bfu = \nabla \cdot \bfu - \bfn \cdot \nabla \bfu \cdot \bfn
\end{equation}
is the surface divergence for $\bfu : \Omega \to \bbr^m$. Here and throughout the paper, we interpret $\nabla \bfu$ as the matrix with $(i,j)$ entry equal to the $j$-th derivative of the $i$-th component of $\bfu$. Note that $\nabla_s u$, $\nabla_s \cdot \bfu$, and $\Delta_s u$ are defined not just on $\Gamma$, but in fact everywhere that $\bfn$ is defined. If $g : \Omega \to \bbr$ has the property that $\subbar{g}{\Gamma} = \jump{u}$, then
\begin{equation}
\label{eqn:nablas-jump}
\subbar{\nabla_s g}{\Gamma} = \jump{\nabla_s u},  
\end{equation}
and
\begin{equation}
\label{eqn:deltas-jump}
\subbar{\Delta_s g}{\Gamma} = \jump{\Delta_s u}.  
\end{equation}
Here \eqref{eqn:nablas-jump} follows by locally parametrizing the interface and taking tangential derivatives and \eqref{eqn:deltas-jump} follows as $\Delta_s u = \Tr(\nabla_s\nabla_s u)$. We will often abuse notation slightly and write $\nabla_s \jump{u} = \jump{\nabla_s u}$ and $\Delta_s \jump{u} = \jump{\Delta_s u}$. These definitions can be extended component-wise to $\bfu,\ \bfg : \Omega \to \bbr^m$. 

We write $C^k(U)$ for the space of functions on an open set $U \subset \bbr^d$ with continuous derivatives up to order $k$ and $LC^k(U)$ for the space of functions on $U$ with Lipschitz continuous derivatives up to order $k$. Recall that a function $\bfu : U \to \bbr^m$ is Lipschitz if there exists a constant $K$ such that
\[ |\bfu(\bfx) - \bfu(\bfy)| \leq K |\bfx - \bfy| \quad \text{for all } \bfx,\bfy \in U,\]
where $|\cdot|$ denotes the Euclidean norm. We will also write $LC^k(U_1, U_2)$ for the space of functions $\bfu$ with domain $U_1 \cup U_2$ such that $\subbar{\bfu}{U_1} \in LC^k(U_1)$ and $\subbar{\bfu}{U_2} \in LC^k(U_2)$. Note that $LC^k(U_1, U_2)$ is not in general the same as $LC^k(U_1 \cup U_2)$ due to the non-locality of the Lipschitz property. 

Finally, we define $C^k(\Gamma)$ to be the space of functions defined on the interface $\Gamma$ that can be extended to a function in $C^k(U)$ for some open set $U$ containing $\Gamma$. We define $LC^k(\Gamma)$ analogously.\footnote{Note that our definitions of $C^k(\Gamma)$ and $LC^k(\Gamma)$ here do not require $\Gamma$ to be a $C^k$ submanifold.}

\subsection{Motivation}
\label{sec:splice:motivation}

For notational simplicity, we will often assume that $\Omega \subset \bbr^2$ and that all Cartesian grids have uniform spacing. However, jump splice techniques extend naturally to $\bbr^3$ and to non-uniform grid spacing with only a change in finite difference operator.  

Let $u_{i,j} = u(\bfx_{i,j})$ with $\bfx_{i,j} = (ih, jh)$ be the values of a function $u$ defined on a Cartesian grid with uniform spacing $h$. The standard 5-point discretization of the Laplacian is then defined by 
\[(\Delta^h u)_{i,j} = \frac{1}{h^2}\left(u_{i+1,j} + u_{i-1,j} + u_{i,j+1} + u_{i,j-1} - 4u_{i,j}\right).\]
It is not difficult to show (see Proposition \ref{prop:laplacian-lipschitz} in the appendix) that
\begin{equation}
\label{eqn:laplacian-error}
(\Delta u)(\bfx_{i,j}) = (\Delta^h u)_{i,j} + \bigo{h^2},
\end{equation}
provided $u \in LC^3(U)$ for some open set $U$ containing the stencil cross
\[C_{i,j} = \{\lambda_1 \bfx_{i-1,j} + (1-\lambda_1) \bfx_{i+1,j} : \lambda_1 \in [0,1]\} \cup \{\lambda_2 \bfx_{i,j-1} + (1-\lambda_2) \bfx_{i,j+1} : \lambda_2 \in [0,1]\}. \]

Now suppose $u \in LC^3(\Omega^+, \Omega^-)$. Hence $u$ and its derivatives may not be continuous across $\Gamma$. At points $\bfx_{i,j}$ sufficiently close to the interface, the set $C_{i,j}$ will intersect $\Gamma$. Since $u$ may not be continuous at the point of intersection, the error estimate \eqref{eqn:laplacian-error} may fail. At these points $\bfx_{i,j}$, we are not able to accurately approximate $(\Delta u)(\bfx_{i,j})$ with a standard finite difference stencil. 

In fact, any fixed finite difference stencil will fail to achieve its expected order of accuracy in the presence of an interface discontinuity. In the next section, we will show that if we are provided with explicit jump information pertaining to $u$, we can ``splice'' away the discontinuity and accurately evaluate any linear finite difference operator.

\subsection{The Splice}
\label{sec:splice:splice}

We now define the jump splice. Consider a linear differential operator $D$ and a finite difference discretization $D^h_{p,q}$ with the property that
\begin{equation}
\label{eqn:d-error}
(Du)(\bfx_{i,j}) = (D^h_{p,q}u)_{i,j} + \bigo{h^p},
\end{equation}
provided $u \in LC^q(U)$ on some convex open set $U$ containing the stencil of $(D^h_{p,q}u)_{i,j}$. Here $q$ is the required smoothness, in the sense of $LC^q$, to obtain order $p$ accuracy.  Examples include the standard 5-point Laplacian (with $D = \Delta$, $p = 2$, and $q = 3$) and standard 4-point centered differences for calculating the gradient (with $D = \nabla$, $p = 2$, and $q = 2$). 

Now suppose $u \in LC^q(\Omega^+ , \Omega^-)$, and that we are given
\begin{equation}
\label{eqn:jump-conditions}
\left\{ 
\begin{aligned}
\jump{u} & = g^0\\
\jump{\partial_\bfn u} & = g^1 \\
\jump{\partial_\bfn^2 u} & = g^2 \\
\vdots & \\
\jump{\partial_\bfn^q u} & = g^q,
\end{aligned}
\right.
\end{equation}
where $g^i \in LC^{q-i}(\Gamma)$ for $0 \leq i \leq q$.\footnote{Recall that $LC^{q-i}(\Gamma)$ is the space of functions that admit an $LC^{q-i}(U)$ extension to an open set $U$ containing the interface.} Away from the interface, $D^h_{p,q}$ can be evaluated accurately with no additional work, but near $\Gamma$, we need to use the jump conditions \eqref{eqn:jump-conditions} to correct for the lack of smoothness in $u$ and thus to recover the error estimate \eqref{eqn:d-error}. Let
\[ \Gamma_\epsilon = \{ \bfx \in \Omega : |\phi(\bfx)| < \epsilon \}, \]
be the band of width $\epsilon = \bigo{h}$ around $\Gamma$, where $\epsilon$ is chosen so that the stencil of $D^h_{p,q}$ evaluated in $\Omega \setminus \Gamma_\epsilon$ does not cross the interface. In the remainder of this section, we motivate and prove the following key result.

\begin{proposition}[Splice Discretization]
\label{prop:jump-splice}
If $v \in LC^q(\Gamma_\epsilon)$ satisfies
\begin{equation}
\label{eqn:v-conditions}
\left\{ 
\begin{aligned}
\subbar{v}{\Gamma} & = g^0\\
\subbar{\partial_\bfn v}{\Gamma} & = g^1 \\
\subbar{\partial_\bfn^2 v}{\Gamma} & = g^2 \\
\vdots & \\
\subbar{\partial_\bfn^q v}{\Gamma} & = g^q,
\end{aligned}
\right.
\end{equation}
then we can discretize $Du$ as 
\begin{equation}
\label{eqn:discretization}
Du = D^h_{p,q}u - D^h_{p,q}(v H(\phi)) + (D^h_{p,q}v) H(\phi) + \bigo{h^p},
\end{equation}
to obtain a $p$th order accurate approximation in all of $\Omega$.
\end{proposition}

In the above proposition, $H$ is the standard Heaviside function
\[ H(z) =
\begin{cases}
1 & \text{if } z \geq 0\\
0 & \text{if } z < 0,
\end{cases}
\]
and we take as a convention that $\subbar{u}{\Gamma} = u^+$, and likewise for derivatives of $u$, recalling that $u^\pm(\bfx) = \lim_{h\to 0^+} u(\bfx \pm h\bfn)$ for $\bfx \in \Gamma$. In practice, the definitions of $H$ at $z = 0$ and $u$ on $\Gamma$ are immaterial provided that they agree in the sense that $u = u^- + (u^+ - u^-)H(\phi)$. 

We will often refer to $v$ in Proposition \ref{prop:jump-splice} as the \emph{jump extrapolation}. It is important to note that \eqref{eqn:discretization} reduces to \eqref{eqn:d-error} whenever the stencil of $D^h_{p,q}$ does not cross the interface; it is for this reason that \eqref{eqn:discretization} holds in all of $\Omega$, even though $v$ is only defined in a band around $\Gamma$.

To motivate Proposition \ref{prop:jump-splice}, suppose we wish to approximate $(Du)(\bfx_{i,j})$ for some $\bfx_{i,j} \in \Omega^-$ sufficiently close to the interface that the stencil of $(D_{p,q}^hu)_{i,j}$ crosses $\Gamma$ and thus \eqref{eqn:d-error} fails to hold. To recover a $p$th order accurate approximation, we will use the jump conditions \eqref{eqn:jump-conditions} to adjust, or ``splice'', the values of $u$ on the other side of the interface in such a way that \eqref{eqn:d-error} holds for the adjusted $u$.

Define the \emph{outer splice} of $u$ as
\begin{equation}
\label{eqn:outer-splice}
 w^-(\bfx) = u(\bfx) - v(\bfx) H(\phi(\bfx)),
\end{equation}
for $\bfx \in \Gamma_\epsilon$. Here $v \in LC^q(\Gamma_\epsilon)$ is to be determined. Note that $w^- = u$ in $\Omega^-$, and therefore 
\begin{equation}
\label{eqn:outer-splice-equality}
(D w^-)(\bfx) = (D u)(\bfx) \qquad \text{for } \bfx \in \Omega^- \cap \Gamma_\epsilon.
\end{equation}
If we can can choose $v$ in such a way that $w^- \in LC^q(\Gamma_\epsilon)$, then \eqref{eqn:outer-splice-equality} combined with the error estimate \eqref{eqn:d-error} applied to $w^-$ show that
\begin{equation}
\label{eqn:outer-error}
(Du)(\bfx_{i,j}) = (D^h_{p,q} w^-)_{i,j} + \bigo{h^p} \qquad \text{for } \bfx_{i,j} \in \Omega^- \cap \Gamma_\epsilon.
\end{equation}
In essence the term $-v H(\phi)$ in \eqref{eqn:outer-splice} is ``subtracting off the jumps'' in $u$ and thereby allowing us to accurately use the finite difference stencil $D^h_{p,q}$ on $w^-$. For $\bfx_{i,j} \in \Omega^+$, we can analagously define the \emph{inner splice}
\begin{equation}
\label{eqn:inner-splice}
 w^+(\bfx) = u(\bfx) + v(\bfx) (1 - H(\phi(\bfx)))
\end{equation}
for $\bfx \in \Gamma_\epsilon$, where $v$ is the same as in \eqref{eqn:outer-splice}. Here we have $w^+ = u$ in $\Omega^+$, and a similar argument shows that 
\begin{equation}
\label{eqn:inner-error}
(Du)(\bfx_{i,j}) = (D^h_{p,q} w^+)_{i,j} + \bigo{h^p} \qquad \text{for } \bfx_{i,j} \in \Omega^+ \cap \Gamma_\epsilon.
\end{equation}
By appealing to the definitions of $w^\pm$, we can combine \eqref{eqn:outer-error} and \eqref{eqn:inner-error} to establish the main result \eqref{eqn:discretization} of Proposition \ref{prop:jump-splice} in $\Omega \setminus \Gamma$. To see that \eqref{eqn:discretization} also holds for $\bfx_{i,j} \in \Gamma$, recall that $H(0) = 1$, and thus $w^+ = u$ on $\Gamma$. It follows that we can invoke the inner splice \eqref{eqn:inner-error} to approximate $(Du)(\bfx_{i,j})$, and this agrees with \eqref{eqn:discretization}. In fact, \eqref{eqn:discretization} holds for any consistent choice of $\subbar{u}{\Gamma}$ and $H(0)$.

We have thus far assumed that we can find a suitable $v \in LC^q(\Gamma_\epsilon)$ so that $w^\pm \in LC^q(\Gamma_\epsilon)$. The key to constructing such a $v$ lies in the following proposition, which we prove in the appendix. 

\begin{proposition}
\label{prop:jump-lipschitz-iterated}
If $w \in LC^k(\Omega^+, \Omega^-)$ and $\jump{\partial^i_\bfn w} = 0$ for $0 \leq i \leq k$, then there exists a unique $\tilde w \in LC^k(\Omega)$ that extends $w$ in the sense that $\subbar{\tilde w}{\Omega^+\cup\Omega^-} = w$.
\end{proposition}

Thus to ensure that $w^\pm \in LC^q(\Gamma_\epsilon)$, we need to choose $v$ such that $\jump{\partial_\bfn^i w^\pm} = 0$ for $0 \leq i \leq q$. To obtain $\jump{w^-} = 0$, we need
\[\begin{aligned}
0 &= \jump{w^-} \\
  &= \jump{u} - \jump{v H(\phi)} \\
  &= g^0 - \subbar{v}{\Gamma}
\end{aligned}, \]
so that $\subbar{v}{\Gamma} = g^0$. This is also the constraint required to obtain $\jump{w^+} = 0$, confirming our choice of using the same $v$ in the definitions of $w^+$ and $w^-$. Similar calculations show that provided $v$ satisfies \eqref{eqn:v-conditions}, we will have $\jump{\partial_\bfn^i w^\pm} = 0$ for $0 \leq i \leq q$, as needed. 

\subsection{The Jump Extrapolation}
\label{sec:splice:extrapolation}

In the previous section, we derived the necessary conditions \eqref{eqn:v-conditions} that the jump extrapolation $v$ must satisfy for Proposition \ref{prop:jump-splice} to hold, but we have not yet explicitly constructed $v$. We now show a particularly intuitive approach to building the jump extrapolation; in Section \ref{sec:splice:calculation}, we will discuss the bootstrapping approach we use in practice.

We assume from this point forward that $\phi \in LC^{q+1}(\Gamma_\epsilon)$. This requires both that $\Gamma$ be $C^{q+1}$ (see \cite{foote:1984}) and that $\epsilon$ be sufficiently small.\footnote{In 2D, we need $\epsilon < \sup_\Gamma |\kappa|^{-1}$, where $\kappa = \nabla \cdot \bfn$ is the curvature. In 3D, we need $\epsilon < \sup_\Gamma |\kappa_{\rm max}|^{-1}$, where $\kappa_{\rm max}$ is the largest eigenvalue in absolute value of $\nabla \bfn$.} In practice, these restrictions do not pose a problem. Because $\epsilon = \bigo{h}$, refinement of the grid will ensure that $\epsilon$ is sufficiently small. Moreover, numerical experiments in Section \ref{sec:elliptic:results} show that jump splice techniques still achieve their expected order of accuracy with interfaces that are only $C^1$. We will also assume in this section that $g^i \in LC^q(\Gamma)$ for $0 \leq i \leq q$.

Because $v$ need only be defined on $\Gamma_\epsilon$, and thus for $\phi$ close to zero, it is natural to construct $v$ as a truncated Taylor series in $\phi$ using the known jump behavior of $u$. To wit, define
\begin{equation}
\label{eqn:v-definition}
v = \bar g^0 + \bar g^1 \phi + \frac{1}{2}\bar g^2 \phi^2 + \cdots + \frac{1}{q!} \bar g^q \phi^q,
\end{equation} 
where 
\[\bar g^i(\bfx) = g^i(\bfx - \phi(\bfx)\bfn(\bfx))\] 
is the constant normal extension\footnote{The closest point to $\bfx$ on the interface $\Gamma$ is $\operatorname{cp}(\bfx) = \bfx - \phi(\bfx)\bfn(\bfx)$.} of $g^i$ into $\Gamma_\epsilon$. Because $\phi \in LC^{q+1}(\Gamma_\epsilon)$, and thus $\bfn \in LC^q(\Gamma_\epsilon)$, it follows that $\bar g^i \in LC^q(\Gamma_\epsilon)$, and therefore also that $v \in LC^q(\Gamma_\epsilon)$. Moreover, because $\partial_\bfn^k \bar g^i = 0$ for $1 \leq k \leq q$ and because $\subbar{\bar g^i}{\Gamma} = g^i$, it follows that $v$ defined in \eqref{eqn:v-definition} satisfies the necessary conditions \eqref{eqn:v-conditions} from Proposition \ref{prop:jump-splice}. We will often refer to this expression as the \emph{canonical jump extrapolation}.

\begin{figure}[t]
\centering
\begin{tabular}{c@{\hskip 0.0in}c@{\hskip 0.0in}c@{\hskip 0.0in}c}
\includegraphics[width=0.25\linewidth]{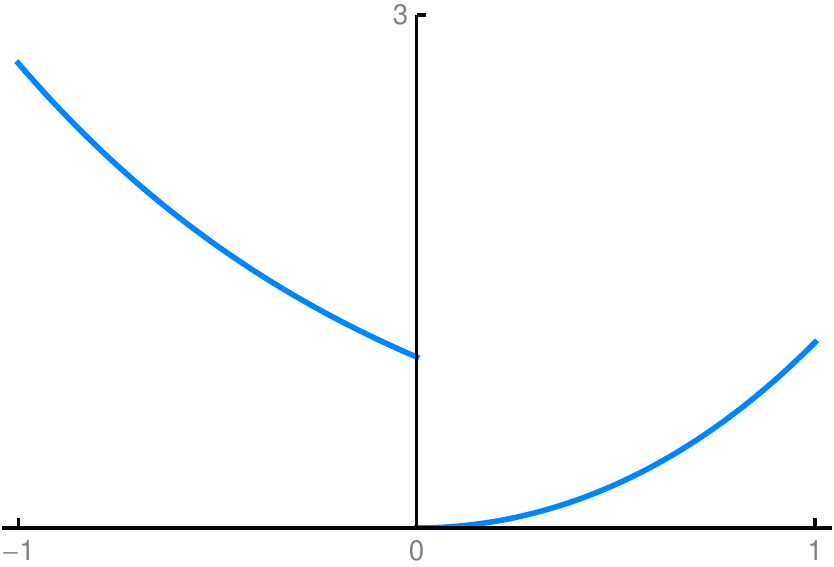} 
&
\includegraphics[width=0.25\linewidth]{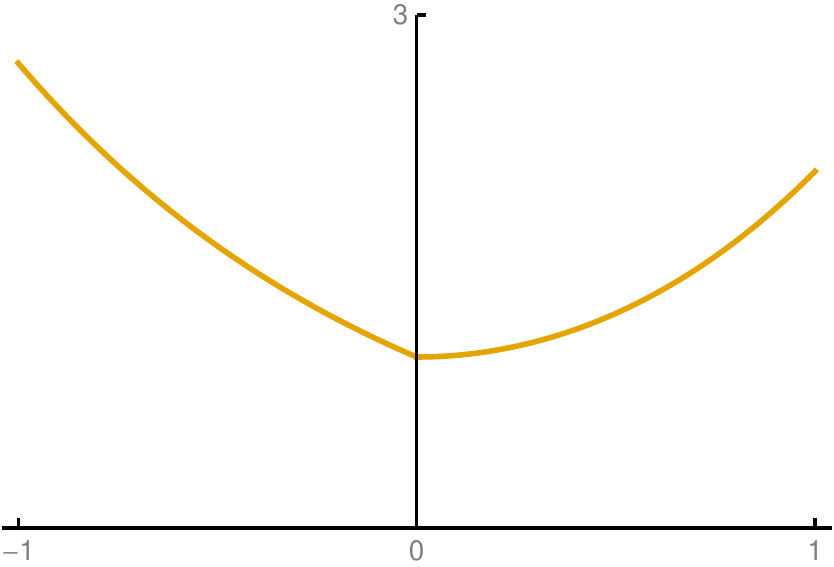} 
&
\includegraphics[width=0.25\linewidth]{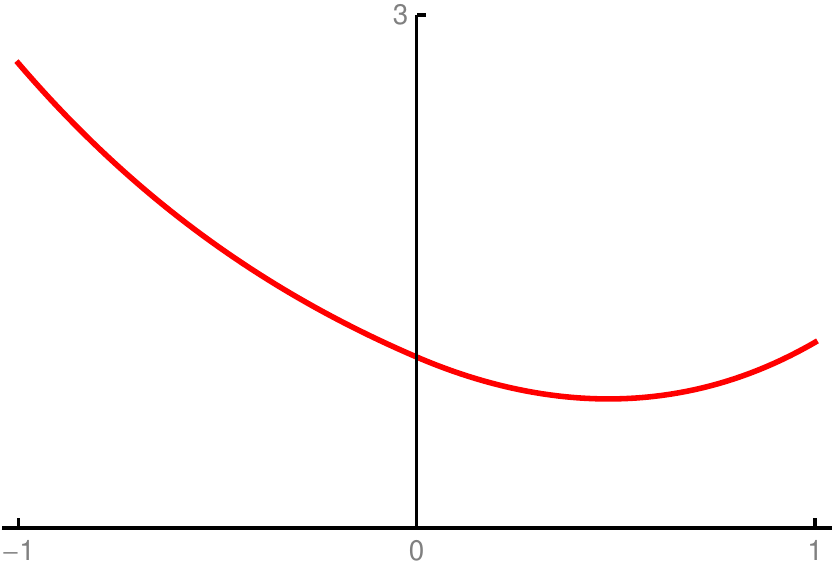} 
&
\includegraphics[width=0.25\linewidth]{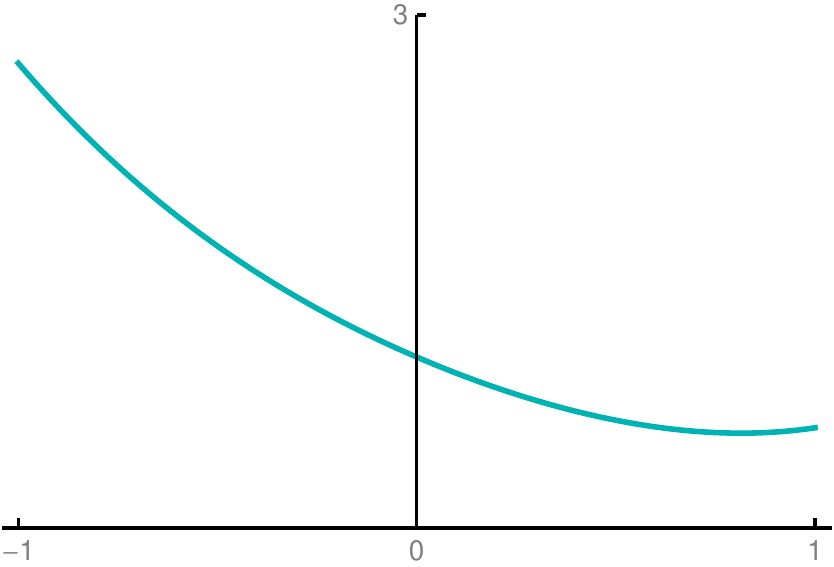} 
\\
$u$ & $u - v^0 H(\phi)$ & $u - v^1 H(\phi)$ & $u - v^2 H(\phi)$\\
\end{tabular}
\caption{Visual demonstration of the jump splice in one dimension. Here $\Gamma = \{0\}$, $\phi = x$, and ${u = e^{-x} + (e^x - 2)H(\phi)}$, with $\jump{u} = -1$ and $\jump{\partial_\bfn^k u} = 1$ for $k \geq 1$. The $v^k$ are constructed from \eqref{eqn:v-definition} by including the first $k$ terms. According to Proposition \ref{prop:jump-splice}, if our goal is to numerically approximate the derivative of $u$ at $x < 0$ close to $\Gamma$, then we should evaluate $\partial^h_x(u - v^2 H(\phi))$.}
\label{fig:js}
\end{figure}

Each term in \eqref{eqn:v-definition} corrects for a corresponding discontinuity in $u$ from \eqref{eqn:jump-conditions} and thereby illustrates how the jump conditions in $u$ give rise to the jump extrapolation $v$. Figure \ref{fig:js} provides a visual example in 1D. In Section \ref{sec:splice:calculation}, we will show that, in most settings, there are more convenient means of constructing the jump extrapolation than \eqref{eqn:v-definition}. The canonical jump extrapolation remains valuable because any other jump extrapolation satisfying the conditions \eqref{eqn:v-conditions} differs by at most $\bigo{\phi^{q+1}}$. This is made precise by the following proposition, which we prove in the appendix.

\begin{proposition}
\label{prop:v-equivalence}
Let $v, \tilde v \in LC^q(\Gamma_\epsilon)$ satisfy the conditions \eqref{eqn:v-conditions}, that is, $\subbar{\partial_\bfn^i v}{\Gamma} = g^i$ for $0\leq i \leq q$. Then $v = \tilde v + \bigo{\phi^{q+1}}$.
\end{proposition}

Recall that, in $\Gamma_\epsilon$, $\bigo{\phi^{q+1}} = \bigo{h^{q+1}}$. It follows that any result that holds for the canonical jump extrapolation will hold, up to $\bigo{h^{q+1}}$ for all other jump extrapolations as well.

\subsection{Accuracy Considerations}
\label{sec:splice:accuracy}

The discussion up until now has assumed that the functions $\phi$ and $g^i$ for $0\leq i\leq q$ are known precisely and that $v$ is exactly computed as described in the previous section. In practice, there will be discretization error in all of these quantities, and the formulation of the jump splice puts limits on the maximum error such that Proposition \ref{prop:jump-splice} will still hold. This is made precise by the following result.

\begin{proposition}
\label{prop:accuracy}
If $v$ satisfies the conditions \eqref{eqn:v-conditions} and $\hat v = v + \bigo{h^{q+1}}$ is an approximation of $v$, then the main error estimate \eqref{eqn:discretization} still holds with $v$ replaced by $\hat v$.
\end{proposition}

To see why this is true, note that if $D_{p,q}^h$ is the finite difference discretization of a linear differential operator $D$ that contains highest derivatives of order $r$, then the relation $q = p + r - 1$ will hold by a Taylor series argument. Here $p$ and $q$ are as described in Section \ref{sec:splice:splice}. We can also write
\[ D^h_{p,q} \hat v = D^h_{p,q} v + \bigo{h^{q-r+1}}, \]
as a finite difference stencil approximating a differential operator with highest derivatives of order $r$ will involve division by $h^r$. Since $q - r + 1 = p$, the result follows. Note that the smoothness of the error term in $\hat v$ is immaterial, since we evaluate $D^h_{p,q}(v H(\phi))$ in \eqref{eqn:discretization}, which is always discontinuous at the interface.

Proposition \ref{prop:accuracy} imposes straightforward criteria on the accuracy of all other quantities. In particular, by appealing to the definition of $v$ in \eqref{eqn:v-definition}, it is clear that if $\hat \phi$ is an approximation of the signed distance function, we need 
\begin{equation}
\label{eqn:accuracy-phi}
\hat \phi = \phi + \bigo{h^{q+1}},
\end{equation}
From the same equation, we can see that if $\hat g^i$ is an approximation of $g^i$, then we need
\begin{equation}
\label{eqn:accuracy-g}
\hat g^i = g^i + \bigo{h^{q-i+1}},
\end{equation}
since $\overline{\hat g}^i \phi^i = \bar g^i \phi^i + \bigo{h^{q+1}}$. 

Thus provided that an approximation $\hat v$ of $v$ is constructed in such a way that \eqref{eqn:accuracy-phi} and \eqref{eqn:accuracy-g} are satisfied, the key error estimate \eqref{eqn:discretization} in Proposition \ref{prop:jump-splice} will still hold.

\subsection{Practical Calculation}
\label{sec:splice:calculation}

The construction of $v$ defined by \eqref{eqn:v-definition} is very important for intuition, but can be quite cumbersome in practice. Indeed, in most applications with a Cartesian grid and an implicitly defined interface, the jump conditions \eqref{eqn:jump-conditions} are not specified directly on $\Gamma$, but rather the $g^i$ are defined in all of $\Gamma_\epsilon$ such that they specify the right behavior on the interface, that is $\subbar{g^i}{\Gamma} = \jump{\partial_\bfn^i u}$. Moreover, in many applications, including those discussed in the rest of this paper, it is far more convenient to work with $\jump{\Delta u}$ and $\jump{\partial_\bfn\Delta u}$ than with $\jump{\partial_\bfn^2 u}$ and $\jump{\partial_\bfn^3 u}$. We now describe an approach to building $v$ that takes these considerations as a starting point, and that is significantly eaiser to implement in practice.

For the remainder of this section, we will restrict to the case that $q \leq 3$ to ease notation, but all results can be extended to arbitrary $q$. This is not too restrictive, as $q \leq 3$ is sufficient to achieve up to second-order accuracy with up to second-order differential operators. In particular, we will now assume that we have $g^0, g^1, g^\Delta, g^{\partial_\bfn\Delta} \in LC^q(\Gamma_\epsilon)$\footnote{Technically, we only need $g^i \in LC^q(\Gamma_\epsilon \setminus \Gamma)$ along with $g^i \in LC^{q-i}(\Gamma)$, in agreement with the original smoothness required for the $g^i$.} such that
\begin{equation}
\label{eqn:modified-jump-conditions}
\left\{ 
\begin{aligned}
\subbar{g^0}{\Gamma} &= \jump{u} \\
\subbar{g^1}{\Gamma} &= \jump{\partial_\bfn u} \\
\subbar{g^\Delta}{\Gamma} &= \jump{\Delta u} \\
 \subbar{g^{\partial_\bfn\Delta}}{\Gamma} &= \jump{\partial_\bfn \Delta u},
\end{aligned}
\right.
\end{equation}
where we use only the first $q+1$ of these conditions for $q < 3$. The key to constructing $v$ given \eqref{eqn:modified-jump-conditions} is the following proposition, which we prove in the appendix. 

\begin{proposition}
\label{prop:modified-v-conditions}
If $v \in LC^q(\Gamma_\epsilon)$, for $q \leq 3$, satisfies the first $q+1$ conditions
\begin{equation}
\label{eqn:modified-v-conditions}
\left\{ 
\begin{aligned}
\subbar{v}{\Gamma} & = g^0\\
\subbar{\partial_\bfn v}{\Gamma} & = g^1 \\
\subbar{\Delta v}{\Gamma} & = g^\Delta \\
\subbar{\partial_\bfn \Delta v}{\Gamma} & = g^{\partial_\bfn\Delta},
\end{aligned}
\right.
\end{equation}
then $v$ also satisfies \eqref{eqn:v-conditions}, that is, $\subbar{\partial_\bfn^i v}{\Gamma} = g^i$ for $0 \leq i \leq q$.
\end{proposition}

Thus if we can construct $v$ to satisfy \eqref{eqn:modified-v-conditions}, then $v$ will also satisfy the original conditions \eqref{eqn:v-conditions} necessary for Proposition \ref{prop:jump-splice} to hold. We do this by building up $v$ through a simple, and easy to implement, recursive relationship. 

We begin by defining $v^0 = a^0$ where $a^0 = g^0$, recalling that $g^0$ is now defined throughout $\Gamma_\epsilon$, and then write
\begin{equation}
\label{eqn:v-recursion}
v^k = v^{k-1} + \frac{1}{k!} a^k \phi^k,  
\end{equation}
for $1 \leq i \leq 3$, where we have
\begin{equation}
\label{eqn:a-definition}
\begin{aligned}
a^1 &= g^1 - \partial_\bfn v^0 \\
a^2 &= g^\Delta - (\Delta v^1 -  (\Delta a^1)\phi)\\
a^3 &= g^{\partial_\bfn\Delta} - \partial_\bfn\Delta v^2.
\end{aligned}
\end{equation}
Here the $a^i$ are derived by successively enforcing the constraints in \eqref{eqn:modified-v-conditions} and discarding $\bigo{\phi}$ terms. For example, to derive $a^1$, we apply $\partial_\bfn$ to both sides of \eqref{eqn:v-recursion} for $k = 1$, discard the term $(\partial_\bfn a^1) \phi$ and solve for $a^1$, obtaining the first equation in \eqref{eqn:a-definition}. We repeat this for $k = 2$ and $k = 3$ by applying $\Delta$ and $\partial_\bfn \Delta$, respectively. 

For $k = 2$, this process yields $a^2 = g^\Delta - \Delta v^1$, but we make the modification indicated in \eqref{eqn:a-definition}. This follows from expanding
\[\Delta v^1 = \Delta g^0 + \kappa (g^1 - \partial_\bfn g^0) + 2(\partial_\bfn g^1 - \partial_\bfn^2 g^0) + \Delta(g^1 - \partial_\bfn g^0)\phi,\]
and observing that the last term $\Delta(g^1 - \partial_\bfn g^0) \phi = (\Delta a^1) \phi$ does not contribute toward satisfying the condition $\subbar{\Delta v^2}{\Gamma} = g^\Delta$, and thus can be removed. This change is equivalent in terms of convergence behavior, but by reducing the composition of finite difference operators in the construction of $a^2$, we achieve significantly improved numerical results. A similar procedure can be employed on the term $\partial_\bfn\Delta v^2$ in $a^3$, but without a similar improvement in numerical error for $q \leq 3$.

Error analysis for this construction of $v$ is somewhat more subtle, because the $g^i$ are now arbitrary in $\Gamma \setminus \Gamma_\epsilon$. Provided we construct $\hat v$ in accordance with \eqref{eqn:v-recursion} from an approximation $\hat g^i$ such that
\begin{equation}
\label{eqn:modified-accuracy-g}
\subbar{\hat g^i}{\Gamma} = g^i + \bigo{h^{q-i+1}},
\end{equation}
along with the same constraint \eqref{eqn:accuracy-phi} as before on $\hat \phi$, the main error estimate \eqref{eqn:discretization} will still hold. (Here $i = 2$ and $i = 3$ correspond to $g^\Delta$ and $g^{\partial_\bfn \Delta}$.) To see this, we can write $\hat g^i = \bar g^i + \nu^i \phi + \bigo{h^{q-i+1}}$ for some $\nu^i \in LC^{q}(\Gamma_\epsilon)$ and follow the construction in \eqref{eqn:v-recursion} and \eqref{eqn:a-definition}, winding up with $\hat v^i = v^i + \bigo{h^{q+1}} + \bigo{\phi^{i+1}}$. Since $\phi = \bigo{h}$ in $\Gamma_\epsilon$, Proposition \ref{prop:accuracy} shows \eqref{eqn:discretization} holds. 

A remarkable consequence of constructing $v$ as described above is that we arrive at a valid jump extrapolation even if $\phi$ is not a signed distance function. In fact, provided $\phi$ is a reasonably smooth function with zero level set $\Gamma$, and provided $|\nabla \phi|$ is both bounded from above and bounded away from zero in $\Gamma_\epsilon$, the procedure in \eqref{eqn:v-recursion} and \eqref{eqn:a-definition} will construct a $v$ that satisfies the preconditions for Proposition \ref{prop:jump-splice}. However, when $\phi$ differs significantly from a signed distance function, numerical error increases substantially. As a result, in this paper we will always reconstruct level set functions into corresponding signed distance functions.

\subsection{Implementation}
\label{sec:splice:implementation}

In this section, we only consider finite difference operators $D_{p,q}^h$ with $q \leq 3$, though extension to arbitrary $q$ is straightforward. We will assume that $\Omega$ is a rectangular domain with a regular $n \times n$ Cartesian grid with $n = 1/h$, but as noted before, extension to 3D is as simple as changing the finite difference operator.

In what follows, $\Delta^h$ is the standard 5-point Laplacian, $\Delta^h_4$ is the 9-point, fourth-order Laplacian, defined by
\begin{equation}
\label{eqn:laplacian-4}
(\Delta^h_4 u)_{i,j} = \frac{1}{12h^2}\left(-u_{i+2,j} + 16 u_{i+1,j} + 16 u_{i-1,j} - u_{i-2,j} - u_{i,j+2} + 16 u_{i,j+1} + 16 u_{i,j-1} - u_{i,j-2} - 60 u_{i,j}\right),
\end{equation} 
$\nabla^h$ is the 4-point, second-order centered difference gradient, and $\nabla^h_4$ is the 8-point, fourth-order centered difference gradient.

Assume that we are given a discrete approximation of the signed distance function, $\hat \phi_{i,j}$, as well as discrete approximations of the first $(q+1)$ of the jump conditions, $\hat g^0_{i,j}, \hat g^1_{i,j}, \hat g^\Delta_{i,j}$, and $\hat g^{\partial_\bfn\Delta}_{i,j}$, all defined in a band around the interface, as developed in \cite{adalsteinsson:1995}. We further assume that these quantities satisfy the accuracy requirements given in \eqref{eqn:accuracy-phi} and \eqref{eqn:modified-accuracy-g}. To construct $\hat v$, we follow the lead of Section \ref{sec:splice:calculation} and define
\begin{equation}
\label{eqn:v-construction}
\begin{aligned}
\hat v^0 &= \hat g^0, \\
\hat a^1 &= \hat g^1 - (\nabla^h_4 \hat v^0) \cdot (\nabla^h_4 \hat \phi), \\
\hat v^1 &= \hat v^0 + \hat a^1 \hat \phi, \\
\hat a^2 &= \hat g^\Delta - \Delta^h_4 \hat v^1 + (\Delta^h_4 \hat a^1)\hat \phi, \\
\hat v^2 &= \hat v^1 + \frac{1}{2}\hat a^2 \hat \phi^2, \\
\hat a^3 &= \hat g^{\partial_\bfn\Delta} - \nabla^h_4(\Delta^h_4 \hat v^2) \cdot (\nabla^h_4 \hat \phi), \\
\hat v^3 &= \hat v^2 + \frac{1}{6} \hat a^3 \hat\phi^3,
\end{aligned}
\end{equation}
and we can then take $\hat v = \hat v^q$ as our jump extrapolation. See Algorithm \ref{alg:jump-extrapolation} for a summary of the implementation.

\begin{algorithm}
  \caption{Construct the jump extrapolation $\hat v$.}
  \begin{itemize}
  \setlength\itemsep{0.0em}
  \item $q$ is the smoothness required (in the sense of $LC^q$) and $0 \leq q \leq 3$.
  \item $s$ is the width of the finite difference stencil $D^h_{p,q}$. Width is defined as the maximum distance between where the stencil is evaluated and any other point in the stencil.
  \item $h$ is the uniform grid spacing.
  \item $b$ is the base band width. $b = s + 2qh$ if $q\leq 2$ and $b = s + 8h$ if $q = 3$.
  \item $\hat \phi_{i,j} = \phi_{i,j} + \bigo{h^{q+1}}$, discretized signed distance function in band of width $b$.
  \item $\hat g^0_{i,j} = g^0_{i,j} + \bigo{h^{q+1}}$, discretized $\jump{u}$ in band of width $b$.
  \item $\hat g^1_{i,j} = g^1_{i,j} + \bigo{h^{q}}$, discretized $\jump{\partial_\bfn u}$ in band of width $b - 2h$, if $q \geq 1$.
  \item $\hat g^\Delta_{i,j} = g^\Delta_{i,j} + \bigo{h^{q-1}}$, discretized $\jump{\Delta u}$ in band of width $b - 4h$, if $q \geq 2$.
  \item $\hat g^{\partial_\bfn\Delta}_{i,j} = g^{\partial_\bfn\Delta}_{i,j} + \bigo{h^{q-2}}$, discretized $\jump{\partial_\bfn\Delta u}$ in band of width $b - 8h$, if $q = 3$.
  \end{itemize}
  \begin{algorithmic}[1]
  \Function{JumpExtrapolation}{$\hat \phi$, $\hat g^0$, $\hat g^1$, $\hat g^\Delta$, $\hat g^{\partial_\bfn\Delta}$, $q$, $h$, $b$}
      \For{$i,j = 1,\ldots, n$ such that $|\hat \phi_{i,j}| < b$} 
      \Comment{form $\hat v^0$ in band of width $b$}
        \Let{$\hat v_{i,j}$}{$\hat g^0_{i,j}$}
      \EndFor
      \If{$q = 0$}
        \State \Return{$\hat v$}
      \EndIf
      \For{$i,j = 1,\ldots, n$ such that $|\hat \phi_{i,j}| < b - 2h$} 
      \Comment{form $O(h^4)$ accurate $\hat \bfn$ in band of width $b - 2h$}
        \Let{$\hat n_{i,j}$}{$(\nabla^h_4 \hat \phi)_{i,j}$}
      \EndFor
      \For{$i,j = 1,\ldots, n$ such that $|\hat \phi_{i,j}| < b - 2h$} 
      \Comment{form $\hat v^1$ in band of width $b - 2h$}
        \Let{$\hat a^1_{i,j}$}{$\hat g^1_{i,j} - (\nabla^h_4 \hat v)_{i,j} \cdot \hat n_{i,j}$}
        \Let{$\hat v_{i,j}$}{$\hat v_{i,j} + \hat a^1_{i,j} \hat \phi_{i,j}$}
      \EndFor
      \If{$q = 1$}
        \State \Return{$\hat v$}
      \EndIf
      \For{$i,j = 1,\ldots, n$ such that $|\hat \phi_{i,j}| < b - 4h$} 
      \Comment{form $\hat v^2$ in band of width $b - 4h$}
        \Let{$\hat a^2_{i,j}$}{$\hat g^\Delta_{i,j} - (\Delta^h_4 \hat v)_{i,j} + (\Delta^h_4 \hat a^1)_{i,j}\hat \phi_{i,j}$}
        \Let{$\hat v_{i,j}$}{$\hat v_{i,j} + \frac{1}{2}\hat a^2_{i,j} \hat \phi_{i,j}^2$}
      \EndFor
      \If{$q = 2$}
        \State \Return{$\hat v$}
      \EndIf
      \For{$i,j = 1,\ldots, n$ such that $|\hat \phi_{i,j}| < b - 8h$} 
      \Comment{form $\hat v^3$ in band of width $b - 8h$}
        \Let{$\hat a^3_{i,j}$}{$\hat g^{\partial_\bfn\Delta}_{i,j} - (\nabla^h_4(\Delta^h_4 \hat v))_{i,j} \cdot \hat n_{i,j}$}
        \Let{$\hat v_{i,j}$}{$\hat v_{i,j} + \frac{1}{6}\hat a^3_{i,j} \hat \phi_{i,j}^3$}
      \EndFor
      \State \Return{$\hat v$}
  \EndFunction
  \end{algorithmic}
  \label{alg:jump-extrapolation}
\end{algorithm}

It should be noted that for $q < 3$ we can replace all fourth-order finite difference operators above with their second-order counterparts and still satisfy the accuracy criterion in Proposition \ref{prop:accuracy}, and for $q = 3$ we can do the same in all but the calculation of $\hat a^1$. However, we still see better numerical results with fourth-order stencils even when $q < 3$. 

With $\hat v$ in hand, evaluating $Du$ is as simple as invoking Proposition \ref{prop:jump-splice}, and we have
\[Du = D^h_{p,q}u - D^h_{p,q}(\hat v H(\hat \phi)) + (D^h_{p,q} \hat v) H(\hat \phi)  + \bigo{h^p}, \]
as desired.

\subsection{Results}
\label{sec:splice:results}

\begin{enumerate}[label=\bfseries Example \thesection.\arabic*.\ , align=left, leftmargin=0cm, itemindent=0cm, labelwidth=0cm, labelsep=0cm, ref=\thesection.\arabic*]
\item \label{ex:splice-1} We investigate the error in evaluating $\Delta u$ for
\begin{equation}
\label{eqn:splice-example1-u}
u(x, y) = (e^x y^2) H(\phi), 
\end{equation}
where the interface $\Gamma$ is an ellipse centered at $(0,0)$ with semi-principal axes $\mathbf{R} = (0.7, 0.3)$. Because there is no closed form for the signed distance function of an ellipse, we must construct $\phi$ numerically. In this paper, we use fifth-order accurate closest point techniques from Saye \cite{saye:2014}. Other approaches to computing the signed distance function can be found, for example, in \cite{adalsteinsson:1999, chopp:2006, peng:1999}. Note that, in this example, we have
\[ \left\{
\begin{aligned} 
g^0 &= e^x y^2 \\
g^1 &= (e^x y^2, 2 e^x y) \cdot \bfn \\
g^\Delta &= e^x (2 + y^2) \\
g^{\partial_\bfn\Delta} &= (e^x(2 + y^2), 2e^x y) \cdot \bfn,
\end{aligned}
\right. \]
where we compute $\bfn = \nabla_4^h \phi$. Convergence results are presented in Table \ref{tab:splice-example1}.

\begin{table}[ht]
\centering
\begin{tabular}{rcccc}
  \toprule
$n$ & $L^\infty(\Omega)$ & Rate & $L^2(\Omega)$ & Rate \\
  \midrule
64   & 7.760\e{-6} &      & 2.222\e{-6} &      \\
128  & 2.036\e{-6} & 1.9 & 5.560\e{-7} & 2.0 \\
256  & 5.129\e{-7} & 2.0 & 1.391\e{-7} & 2.0 \\
512  & 1.282\e{-7} & 2.0 & 3.494\e{-8} & 2.0 \\
1024 & 3.214\e{-8} & 2.0 & 8.730\e{-9} & 2.0 \\
2048 & 8.179\e{-9} & 2.0 & 2.182\e{-9} & 2.0 \\
  \bottomrule
\end{tabular}
\caption{Convergence results for Example \ref{ex:splice-1}. Errors are for approximating $\Delta u$ with jump splice techniques, where $u$ is given by \eqref{eqn:splice-example1-u}.}
\label{tab:splice-example1}
\end{table}
\end{enumerate}

\section{Elliptic Problems}
\label{sec:elliptic}

Having developed jump splice methodology, we now have the tools to solve elliptic problems of the form \eqref{eqn:jump-elliptic} when $\jump{\beta} = 0$. The finite difference error result in Proposition \ref{prop:jump-splice} is not only useful to approximate derivatives, but can also be readily used to invert elliptic operators, as we now show.

\subsection{Poisson Equation}
\label{sec:elliptic:poisson}

We begin with the Poisson equation given by
\begin{equation} 
\label{eqn:jump-poisson}
\left\{ \setlength\arraycolsep{2pt} \begin{array}{rll}
\Delta u &= f & \text{on } \Omega \setminus \Gamma \\
u &= h &\text{on } \partial \Omega \\
\jump{u} &= g^0 & \text{across } \Gamma \\
\jump{\partial_\bfn u} &= g^1 & \text{across } \Gamma.
\end{array} \right.
\end{equation}
Here we will assume that $g^0 \in LC^3(\Gamma)$, $g^1 \in LC^2(\Gamma)$, and $f \in LC^1(\Omega^+, \Omega^-)$. In most applications, we are also given the jumps $\jump{f}$ and $\jump{\partial_\bfn f}$. Provided this is so, \eqref{eqn:jump-poisson} immediately implies that we have
\begin{equation}
\label{eqn:jump-poisson-extra-jumps}
\begin{aligned}
g^\Delta &= \jump{f} \\
g^{\partial_\bfn \Delta} &= \jump{\partial_\bfn f},
\end{aligned}
\end{equation}
where $g^\Delta \in LC^1(\Gamma)$ and $g^{\partial_\bfn \Delta} \in LC^0(\Gamma)$ by our regularity assumption on $f$.

We will use the 5-point Laplacian $\Delta^h$ as our finite difference discretization $D^h_{p,q}$ of $D = \Delta$, for which the required smoothness is $q = 3$. We can then construct the jump extrapolation $v$ in accordance with Section \ref{sec:splice:calculation}, using the jump conditions in \eqref{eqn:jump-poisson} and \eqref{eqn:jump-poisson-extra-jumps}. Finally, we discretize the Poisson equation using \eqref{eqn:discretization} and we have
\begin{equation} 
\label{eqn:spliced-jump-poisson}
\left\{ \setlength\arraycolsep{2pt} \begin{array}{rll}
\Delta^h u &= f + \Delta^h (v H(\phi)) - (\Delta^h v) H(\phi) & \text{on } \Omega \\
u &= h &\text{on } \partial \Omega. \\
\end{array} \right.\end{equation}
The jump conditions have been fully integrated into the right-hand side of the discretized Poisson solve. Because $v$ is determined only by the jump information $g^i$ and the signed distance function $\phi$, the right-hand side does not depend on $u$. We need only invert the standard 5-point Laplacian $\Delta^h$ to solve for a second-order accurate approximation to $u$.

It should be noted that it is possible to dispense with the fourth jump condition $g^{\partial_\bfn \Delta} = \jump{\partial_\bfn f}$ and still solve \eqref{eqn:spliced-jump-poisson} with second-order accuracy. The key here is a result from Beale and Layton \cite{beale:2006}, which shows that we only need a local truncation error of $\bigo{h}$ near the interface to have an overall $\bigo{h^2}$ accurate solution to \eqref{eqn:spliced-jump-poisson}. We can thus construct $v$ with $q = 2$, which does not require the fourth jump condition, and still achieve second-order accuracy. Another consequence is that we need only satisfy the $q = 2$ accuracy conditions in Section \ref{sec:splice:accuracy} to achieve overall second-order accuracy, even if we use the $q = 3$ construction. That said, when $\jump{\partial_\bfn f}$ is available, we achieve better numerical results with the $q = 3$ solution.

\subsection{Implementation}
\label{sec:elliptic:implementation}

Under the same assumptions as in Section \ref{sec:splice:implementation}, we construct the jump extrapolation $\hat v$ from the jump information in \eqref{eqn:jump-poisson} and \eqref{eqn:jump-poisson-extra-jumps} using Algorithm \ref{alg:jump-extrapolation} with $q = 3$. We also assume we have a discrete approximation  $\hat f_{i,j} = f(\bfx_{i,j}) + \bigo{h^2}$. 

To solve the Poisson equation with jumps \eqref{eqn:jump-poisson}, we simply perform a linear solve
\[ \Delta^h u = f + \Delta^h (\hat v H(\phi)) - (\Delta^h \hat v) H(\phi), \]
where here $\Delta^h$ is imbued with the appropriate boundary condition. This system is a standard Poisson solve on a rectangular grid, and therefore can be accomplished quickly with conjugate gradients or multigrid. Note that, with geometric multigrid, this solve can be performed in just $\bigo{N}$ time, where $N$ is the total number of grid points, which is asymptotically optimal.

\subsection{Results}
\label{sec:elliptic:results}

We have performed extensive tests of the convergence and accuracy properties of the jump splice methodology applied to the Poisson equation. A few selections are presented here. In all of the following examples, we take our domain to be $\Omega = [-1,1]^d$, where $d = 2$ or $d = 3$. 

\begin{enumerate}[label=\bfseries Example \thesection.\arabic*.\ , align=left, leftmargin=0cm, itemindent=0cm, labelwidth=0cm, labelsep=0cm, ref=\thesection.\arabic*]

\item \label{ex:leveque-1} Here we compare the results of jump splice methodology to the Immersed Interface Method \cite{leveque:1994}. We take our interface $\Gamma$ to be the circle of radius $R = 0.5$ centered at the origin, and solve Laplace's equation $\Delta u = 0$ subject to the jump condition $\jump{\partial_\bfn u} = 2$ and with boundary condition given by the exact solution
\begin{equation}
\label{eqn:example1-u}
u(\bfx) = 1 + \log(2|\bfx|) \left(1 - H(\phi)\right).
\end{equation}

\begin{figure}[ht]
\centering
\includegraphics[width={0.6\linewidth}, clip=true, trim={0.6\linewidth} {0.2\linewidth} {0.6\linewidth} {0.1\linewidth}]{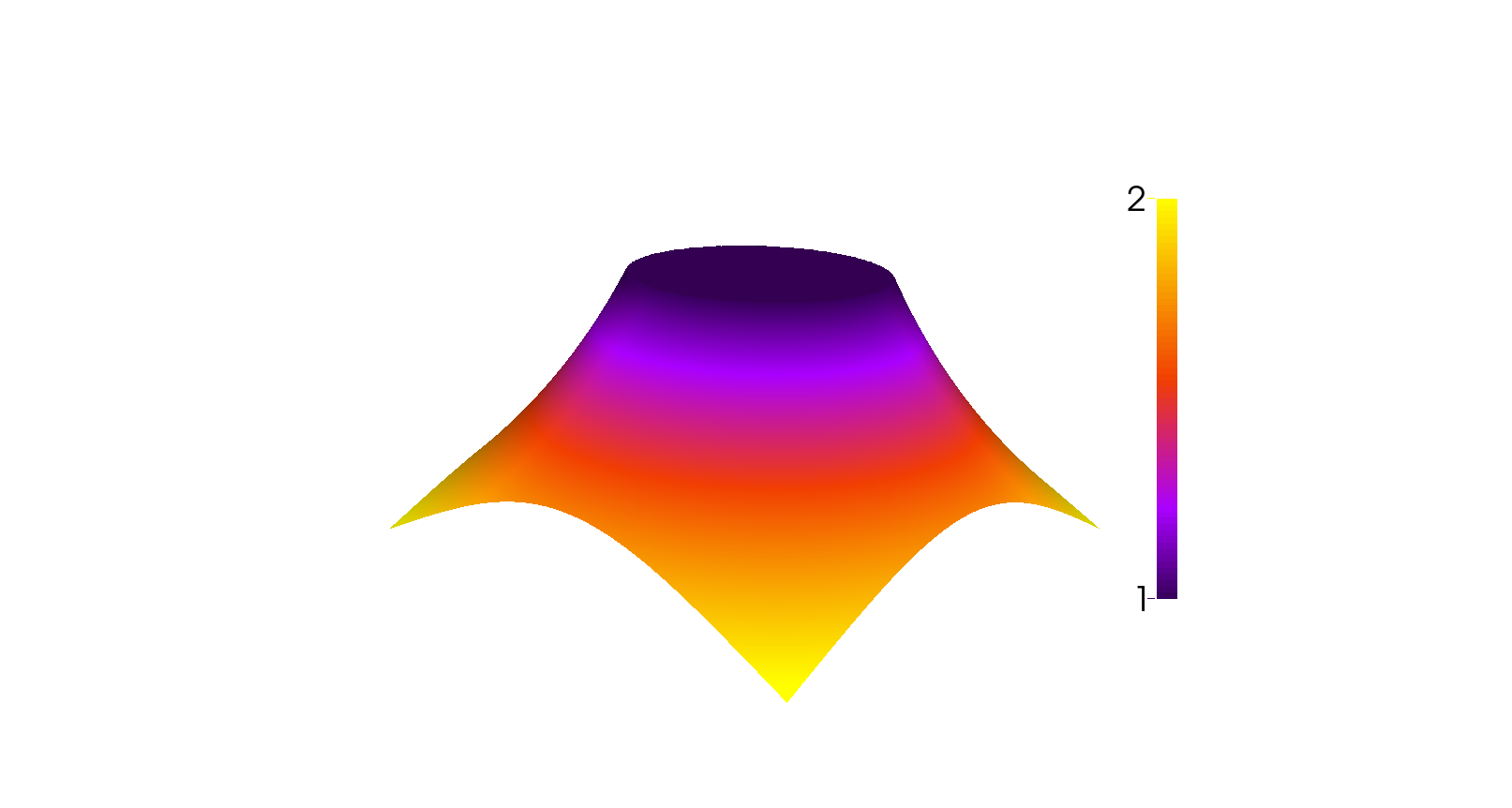}
\caption{Calculated solution $u$ in Example \ref{ex:leveque-1} on a $160 \times 160$ grid. Exact solution given by \eqref{eqn:example1-u}.}
\label{fig:leveque-1-u}
\end{figure}

%DATA: jumpPoissonPaper with 1e-14 multigrid tolerance.
\begin{table}[ht]
\centering
\begin{tabular}{rccccccc}
  \toprule
& \multicolumn{2}{c}{IIM} & \phantom{abc} & \multicolumn{4}{c}{Jump Splice} \\
  \cmidrule{2-3} \cmidrule{5-8}
$n$ & $L^\infty(\Omega)$ & Rate && $L^\infty(\Omega)$ & Rate & $L^2(\Omega)$ & Rate \\
  \midrule
20  & 2.391\e{-3} &      && 2.132\e{-3} &      & 2.259\e{-3} &\\
40  & 8.346\e{-4} & 1.5  && 5.129\e{-4} & 2.1  & 5.269\e{-4} & 2.1\\
80  & 2.445\e{-4} & 1.8  && 1.233\e{-4} & 2.1  & 1.253\e{-4} & 2.1\\
160 & 6.686\e{-5} & 1.9  && 3.206\e{-5} & 1.9  & 3.258\e{-5} & 1.9\\
320 & 1.567\e{-5} & 2.1  && 7.949\e{-6} & 2.0  & 8.064\e{-6} & 2.0\\
640 &             &      && 1.981\e{-6} & 2.0  & 2.009\e{-6} & 2.0\\
1280&             &      && 4.961\e{-7} & 2.0  & 5.030\e{-7} & 2.0\\
2560&             &      && 1.239\e{-7} & 2.0  & 1.256\e{-7} & 2.0\\
  \bottomrule
\end{tabular}
\caption{Comparison of numerical results between Immersed Interface Method (IIM) and jump splice for Example \ref{ex:leveque-1}.}
\label{tab:leveque-1}
\end{table}

The solution obtained using jump splice techniques can be seen in Figure \ref{fig:leveque-1-u}. Table \ref{tab:leveque-1} shows an analysis of convergence and a comparison to data from \cite{leveque:1994}.

\item \label{ex:leveque-3} We again compare results with the IIM in \cite{leveque:1994}. $\Gamma$ is still the circle of radius $R = 0.5$ and we again solve $\Delta u = 0$, but this time we stipulate jumps and boundary conditions such that
\begin{equation}
\label{eqn:example2-u}
u(x,y) = (e^x \cos y) H(\phi)
\end{equation}
is the exact solution. The solution obtained with jump splice can be seen in Figure \ref{fig:leveque-3-u} and convergence results are presented in Table \ref{tab:leveque-3}. 

\begin{figure}[ht]
\centering
\includegraphics[width={0.6\linewidth}, clip=true, trim={0.6\linewidth} {0.2\linewidth} {0.6\linewidth} {0.1\linewidth}]{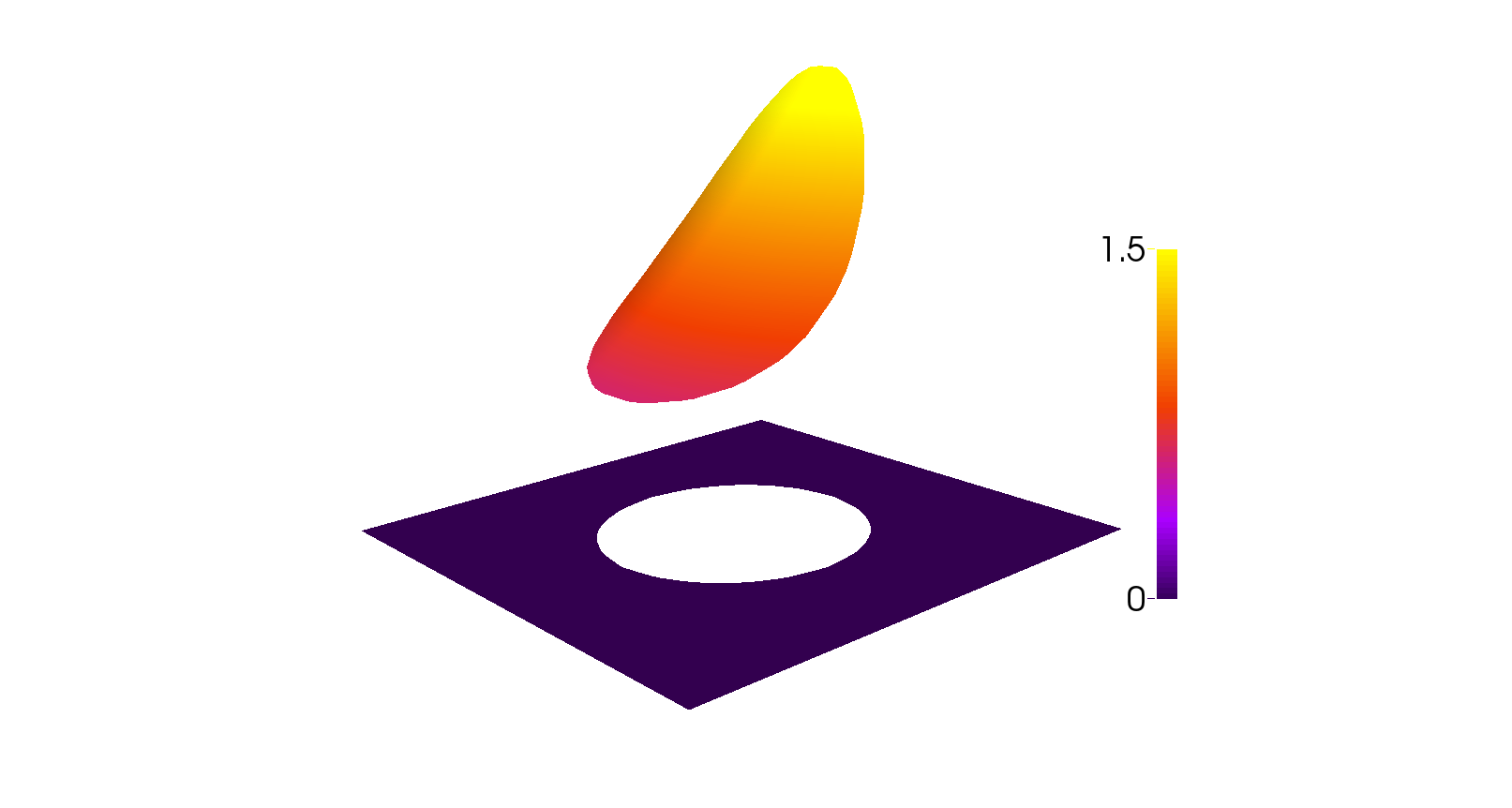}
\caption{Calculated solution $u$ in Example \ref{ex:leveque-3} on a $160 \times 160$ grid. Exact solution given by \eqref{eqn:example2-u}.}
\label{fig:leveque-3-u}
\end{figure}

\begin{table}[ht]
\centering
\begin{tabular}{rccccccc}
  \toprule
& \multicolumn{2}{c}{IIM} & \phantom{abc} & \multicolumn{4}{c}{Jump Splice} \\
  \cmidrule{2-3} \cmidrule{5-8}
$n$ & $L^\infty(\Omega)$ & Rate && $L^\infty(\Omega)$ & Rate & $L^2(\Omega)$ & Rate \\
  \midrule
20  & 4.379\e{-4} &      && 2.066\e{-2} &      & 7.980\e{-3} &\\
40  & 1.079\e{-4} & 2.0  && 6.728\e{-5} & 8.3  & 5.741\e{-5} & 7.1\\
80  & 2.778\e{-5} & 2.0  && 1.689\e{-5} & 2.0  & 1.438\e{-5} & 2.0\\
160 & 7.499\e{-6} & 1.9  && 4.209\e{-6} & 2.0  & 3.578\e{-6} & 2.0\\
320 & 1.740\e{-6} & 2.1  && 1.053\e{-6} & 2.0  & 8.950\e{-7} & 2.0\\
640 &             &      && 2.633\e{-7} & 2.0  & 2.238\e{-7} & 2.0\\
1280&             &      && 6.577\e{-8} & 2.0  & 5.589\e{-8} & 2.0\\
2560&             &      && 1.633\e{-8} & 2.0  & 1.386\e{-8} & 2.0\\
  \bottomrule
\end{tabular}
\caption{Comparison of numerical results between Immersed Interface Method (IIM) and jump splice for Example \ref{ex:leveque-3}.}
\label{tab:leveque-3}
\end{table}

\item \label{ex:quarteroni-1} We now investigate application of jump splice methodology to an interface that is $C^1$ but not $C^2$. We compare to results from the Simplified Exact Subgrid Interface Correction (SESIC) method \cite{discacciati:2013}, which is a recently developed finite element method for \eqref{eqn:jump-poisson} that performs well on non-smooth interfaces. The interface $\Gamma$ is defined by the level set function
\begin{equation}
\phi(x, y) = \begin{cases}
0.2 - \sqrt{x^2 + (\frac{1}{2} - |y|)^2} & \text{if } |y| > \frac{1}{2} \\
0.2 - |x| & \text{if } |y| \leq \frac{1}{2} \\
\end{cases}
\end{equation}
and we solve $\Delta u = 0$ with jump conditions and Dirichlet boundary conditions given by the exact solution
\begin{equation}
\label{eqn:quarteroni-1-exact}
u(\bfx) = (1 - \log(2|\bfx|)) (1 - H(\phi)). 
\end{equation}
The solution obtained with jump splice can be seen in Figure \ref{fig:quarteroni-1-u}, and convergence results are presented in Table \ref{tab:quarteroni-1}. Because the construction of the jump extrapolation for $q = 3$ requires all quantities to be defined in a band around $\Gamma$ of width approximately $10h$, the jump splice suffers from poor performance on extremely coarse grids, as seen here for $n = 20$ and $n = 40$. Results can be significantly improved by using second-order stencils in the construction of the jump extrapolation or by using the $q = 2$ construction on coarse grids. Note also that jump splice techniques were developed assuming smooth $\Gamma$, but second-order convergence is achieved here even with a $C^1$ interface.

\begin{figure}[ht]
\centering
\includegraphics[width={0.4\linewidth}]{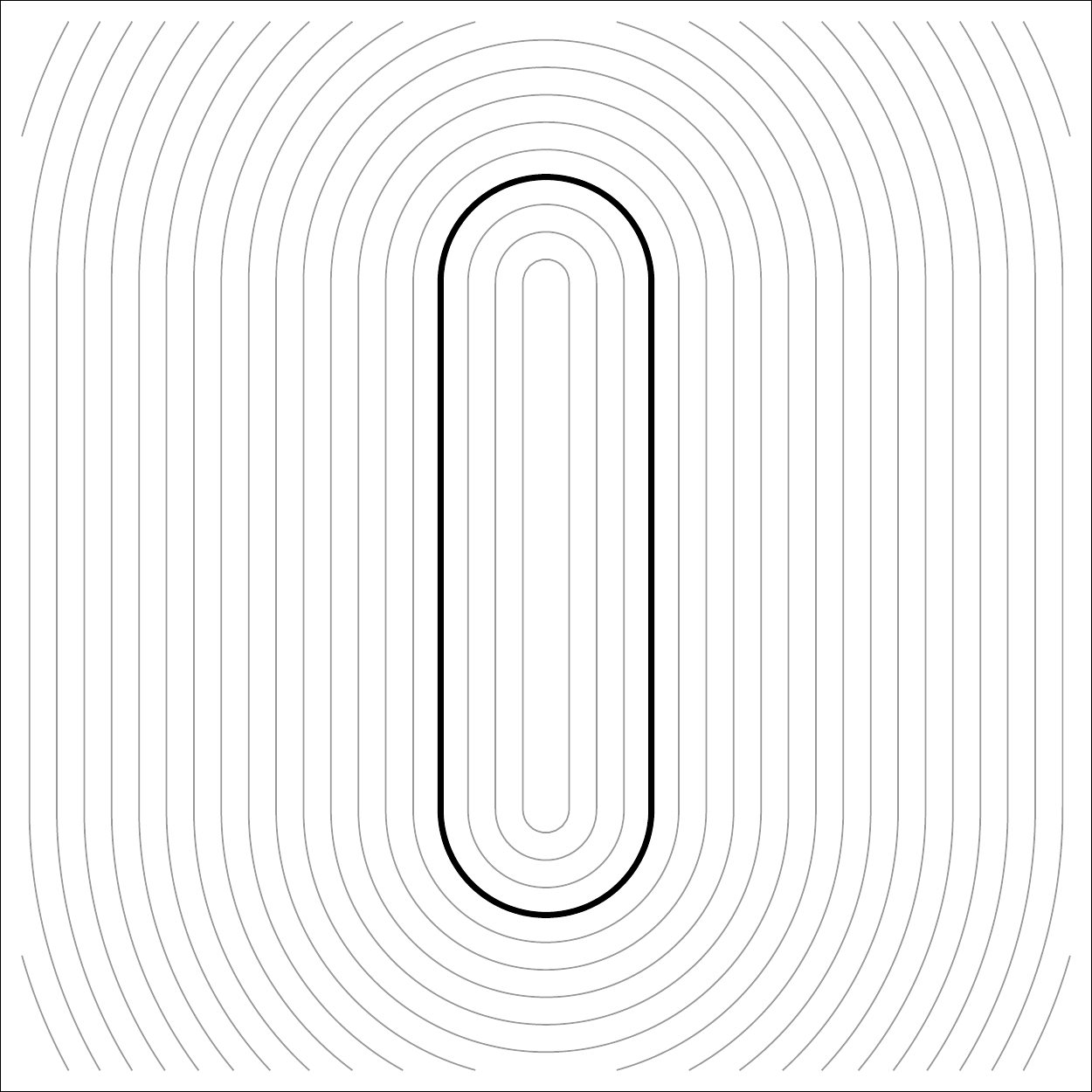}
\qquad
\includegraphics[width={0.5\linewidth}, clip=true, trim={0.9\linewidth} {0.2\linewidth} {0.8\linewidth} {0.1\linewidth}]{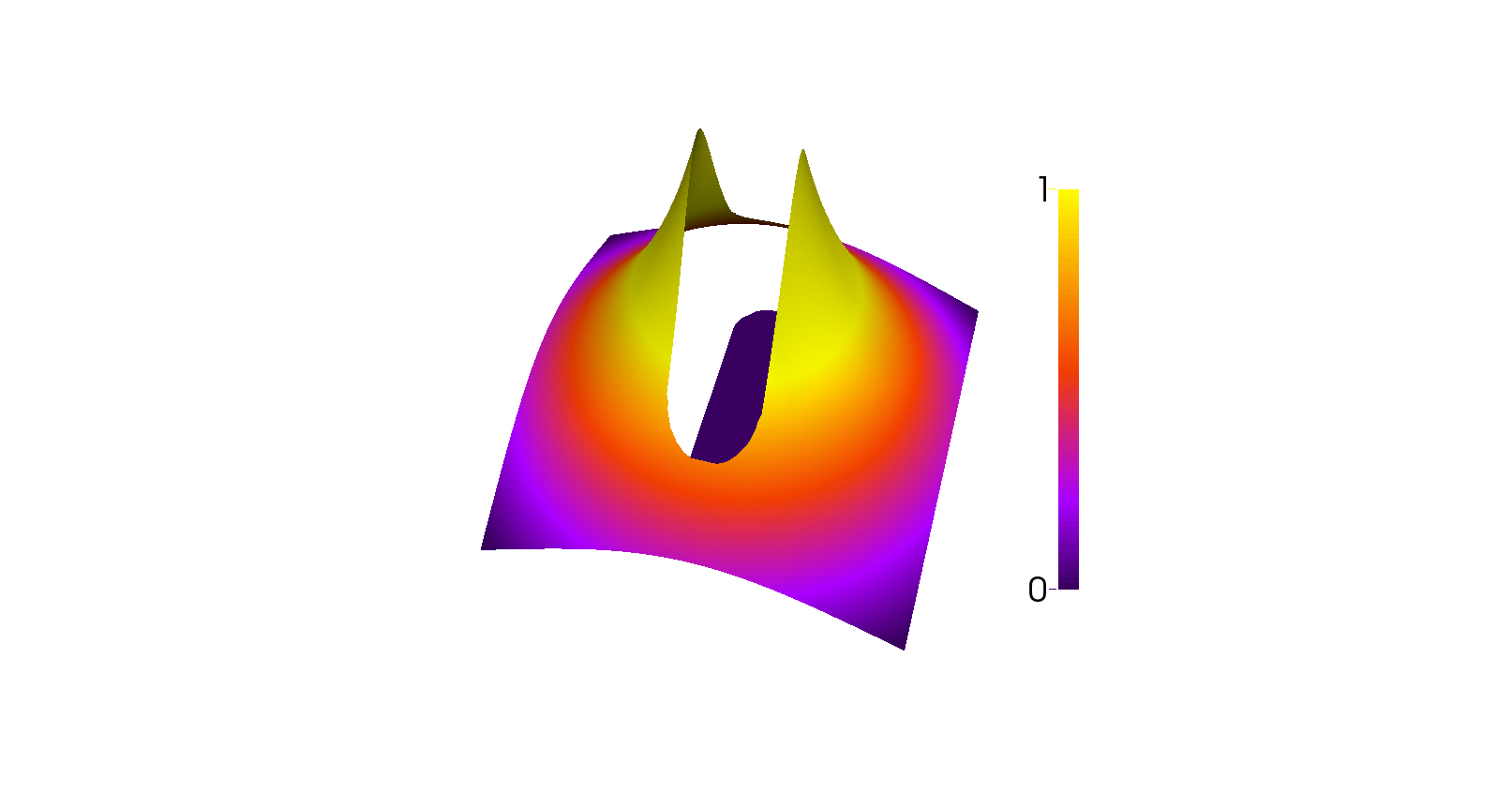}
\caption{Signed distance function (left, $\Gamma$ in bold) and computed solution $u$ to Example \ref{ex:quarteroni-1} on a $160 \times 160$ grid. Exact solution given by \eqref{eqn:quarteroni-1-exact}.}
\label{fig:quarteroni-1-u}
\end{figure}

\begin{table}[ht]
\centering
\begin{tabular}{rccccccccc}
  \toprule
& \multicolumn{4}{c}{SESIC} & \phantom{abc} & \multicolumn{4}{c}{Jump Splice} \\
  \cmidrule{2-5} \cmidrule{7-10}
$n$ & $L^\infty(\Omega)$ & Rate & $L^2(\Omega)$ & Rate && $L^\infty(\Omega)$ & Rate & $L^2(\Omega)$ & Rate \\
  \midrule
19/20   & 3.50\e{-2} &      & 1.19\e{-2} &      && 2.226\e{-1} &      & 1.485\e{-1} &      \\
39/40   & 1.09\e{-2} & 1.6  & 3.21\e{-3} & 1.8  && 1.489\e{-1} & 0.6  & 7.179\e{-2} & 1.1  \\
79/80   & 3.24\e{-3} & 1.7  & 8.83\e{-4} & 1.8  && 8.849\e{-4} & 7.4  & 3.032\e{-4} & 7.9  \\
159/160 & 1.02\e{-3} & 1.7  & 2.65\e{-4} & 1.7  && 2.120\e{-4} & 2.1  & 7.251\e{-5} & 2.1  \\
320     &            &      &            &      && 5.293\e{-5} & 2.0  & 1.813\e{-5} & 2.0  \\
640     &            &      &            &      && 1.323\e{-5} & 2.0  & 4.533\e{-6} & 2.0  \\
1280    &            &      &            &      && 3.307\e{-6} & 2.0  & 1.134\e{-6} & 2.0  \\
2560    &            &      &            &      && 8.268\e{-7} & 2.0  & 2.834\e{-7} & 2.0  \\
  \bottomrule
\end{tabular}
\caption{Comparison of numerical results between Simplified Exact Subgrid Interface Correction (SESIC) method and jump splice for Example \ref{ex:quarteroni-1}.}
\label{tab:quarteroni-1}
\end{table}

\item \label{ex:quarteroni-0} Next we apply jump splice methodology to an interface that is $C^0$ but not $C^1$, and compare once again with SESIC. The interface $\Gamma$ is defined by the level set function
\begin{equation}
\phi(x, y) = \begin{cases}
0.5 - \sqrt{(x - \frac{\sqrt{2}}{4})^2 + y^2} & \text{if } x \geq 0 \\
0.5 - \sqrt{(x + \frac{\sqrt{2}}{4})^2 + y^2} & \text{if } x < 0, \\
\end{cases}
\end{equation}
which we reconstruct into a signed distance function using fifth-order techniques from \cite{saye:2014}, and the exact solution is the same as given by \eqref{eqn:quarteroni-1-exact}, now with different $\phi$. Solution obtained with jump splice can be seen in Figure \ref{fig:quarteroni-0-u}, and convergence results are presented in Table \ref{tab:quarteroni-0}.

\begin{figure}[ht]
\centering
\includegraphics[width=0.4\linewidth]{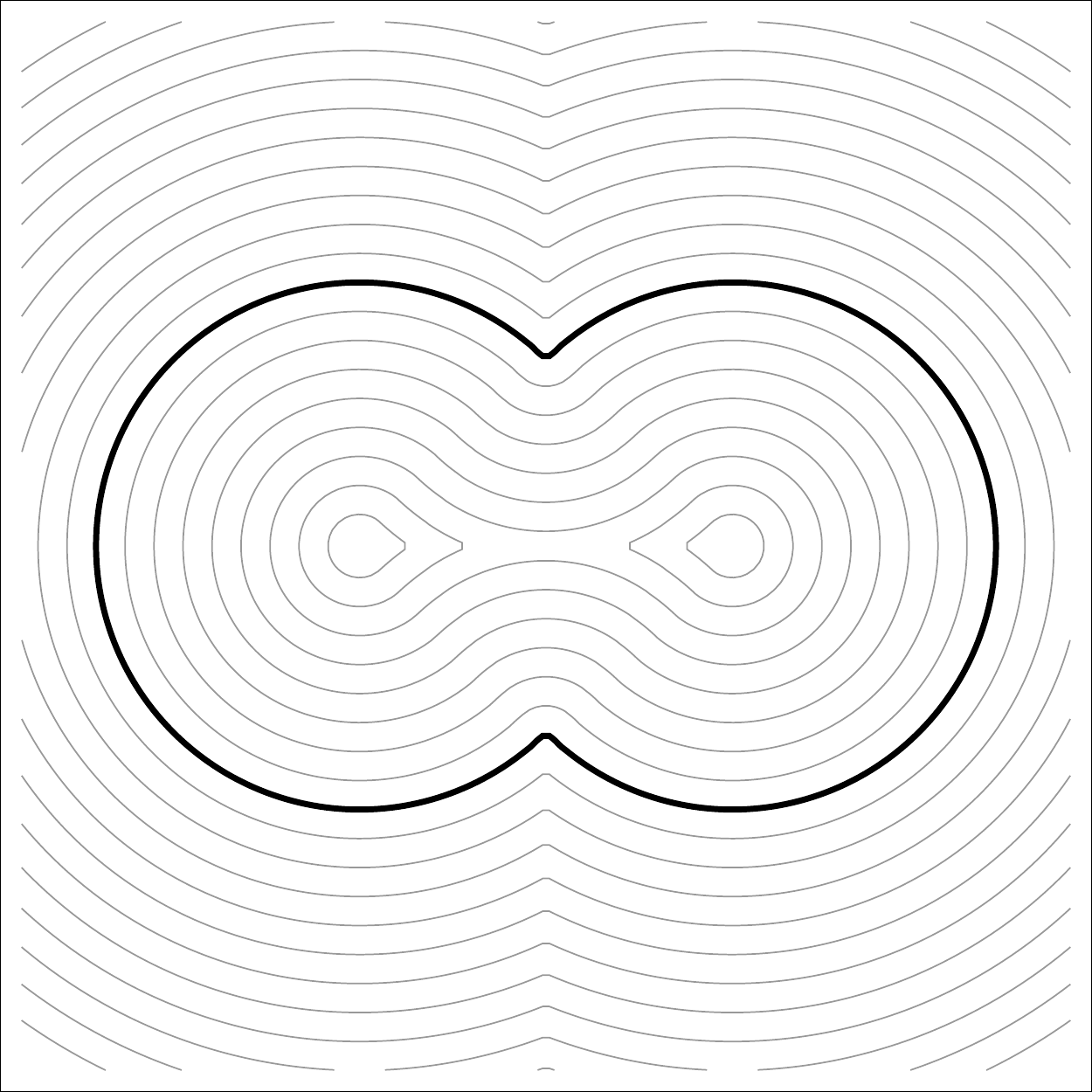}
\qquad
\includegraphics[width={0.5\linewidth}, clip=true, trim={0.9\linewidth} {0.2\linewidth} {0.8\linewidth} {0.1\linewidth}]{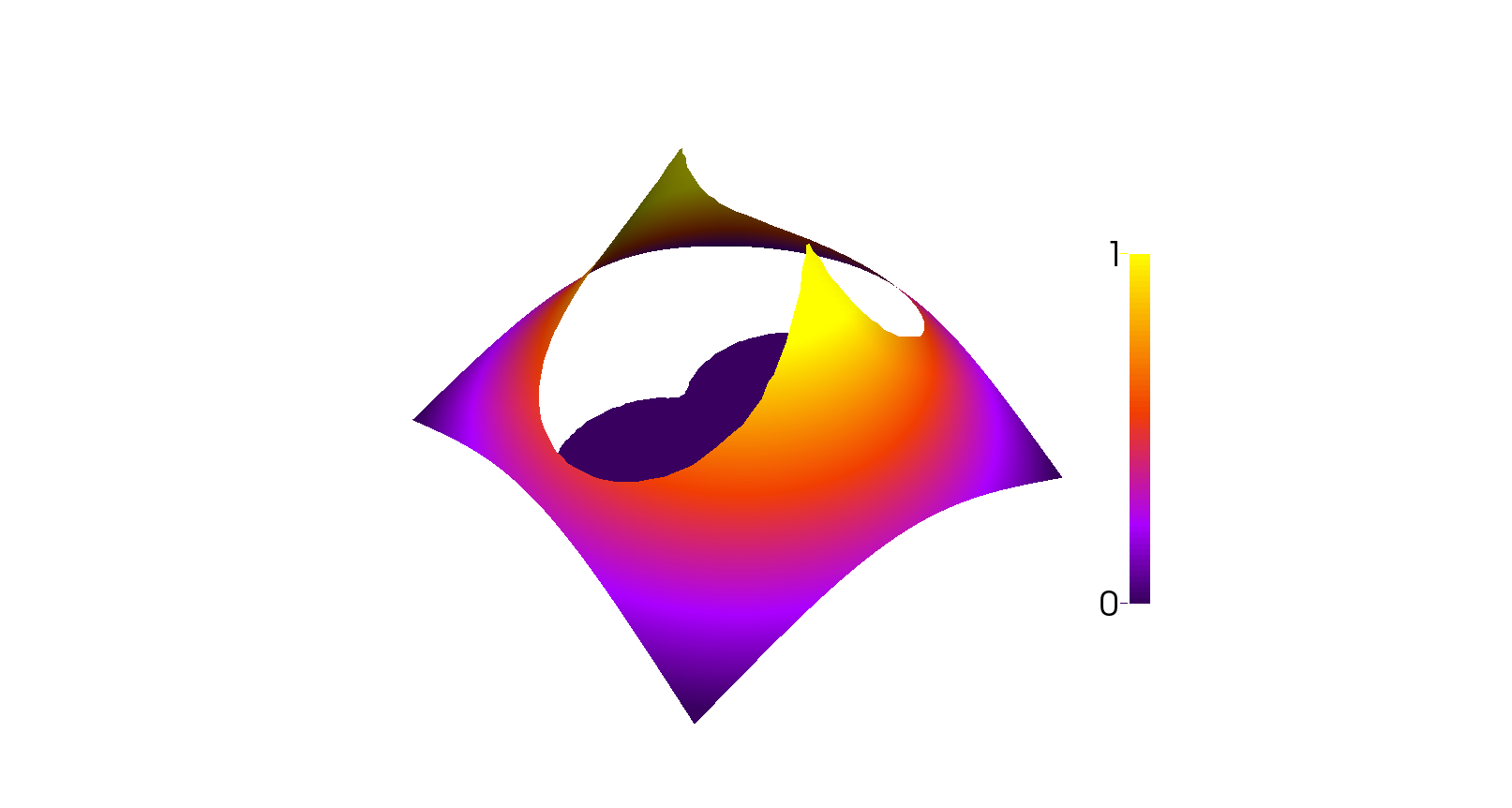}
\caption{Signed distance function (left, $\Gamma$ in bold) and computed solution $u$ to Example \ref{ex:quarteroni-0} on a $160 \times 160$ grid. Exact solution given by \eqref{eqn:quarteroni-1-exact}.}
\label{fig:quarteroni-0-u}
\end{figure}

\begin{table}[ht]
\centering
% First Tabular: With reconstruction and with approximate normal.
\begin{tabular}{rccccccccc}
  \toprule
& \multicolumn{4}{c}{SESIC} & \phantom{abc} & \multicolumn{4}{c}{Jump Splice} \\
  \cmidrule{2-5} \cmidrule{7-10}
$n$ & $L^\infty(\Omega)$ & Rate & $L^2(\Omega)$ & Rate && $L^\infty(\Omega)$ & Rate & $L^2(\Omega)$ & Rate \\
  \midrule
19/20   & 5.13\e{-2} &      & 2.73\e{-2} &      && 1.822\e{-1} &      & 6.730\e{-2} &      \\
39/40   & 2.87\e{-2} & 0.8  & 1.41\e{-2} & 0.9  && 7.645\e{-3} & 4.6  & 3.895\e{-3} & 4.1  \\
79/80   & 1.65\e{-2} & 0.8  & 7.12\e{-3} & 1.0  && 6.363\e{-3} & 0.3  & 3.669\e{-3} & 0.1  \\
159/160 & 1.00\e{-2} & 0.7  & 4.07\e{-3} & 1.0  && 2.970\e{-3} & 1.1  & 1.373\e{-3} & 1.4  \\
320     &            &      &            &      && 1.755\e{-3} & 0.8  & 7.215\e{-4} & 0.9  \\
640     &            &      &            &      && 9.087\e{-4} & 1.0  & 3.174\e{-4} & 1.2  \\
1280    &            &      &            &      && 5.457\e{-4} & 0.7  & 1.784\e{-4} & 0.8  \\
2560    &            &      &            &      && 3.110\e{-4} & 0.8  & 9.234\e{-5} & 1.0  \\
  \bottomrule
\end{tabular}
% Second Tabular: Without reconstruction of level set function and with exact normal.
%\begin{tabular}{rccccccccc}
%  \toprule
%& \multicolumn{4}{c}{SESIC} & \phantom{abc} & \multicolumn{4}{c}{Jump Splice} \\
%  \cmidrule{2-5} \cmidrule{7-10}
%$n$ & $L^\infty(\Omega)$ & Rate & $L^2(\Omega)$ & Rate && $L^\infty(\Omega)$ & Rate & $L^2(\Omega)$ & Rate \\
%  \midrule
%19/20   & 5.13\e{-2} &      & 2.73\e{-2} &      && 1.839\e{-1} &      & 6.200\e{-2} &      \\
%39/40   & 2.87\e{-2} & 0.8  & 1.41\e{-2} & 0.9  && 2.065\e{-2} & 3.2  & 1.237\e{-2} & 2.3  \\
%79/80   & 1.65\e{-2} & 0.8  & 7.12\e{-3} & 1.0  && 1.174\e{-2} & 0.8  & 5.171\e{-3} & 1.3  \\
%159/160 & 1.00\e{-2} & 0.7  & 4.07\e{-3} & 1.0  && 6.738\e{-3} & 0.8  & 2.682\e{-3} & 1.0  \\
%320     & -          & -    & -          & -    && 4.114\e{-3} & 0.7  & 1.539\e{-3} & 0.8  \\
%640     &            &      &            &      && 2.430\e{-3} & 0.8  & 7.875\e{-4} & 1.0  \\
%1280    &            &      &            &      && 1.310\e{-3} & 0.9  & 3.936\e{-4} & 1.0  \\
%2560    &            &      &            &      && 7.027\e{-4} & 0.9  & 1.958\e{-4} & 1.0  \\
%  \bottomrule
%\end{tabular}
\caption{Comparison of numerical results between Simplified Exact Subgrid Interface Correction (SESIC) method and jump splice for Example \ref{ex:quarteroni-0}.}
\label{tab:quarteroni-0}
\end{table}

\item \label{ex:ellipse} Finally, we perform numerical tests in 3D. Note that the jump splice formulation is unchanged, apart from replacing the finite difference operators $\Delta^h$ and $\nabla^h$ with their 3D counterparts. We define the interface $\Gamma$ to be an ellipsoid with semi-principal axes $\mathbf{R} = (0.7, 0.3, 0.5)$. The signed distance function $\phi$ is again constructed using fifth-order techniques from \cite{saye:2014}. The jump and boundary conditions are given by the exact solution
\begin{equation}
\label{eqn:example5-u}
u(\bfx) = \left(\frac{1}{|\bfx|}\right)(1 - H(\phi)),
\end{equation}
to the equation $\Delta u = 0$. Numerical results are presented in Table \ref{tab:ellipse}.

\begin{table}[ht]
\centering
\begin{tabular}{rcccc}
  \toprule
$n$ & $L^\infty(\Omega)$ & Rate & $L^2(\Omega)$ & Rate \\
  \midrule
64  & 2.969\e{-3} &      & 6.513\e{-4} &     \\
128 & 7.802\e{-4} & 1.9  & 1.616\e{-4} & 2.0 \\
256 & 1.952\e{-4} & 2.0  & 4.068\e{-5} & 2.0 \\
512 & 4.790\e{-5} & 2.0  & 1.007\e{-5} & 2.0 \\
  \bottomrule
\end{tabular}
\caption{Convergence results in 3D for Example \ref{ex:ellipse}. Errors are for solving $\Delta u = 0$, in the presence of jumps across an ellipsoid, using jump splice techniques.}
\label{tab:ellipse}
\end{table}

\end{enumerate}

\subsection{General Elliptic Problem}
\label{sec:general-elliptic}

Until now we have exclusively discussed the Poisson equation \eqref{eqn:jump-poisson}, but we now return to the general elliptic equation \eqref{eqn:jump-elliptic} with which we began.

When $\beta$ is smooth across the interface, jump splicing methods apply naturally to solving \eqref{eqn:jump-elliptic}. In particular, we can write $\jump{\partial_\bfn u} = g^1 / \beta$, and $-\nabla \cdot (\beta \nabla u)$ can be discretized as a symmetric positive-definite finite difference operator $D^h_{2,3}$ derived from a variational formulation, as in \cite{almgren:1998}. We can then appeal to \eqref{eqn:discretization} to arrive at a symmetric positive-definite, second-order discretization of \eqref{eqn:jump-elliptic}.

When $\beta$ is discontinuous across the interface, the jump splice framework cannot directly discretize \eqref{eqn:jump-elliptic}. To see why, observe that we can write
\begin{equation}
\label{eqn:jump-avg}
\jump{\beta \partial_\bfn u} = \jump{\beta} \avg{\partial_\bfn u} + \avg{\beta} \jump{\partial_\bfn u},
\end{equation}
where $\avg{u}(\bfx) = (u^+(\bfx) + u^-(\bfx))/2$ denotes the average value of a function $u$ across the interface for $\bfx \in \Gamma$. Though $\jump{\beta \partial_\bfn u}$, $\jump{\beta}$, and $\avg{\beta}$ are given by the formulation of the problem, $\avg{\partial_\bfn u}$ is unknown, and thus we are unable to solve for the jump condition $\jump{\partial_\bfn u}$. Without this, we cannot construct the jump extrapolation $v$ given by \eqref{eqn:v-definition}.

If $\beta$ is constant on each side of the interface, we can resolve the lack of information by introducing an unknown function $\lambda$ defined in $\Gamma_\epsilon$. We then simultaneously solve the modified general elliptic problem given by
\begin{equation} 
\label{eqn:lagrange-jump-elliptic}
\left\{ \setlength\arraycolsep{2pt} \begin{array}{rll}
-\Delta u &= f/\beta & \text{on } \Omega \setminus \Gamma \\
u &= h &\text{on } \partial \Omega \\
\jump{u} &= g^0 & \text{across } \Gamma \\
\jump{\partial_\bfn u} &= \lambda & \text{across } \Gamma,
\end{array} \right.
\end{equation}
and the constraint $\jump{\beta \partial_\bfn u} = g^1$ using ideas similar to those in \cite{li:1998}. The key to enforcing the constraint is to observe that we can approximate
\begin{equation}
\label{eqn:taylor-avg}
\avg{u}(\bfx) = u(\bfx) - v(\bfx) H(\phi(\bfx)) + \frac{1}{2} v(\bfx) + \bigo{\phi^4},
\end{equation}
where $v$ is the $q = 3$ jump extrapolation associated with \eqref{eqn:lagrange-jump-elliptic} and $\bfx \in \Gamma$. We can similarly approximate $\avg{\partial_\bfn u}$ by replacing $v$ with $\partial_\bfn v$ and $u$ with $\partial_\bfn u$ in \eqref{eqn:taylor-avg}. The constraint can then be written as
\begin{equation}
\label{eqn:lagrange-constraint}
g^1 = \jump{\beta} \left( \partial_\bfn u - (\partial_\bfn v) H + \frac{1}{2}\partial_\bfn v\right) + \avg{\beta} \lambda,
\end{equation}
and together \eqref{eqn:lagrange-jump-elliptic} and \eqref{eqn:lagrange-constraint} lead to linear system that can be solved to recover $u$. Note that $v$ depends only on $g^0$, $\lambda$, $f$, and $\beta$, and the mapping between $\lambda$ and $v$ is linear, as will be shown in Section \ref{sec:ns:temporal:jump-operators}.

Unfortunately, symmetry of the linear system is lost with this approach, and obtaining a symmetric method is the subject of current work.

\subsection{Discussion}
\label{sec:splice:discussion}

The examples in Section \ref{sec:elliptic:results} show robust second-order convergence for the jump splice method applied to solving the Poisson equation on a variety of different interfaces, in both 2D and 3D. In particular, although we derived the jump splice method by assuming that the interface $\Gamma$ was smooth, Example \ref{ex:quarteroni-1} shows that we still achieve second-order convergence with a $C^1$ interface. Example \ref{ex:quarteroni-0} goes further and shows that we still achieve roughly first-order convergence on a $C^0$ interface, where the unit normal is not strongly well-defined everywhere. We also note that the numerical errors of the jump splice are remarkably small, typically less than those seen for IIM or SESIC. Finally, because the jump splice allows use of standard symmetric positive-definite linear solvers, we achieve excellent computational performance; calculations on a $512 \times 512$ grid with one core require just $2$ seconds using basic geometric multigrid, and less than 10\% of the execution time is spent building the jump spliced right-hand side.

\section{Integration}
\label{sec:integration}

We now briefly illustrate the versatility of the jump splice by showing how Proposition \ref{prop:jump-splice} can be used to perform integration over implicitly defined surfaces. See \cite{engquist:2005, smereka:2006, wen:2009} for other approaches to this type of quadrature with level sets. We will use the methods described here to calculate the volume enclosed by an interface when we examine convergence in volume for the Navier-Stokes equations in the next section.

\subsection{Implicit Surface Integrals}
\label{sec:integration:surface}

We can use jump splice techniques to evaluate integrals of the form
\begin{equation}
\label{eqn:surface-integral}
I = \int_\Gamma \alpha\, ds, 
\end{equation}
where the interface $\Gamma$ is defined implicitly by a signed distance function and where we assume $\alpha \in LC^3(\Gamma_\epsilon)$. This is particularly useful for obtaining highly accurate calculations of volume and surface area, because
\begin{equation}
\label{eqn:integration-area}
\operatorname{Area}(\Gamma) = \int_\Gamma 1\, ds
\end{equation}
and
\begin{equation}
\label{eqn:integration-volume}
 \operatorname{Volume}(\Omega^+) = -\int_\Gamma (\bfx \cdot \bfe_1) (\bfn \cdot \bfe_1) \, ds,
\end{equation}
where $\bfe_1$ is the unit vector along the first Cartesian coordinate axis and \eqref{eqn:integration-volume} follows by the divergence theorem, recalling that $\bfn$ is the inward unit normal. The term area here refers to codimension-one measure, typically referred to as perimeter in two dimensions and surface area in three dimensions. We will make extensive use of \eqref{eqn:integration-volume} in investigating volume conservation when applying jump splice techniques to the Navier-Stokes equations in Section \ref{sec:ns}.

To see how the jump splice is used, note that by the coarea formula, we can rewrite this integral as
\[ I = \int_{\Omega} \alpha(\bfx) \delta(\phi(\bfx))\, d\bfx, \]
recalling that because $\phi$ here is taken to be a signed distance function, we have $|\nabla \phi| \equiv 1$. Next we observe that the distributional elliptic equation
\[
\left\{ \setlength\arraycolsep{2pt} \begin{array}{rll}
\Delta u &= \alpha \delta(\phi) & \text{on } \Omega \\
u &= 0 &\text{on } \partial \Omega, \\
\end{array} \right.
\]
can be written in the form of a Poisson equation with jumps, as in \eqref{eqn:jump-poisson} with $g^0 = 0$, $g^1 = \alpha$, and $f = 0$, and thus can be solved numerically using jump splice methodology as
\[
\left\{ \setlength\arraycolsep{2pt} \begin{array}{rll}
\Delta^h_p u &= \Delta^h_p (v H(\phi)) - (\Delta^h_p v) H(\phi) & \text{on } \Omega \\
u &= 0 &\text{on } \partial \Omega, \\
\end{array} \right.
\]
where $\Delta^h_p$ is an order $p$ accurate approximation to the Laplacian such that $\Delta u = \Delta^h_p u + \bigo{h^p}$ provided $u \in LC^q$ with $q = p+1$. Here $v$ is constructed as in Section \ref{sec:splice:calculation}, with $g^\Delta = 0$ and $g^{\partial_\bfn\Delta} = 0$. 

By analogy between the distributional elliptic equation and its discretization, we can see that a good approximation for $\alpha \delta(\phi)$ is given by
\begin{equation}
\label{eqn:fdelta}
\delta_{\alpha,p}^h =  \Delta^h_p (v H(\phi)) - (\Delta^h_p v) H(\phi). 
\end{equation}
We can then formulate a discretization of the integral $I$ as
\[ \hat I = h^d\sum_{i_1,\ldots, i_d} (\delta^h_{\alpha,p})_{i_1,i_2,\ldots, i_d}, \]
for $\Omega \subset \bbr^d$. Numerical experiments, including those in the next section, indicate that
\[ \hat I = I + \bigo{h^p}. \]
A detailed analysis of the convergence properties of this quadrature rule is the subject of future work.

\subsection{Results}
\label{sec:integration:results}

We have performed extensive convergence tests for jump splice integration, and we present a few examples below. Once again, we take our domain to be $\Omega = [-1,1]^d$, where $d = 2$ or $d =3$.

In the following examples, we use the fourth-order accurate discretization of the Laplacian $\Delta^h_4$, for which $p = 4$ and $q = 5$ in the notation of Section \ref{sec:splice:splice}. However, we construct the jump extrapolation $v$ only up to order $q = 3$. While $v$ constructed this way does not allow us to evaluate $\Delta^h_4$ with fourth order accuracy, we still achieve fourth-order accurate integration, as shown in the results below.

\begin{enumerate}[label=\bfseries Example \thesection.\arabic*.\ , align=left, leftmargin=0cm, itemindent=0cm, labelwidth=0cm, labelsep=0cm, ref=\thesection.\arabic*]

\item \label{ex:integration-area} We test jump splice integration by evaluating the perimeter $P$ of a circle $\Gamma$ with radius $R = 0.5$ centered at the origin $(0,0)$. We use the fourth-order Laplacian $\Delta^h_4$ along with \eqref{eqn:fdelta} to evaluate the integral given in \eqref{eqn:integration-area}. The exact result is $P = \pi$. See Table \ref{tab:integration-area} for convergence results.

\begin{table}[ht]
\centering
\begin{tabular}{rcc}
  \toprule
$n$ & Error & Rate \\
  \midrule
64   & 5.422\e{-5}  &      \\
128  & 3.142\e{-6}  & 4.1  \\
256  & 1.610\e{-7}  & 4.3  \\
512  & 1.311\e{-8}  & 3.6  \\
1024 & 7.062\e{-10} & 4.2  \\
2048 & 1.283\e{-11} & 5.8  \\
  \midrule
\multicolumn{2}{c}{Average} & 4.4 \\
  \bottomrule
\end{tabular}
\caption{Convergence results for Example \ref{ex:integration-area}. Errors are given for evaluating the perimeter of a circle using jump splice integration.}
\label{tab:integration-area}
\end{table}

\item \label{ex:integration-ex} We integrate the function $\alpha(x,y) = e^x$ over an ellipse $\Gamma$ with semi-principal axes $\mathbf{R} = (0.35, 0.7)$. As in Example \ref{ex:integration-area}, we use the fourth-order Laplacian $\Delta^h_4$ along with \eqref{eqn:fdelta}. The answer is given to ten decimal places by
\[ \int_\Gamma \alpha\, ds \approx 3.5123690943. \]
See Table \ref{tab:integration-ex} for convergence results. 

\begin{table}[ht]
\centering
\begin{tabular}{rcc}
  \toprule
$n$ & Error & Rate \\
  \midrule
64   & 2.289\e{-4}  &      \\
128  & 1.414\e{-5}  & 4.0 \\
256  & 7.309\e{-7}  & 4.3 \\
512  & 7.100\e{-8}  & 3.4 \\
1024 & 7.116\e{-9}  & 3.3 \\
2048 & 1.250\e{-10} & 5.8 \\
  \midrule
\multicolumn{2}{c}{Average} & 4.2 \\
  \bottomrule
\end{tabular}
\caption{Convergence results for Example \ref{ex:integration-ex}. Errors are given for evaluating the surface integral of $e^x$ over an ellipse using jump splice integration.}
\label{tab:integration-ex}
\end{table}

\item \label{ex:integration-vol} We use \eqref{eqn:integration-volume} to evaluate the volume of an ellipsoid $\Gamma$ with semi-principal axes ${\mathbf{R} = (0.35, 0.7, 0.5)}$. We use the 3D analog of $\Delta^h_4$ along with \eqref{eqn:fdelta}. The exact answer is given by $V = \pi(0.35)(0.7)(0.5)$. See Table \ref{tab:integration-vol} for convergence results.

\begin{table}[ht]
\centering
\begin{tabular}{rcc}
  \toprule
$n$ & Error & Rate \\
  \midrule
64   & 3.801\e{-5} &      \\
128  & 7.703\e{-7} & 5.6 \\
256  & 9.100\e{-8} & 3.1 \\
512  & 4.445\e{-9} & 4.4 \\
  \midrule
\multicolumn{2}{c}{Average} & 4.4  \\
  \bottomrule
\end{tabular}
\caption{Convergence results for Example \ref{ex:integration-vol}. Errors are given for calculating the volume on an ellipsoid using jump splice integration.}
\label{tab:integration-vol}
\end{table}

\end{enumerate}

\clearpage

\section{Application to Incompressible Navier-Stokes Equations}
\label{sec:ns}

Singular forces at a fluid-fluid interface, as occur in surface tension and membrane elasticity, give rise to jumps in the fluid velocity $\bfu$ and pressure $p$. A vast literature exists on methods (see, for example, \cite{peskin:2002, li:2001, leveque:2003}) to solve the incompressible Navier-Stokes equations in the presence of singular forces, and some of these approaches smooth out the discontinuities in $\bfu$ and $p$ and thereby achieve only first-order accuracy. Our goal is to use jump splice techniques to solve the incompressible Navier-Stokes equations and, by preserving discontinuities, obtain second-order accurate solutions in the presence of singular forces. 

This section illustrates the versatility of jump splice methodology; here we must not only solve elliptic equations with prescribed jumps, but also evaluate derivatives arbitrarily close to the interface. The jump splice unifies these tasks into a single coherent framework. We proceed as follows.

\begin{itemize}
\setlength\itemsep{0.2em}
\item We begin by reviewing the singular force Navier-Stokes equations and their corresponding jump conditions in Section \ref{sec:ns:singular-force}.
\item Next, in Section \ref{sec:ns:projection}, we discuss a basic projection method used to solve for fluid flow in the absence of singular forces.
\item In Section \ref{sec:ns:temporal}, we extend jump splice techniques to handle quantities that vary in both time and space. To do this, we introduce temporal jump splicing for time derivatives and jump operators for the determination of intermediate quantities in the projection method.
\item We then use these techniques to modify the approximate projection method to accommodate jumps in the velocity and pressure while preserving second-order accuracy in Sections \ref{sec:ns:projection-spliced} and \ref{sec:ns:spatial}.
\item In Sections \ref{sec:ns:surface-tension} and \ref{sec:ns:implementation}, we restrict to the case of surface tension and describe the full algorithm in detail.
\item Finally, in Section \ref{sec:ns:results}, we show extensive convergence results and compare with the smoothed $\delta$ approach.
\end{itemize}

\subsection{Singular Force Navier-Stokes Equations}
\label{sec:ns:singular-force}

The singular force Navier-Stokes equations are typically written as
\begin{equation}
\label{eqn:singularns}
\left\{ \setlength\arraycolsep{2pt} \begin{array}{rll}
\rho\left(\partial_t \bfu + (\bfu \cdot \nabla) \bfu \right) &= - \nabla p + \mu \Delta \bfu + \bff \delta(\phi) & \text{in } \Omega \\
\nabla \cdot \bfu &= 0 &\text{in } \Omega \\
\bfu &= 0 &\text{on } \partial \Omega,\\
\end{array} \right.
\end{equation}
where $\rho$ and $\mu$ denote density and viscosity and are herein assumed to be constant, $p$ is the scalar pressure field, $\bfu$ is the fluid velocity field, and $\bff$ represents all singular interface forces. We do not include a bulk forcing term here, but none of the resulting analysis is changed by including an additional non-singular force on the right-hand side.

The singular force $\bff \delta(\phi)$ in \eqref{eqn:singularns} gives rise to discontinuities in the velocity and pressure across the interface that are entirely determined by $\bff$. In what follows, we assume $\bff$ is defined in a band $\Gamma_\epsilon$ around the interface, and we decompose $\bff$ into tangential and normal components as
\[\bff = \bff_s + f_\bfn \bfn, \]
where $f_\bfn = \bff \cdot \bfn$ and $\bff_s \cdot \bfn = 0$. Lai and Li \cite{lai:2001} as well as Xu and Wang \cite{xu:2006} have shown that
\begin{equation}
\label{eqn:nsjumpconditions1}
\begin{aligned}[c]
\jump{\bfu} &= 0  \\
\jump{\partial_\bfn \bfu} &= -\frac{1}{\mu}\bff_s
\end{aligned}
\qquad
\begin{aligned}[c]
\jump{p} &= f_\bfn\\
\jump{\partial_\bfn p} &= \nabla_s \cdot \bff_s,
\end{aligned}
\end{equation}
where we have written the jump conditions in coordinate-independent form. From these conditions, by differentiating\footnote{We expect $\bfu$ and $p$ to be smooth on $\Omega \setminus \Gamma$, as \eqref{eqn:singularns} reduces to the viscous Navier-Stokes equations on either side of the interface.} \eqref{eqn:singularns} on each side of the interface and taking jumps, we have that
\begin{equation}
\label{eqn:nsjumpconditions2}
\begin{aligned}[c]
\jump{\Delta \bfu } 
&= \frac{1}{\mu}\biggl((\nabla_s \cdot \bff_s) \bfn + \nabla_s f_\bfn\biggr)  \\
\jump{\partial_\bfn \Delta \bfu} 
&= -\frac{\rho}{\mu^2}\biggl(\partial_t \bff_s + \nabla\bff_s\cdot\bfu + \nabla \bfu \cdot \bff_s - 2(\bfn \cdot \nabla \bfu \cdot \bff_s)\bfn +  (\bfn \cdot \nabla \bfu \cdot \bfn)\bff_s \biggr) \\
&+- \frac{1}{\mu}\biggl((\Delta_s f_\bfn)\bfn + \kappa (\nabla_s \cdot \bff_s)\bfn + \nabla_s f_\bfn \cdot \nabla \bfn - \nabla_s(\nabla_s \cdot \bff_s)\biggr) \\
\jump{\Delta p} &= \frac{2\rho}{\mu} \biggl( \bfn \cdot \nabla\bfu \cdot \bff_s \biggr) \\
\jump{\partial_\bfn \Delta p} &= \frac{2\rho}{\mu}\biggl(\bfn \cdot \nabla (\partial_\bfn \bfu) \cdot \bff_s - (\nabla_s \cdot \bff_s)(\bfn \cdot \nabla \bfu \cdot \bfn) - \kappa (\bfn \cdot \nabla \bfu \cdot \bff_s) + \Tr(\nabla_s \bff_s \cdot \nabla \bfu) \biggr. \\
&\phantom{=\frac{2\rho}{\mu}(} \biggl. - \bfn \cdot \nabla \bfu \cdot \nabla \bfn \cdot \bff_s - \bfn \cdot \nabla\bfu \cdot \nabla_s f_\bfn + \mu^{-1}(\bfn \cdot \nabla_s \bff_s \cdot \bff_s)\biggr)
\end{aligned}
\end{equation}
These equations provide all of the information needed to discretize \eqref{eqn:singularns} using the jump splice framework. 

\subsection{Approximate Projection Method}
\label{sec:ns:projection}

In the absence of singular forces, and thus in the absence of jump conditions, we solve the Navier-Stokes equations using an approximate projection method based on \cite{almgren:1996}, which is in turn based on earlier work in \cite{chorin:1968, bell:1989}. In particular, we discretize in time as
\begin{subequations}
\label{eqn:projectionmethod}
\begin{align}
\label{eqn:projectionmethod-1}
\frac{\bfu^{*} - \bfu^n}{\Delta t} &= -(\bfu^n \cdot \nabla) \bfu^n - \frac{1}{\rho}\nabla p^n + \frac{\mu}{\rho} \Delta \bfu^*\\
\frac{\bfu^{n+1} - \bfu^*}{\Delta t} &= -\frac{1}{\rho} \nabla \psi \\
\frac{p^{n+1} - p^n}{\Delta t} &= \frac{1}{\Delta t} \psi,
\end{align}
\end{subequations}
where the pressure update $\psi$ is determined by solving
\begin{equation}
\label{eqn:pressureupdate}
\left\{ \setlength\arraycolsep{4pt} \begin{array}{rll}
\Delta \psi &= \frac{\rho}{\Delta t} \left(\nabla \cdot \bfu^*\right) & \text{in } \Omega \\
\nabla \psi \cdot \bfnu &= 0 &\text{on } \partial \Omega,
\end{array} \right.
\end{equation}
where $\bfu^k$ and $p^k$ denote quantities evaluated at time $t = t_k$ for $k = n, n+1$ and $\bfnu$ is the outward normal to $\partial\Omega$. We also enforce $\bfu^* = 0$ on $\partial\Omega$ in \eqref{eqn:projectionmethod-1}. The scheme defined by \eqref{eqn:projectionmethod} and \eqref{eqn:pressureupdate} leads to a method that is first-order accurate in time. This is sufficient for our purposes, as singular force simulations tend to have stringest CFL constraints such that the time step is limited more by stability than by accuracy; for example, surface tension requires a time step of $\Delta t = \bigo{h^{3/2}}$, as shown in \cite{brackbill:1992}.

Spatial discretization is straightforward. We use a second-order Essentially Non-Oscillatory (ENO) method from \cite{osher:2002} for the advection term, second-order centered differences for calculating gradients, and the standard five-point Laplacian for both the viscous term and the elliptic pressure update solve. Importantly, we employ an offset grid such that $\bfu$ takes values on cell centers, $p$ and $\psi$ take values on cell nodes, and the gradient and divergence operators, $\nabla^h$ and $\nabla^h\cdot$, take cell-centered fields to node-centered fields and vice-versa. The numerical boundary conditions for the pressure update solve follow from the finite element method formulation in \cite{almgren:1996}, ensuring the symmetry of $\Delta^h$ in the presence of Neumann boundary conditions on a node-centered grid. This results in a method that is fully second-order accurate in space and quite simple to implement with the use of standard symmetric elliptic solvers for the viscous and pressure linear systems.

\subsection{Temporal Jump Splice}
\label{sec:ns:temporal}

Before we can apply jump splice techniques to the projection method, we need to develop the final pieces of theory that will allow us to discretize quantities that depend on both space and time. In Section \ref{sec:ns:temporal:jumps}, we show that Proposition \ref{prop:jump-splice} can be adapted to differentiation in time without explicitly calculating jumps in the time derivatives. Then, in Section \ref{sec:ns:temporal:jump-operators}, we introduce the concept of a jump operator, which will allow us to determine appropriate jump extrapolations for the intermediate quantities $\bfu^*$ and $\psi$ in \eqref{eqn:projectionmethod} and \eqref{eqn:pressureupdate}.

\subsubsection{Temporal Jumps}
\label{sec:ns:temporal:jumps}

If a time-varying function $u : \Omega \times [0,T] \to \bbr$ is discontinuous in space across a moving interface $\Gamma$, it will in general also be discontinuous in time. As a result, the standard first-order temporal finite difference operator may not achieve its expected order of accuracy at grid points near the interface. However, there is a straightforward solution. 

Fix a grid point $\bfx_{i,j}$ and suppose that $\phi(\bfx_{i,j}, t_n) < 0$. A temporal discontinuity exists at $\bfx_{i,j}$ only when the interface $\Gamma$, across which $u$ has a spatial discontinuity, crosses $\bfx_{i,j}$. Let $v$ be the $q \geq 1$ jump extrapolation of $u$ from \eqref{eqn:v-definition}, where all quantities now depend on time. Because the outer splice $w^- = u - vH(\phi)$ is at least $LC^1(\Gamma_\epsilon)$ in space, it thus follows that $w^-$ is at worst $LC^1$ is time. Then by a standard jump splicing argument
\[(\partial_t u)(\bfx_{i,j},t_n) = \partial^h_t (u - v H(\phi))(\bfx_{i,j}, t_n) + \bigo{\Delta t}.\]
Note that here $\partial^h_t$ is the standard first-order forward difference operator in time. Conversely, if $\phi(\bfx_{i,j},t_n) > 0$, we use the inner splice and have
\[(\partial_t u)(\bfx_{i,j},t_n) = \partial^h_t (u + v (1 - H(\phi)))(\bfx_{i,j}, t_n) + \bigo{\Delta t},\]
Combining these expressions yields, for arbitrary $\bfx_{i,j}$,
\[(\partial_t u)(\bfx_{i,j},t_n) = (\partial^h_t u)(\bfx_{i,j},t_n) + (\partial^h_t v)(\bfx_{i,j},t_n) H(\phi(\bfx_{i,j},t_n)) - \partial^h_t(v H(\phi))(\bfx_{i,j}, t_n) + \bigo{\Delta t}. \]
This expression is just \eqref{eqn:discretization} from Proposition \ref{prop:jump-splice} with $D = \partial_t$ and $D_{p,q}^h = \partial_t^h$ (with $p = q = 1$), except that we never had to directly calculate the temporal jump conditions $\jump{\partial_t^i u}$, as they are implicitly determined from the spatial jumps encoded in $v$. We can simplify this expression further by writing $u^n_{i,j} = u(\bfx_{i,j},t_n)$ and similarly for $v$ and $\phi$, and then we have
\begin{equation}
\label{eqn:splicetime}
(\partial_t u)(\bfx_{i,j},t_n) = \frac{u_{i,j}^{n+1} - u^n_{i,j}}{\Delta t} - \left(\frac{H(\phi_{i,j}^{n+1}) - H(\phi_{i,j}^n)}{\Delta t}\right) v^{n+1}_{i,j} + \bigo{\Delta t}.
\end{equation}
This is the spliced temporal difference operator.

\subsubsection{Jump Operators}
\label{sec:ns:temporal:jump-operators}

We now introduce the notion of a jump operator, which generalizes the canonical jump extrapolation discussed in Section \ref{sec:splice:extrapolation}, and which which will in turn allow us to naturally determine appropriate jump extrapolations for the intermediate quantities in a time evolution equation. In particular, we will use jump operators in the next section to determine jump extrapolations for $\bfu^*$ and $\psi$ in \eqref{eqn:projectionmethod} and \eqref{eqn:pressureupdate}.

The mapping between a function $u$ with jump conditions $g^i = \jump{\partial_\bfn^i u}$ for $0 \leq i \leq q$ and its canonical jump extrapolation $v$, from \eqref{eqn:v-definition}, can be written as
\begin{equation}
J_q(u) = \bar g^0 + \bar g^1 \phi + \cdots + \frac{1}{q!} \bar g^q \phi^q,
\end{equation}
and we refer to $J_q$ as a jump operator. Jump operators are valuable because they are linear in their argument $u$. Suppose we have two functions $u_1$, $u_2$ with respective jump conditions $g_1^i = \jump{\partial_\bfn^i u_1}$ and $g_2^i = \jump{\partial_\bfn^i u_2}$. Then because jumps are linear, the function $u_1 + u_2$ has jump conditions $g_1^i + g_2^i = \jump{\partial_\bfn^i(u_1 + u_2)}$, and thus
\begin{align*}
J_k(u_1 + u_2) 
&= \overline{g_1^0 + g_2^0} + (\overline{g_1^1 + g_2^1}) \phi + \cdots + \frac{1}{q!}\left(\overline{g_1^q + g_2^q}\right)\phi^q \\
&= \bar g_1^0 + \bar g_1^1\phi + \cdots + \frac{1}{q!} \bar g_1^q \phi^q \\
&+ \bar g_2^0 + \bar g_2^1\phi + \cdots + \frac{1}{q!} \bar g_2^q \phi^q \\
&= J_q(u_1) + J_q(u_2),
\end{align*}
where we have used that the constant normal extrapolation of a sum is the sum of the constant normal extrapolations, that is, $\overline{g_1^i + g_2^i} = \bar g_1^i + \bar g_2^i$. By a similar argument, we have $J_q(c u_1) = c J_q(u_1)$ for any $c \in \bbr$. 

Additionally, jump operators commute, up to order $\bigo{\phi^q}$, with the gradient. That is, 
\begin{equation} 
\label{eqn:jumpopderiv}
J_q(\nabla u) = \nabla J_q(u) + \bigo{\phi^q}.
\end{equation}
To see this, note that because $J_q(u)$ satisfies the jump extrapolation conditions \eqref{eqn:v-conditions} in place of $v$, we have $\subbar{D(J_q(u))}{\Gamma} = \jump{Du}$ for any linear differential operator $D$ with highest derivatives of order less than or equal to $q$. In particular,
\[\subbar{\partial_\bfn^i \nabla J_q(u)}{\Gamma} = \jump{\partial_\bfn^i \nabla u}, \quad \text{for } 0 \leq i \leq q-1.\] 
Thus $\nabla J_q(u)$ satisfies the $0 \leq i \leq q-1$ jump conditions for $\nabla u$, and thus differs from $J_{q-1}(\nabla u)$ by at most $\bigo{\phi^q}$, in accordance with Proposition \ref{prop:v-equivalence}. As $J_q(\nabla u)$ and $J_{q-1}(\nabla u)$ also differ by a term of order $\bigo{\phi^q}$, \eqref{eqn:jumpopderiv} follows.

Finally, we note the useful relationship
\begin{equation}
\label{eqn:jumpopjump}
J_q(J_q(u) H(\phi)) = J_q(u),
\end{equation}
as $\jump{\partial_\bfn^i J_q(u) H(\phi)} = \subbar{\partial_\bfn^i J_q(u)}{\Gamma} = g^i$ for $0 \leq i \leq q$. 

Linearity, combined with \eqref{eqn:jumpopderiv} and \eqref{eqn:jumpopjump} allow the jumps of intermediate quantities in a jump evolution equation to be readily calculated. In particular, these relationships play a key role in deriving the jump spliced version of the approximate projection method, as will be demonstrated shortly.

\subsection{Jump Spliced Projection Method}
\label{sec:ns:projection-spliced}

Now we return to the projection method, given by equations \eqref{eqn:projectionmethod} and \eqref{eqn:pressureupdate}, and make the appropriate modifications to accommodate jumps induced by the singular force. 

First, let
\[ \bfv_\bfu = J_3(\bfu),\]
and
\[ v_p = J_3(p),\]
be the $q = 3$ jump extrapolations of $\bfu$ and $p$, respectively. These are constructed by appealing to the jump conditions given in \eqref{eqn:nsjumpconditions1} and \eqref{eqn:nsjumpconditions2}. We require $q = 3$ to achieve overall second-order accuracy in space when applying second-order differential operators, as discussed in Section \ref{sec:splice:splice}.

We use the level set method \cite{osher:1988, sethian:1999, osher:2002} to track the location of the interface. We have
\begin{equation}
\label{eqn:levelsetmethod}
\frac{\phi^{n+1} - \phi^n}{\Delta t} = -(\bfu^n \cdot \nabla) \phi^n.
\end{equation}
Because $\phi^{n+1}$ defined above will not, in general, be a signed distance function, we will need to reconstruct the signed distance function every time step.

Next, we adjust the temporal discretization of the Navier-Stokes equations in \eqref{eqn:projectionmethod} by adding temporal jump splicing, obtaining
\begin{subequations}
\label{eqn:timesplicedprojectionmethod}
\begin{align}
\label{eqn:timesplicedprojectionmethod-u1}
\frac{\bfu^{*} - \bfu^n}{\Delta t} &= -(\bfu^n \cdot \nabla) \bfu^n - \frac{1}{\rho}\nabla p^n + \frac{\mu}{\rho} \Delta \bfu^* \\
\label{eqn:timesplicedprojectionmethod-u2}
\frac{\bfu^{n+1} - \bfu^*}{\Delta t} &= -\frac{1}{\rho} \nabla \psi + \bfv_\bfu^{n+1} \left(\frac{H(\phi^{n+1}) - H(\phi^n)}{\Delta t}\right) \\
\label{eqn:timesplicedprojectionmethod-pressure}
\frac{p^{n+1} - p^n}{\Delta t} &= \frac{1}{\Delta t} \psi + v_p^{n+1}\left(\frac{H(\phi^{n+1}) - H(\phi^n)}{\Delta t}\right)
\end{align}
\end{subequations}
Note that \eqref{eqn:timesplicedprojectionmethod-u1} and \eqref{eqn:timesplicedprojectionmethod-u2} together constitute the discretization of a single temporal derivative of $\bfu$, and thus generate just one temporal splice correction. At this point, the jump conditions for $\bfu^{n+1}$ and $p^{n+1}$ are fully determined by \eqref{eqn:nsjumpconditions1} and \eqref{eqn:nsjumpconditions2}, so all that remains is to ascertain suitable jump conditions for the intermediate functions $\bfu^*$ and $\psi$. For this, we use jump operators.

In \eqref{eqn:timesplicedprojectionmethod}, there are two interfaces under consideration, $\Gamma^n = \Gamma(t_n)$ with signed distance function $\phi^n$ and $\Gamma^{n+1} = \Gamma(t_{n+1})$ with signed distance function $\phi^{n+1}$. As a result, there are two distinct jump operators, $J_3^n$ at time $t_n$ and $J_3^{n+1}$ at time $t_{n+1}$. Moreover, we have
\[ J^n_3(\bfu^n) = \bfv^n_\bfu, \qquad J^{n+1}_3(\bfu^n) = 0, \]
and likewise for $p$, as $J^n_3$ is only nonzero for quantities with explicitly defined jumps across the interface $\Gamma^n$. In practice, all quantities we consider will have discontinuities for only one of these two jump operators.

We apply $J^n_3$ to \eqref{eqn:timesplicedprojectionmethod-pressure}, obtaining
\[ \frac{J_3^n(p^{n+1}) - J_3^n(p^n)}{\Delta t} = \frac{1}{\Delta t} J_3^n(\psi) + \frac{1}{\Delta t} J_3^n(v_p^{n+1} H(\phi^{n+1})) - \frac{1}{\Delta t}J_3^n(v_p^{n+1} H(\phi^n)),\]
where we have made extensive use of the linearity of $J$. Using \eqref{eqn:jumpopjump} and the definition of $v_p$, this reduces to
\begin{equation}
\label{eqn:jump-phi-n}
J_3^n(\psi) = v_p^{n+1} - v_p^n,
\end{equation} 
and this determines the jump condition for $\psi$ across $\Gamma^n$. Next, we repeat the same process with $J^{n+1}_3$ and obtain
\begin{equation}
\label{eqn:jump-phi-n+1}
J_3^{n+1}(\psi) = 0.
\end{equation}
These equations fully determine the jump conditions for $\psi$ that will be imposed when we solve the pressure update equation \eqref{eqn:pressureupdate} and that will be utilized in accurately evaluating $\nabla \psi$ in \eqref{eqn:projectionmethod-1}.

We proceed similarly for $\bfu^*$ in \eqref{eqn:timesplicedprojectionmethod-u2}. Applying $J^n_3$ gives
\[ \frac{J_3^n(\bfu^{n+1}) - J_3^n(\bfu^*)}{\Delta t} = -\frac{1}{\rho} J_3^n(\nabla \psi) + \frac{1}{\Delta t} J_3^n(\bfv_\bfu^{n+1} H(\phi^{n+1})) - \frac{1}{\Delta t} J_3^n(\bfv_\bfu^{n+1} H(\phi^n)),\]
and using linearity, along with $\eqref{eqn:jumpopderiv}$, and neglecting terms of order $\bigo{\phi^3\Delta t}$, this reduces to
\begin{align}
\label{eqn:jump-u-n}
J_3^n(\bfu^*) 
&= \bfv_\bfu^{n+1} + \frac{\Delta t}{\rho} \nabla J_3^n(\psi) \nonumber \\ 
&= \bfv_\bfu^{n+1} + \frac{\Delta t}{\rho}\left( \nabla v_p^{n+1} - \nabla v_p^n\right)
\end{align} 
Applying $J^{n+1}_3$ and reducing then leads to
\begin{equation}
\label{eqn:jump-u-n+1}
J_3^{n+1}(\bfu^*) = 0.
\end{equation}
These equations fully determine the jump conditions for $\bfu^*$ that will be imposed when we solve the backward Euler update in eqref{eqn:projectionmethod-1}. 

All that remains is to determine the temporal-spliced version of the pressure update equation \eqref{eqn:pressureupdate}, which can easily be seen to be
\begin{equation}
\label{eqn:timesplicedprojectionsolve}
\Delta \psi = \frac{\rho}{\Delta t} \nabla \cdot \bfu^* + \rho (\nabla \cdot \bfv_\bfu^{n+1}) \left(\frac{H(\phi^{n+1}) - H(\phi^n)}{\Delta t}\right),
\end{equation}
where $\nabla \cdot \bfu^*$ will be evaluated and $\Delta \psi$ inverted using jump splice techniques informed by \eqref{eqn:jump-phi-n}, \eqref{eqn:jump-phi-n+1}, \eqref{eqn:jump-u-n}, and \eqref{eqn:jump-u-n+1}. 

\subsection{Spatial Discretization}
\label{sec:ns:spatial}

Following the lead of the previous section, we now numerically approximate the spatial derivatives in \eqref{eqn:levelsetmethod} and \eqref{eqn:timesplicedprojectionmethod}.

For the evolution of the interface, we use the second-order ENO method described in \cite{osher:2002}. That is,
\begin{equation}
\label{eqn:levelsetmethod-eno}
\phi^{n+1} = \phi^n - \Delta t \, \operatorname{ENO}(\bfu^n, \phi^n).
\end{equation}
As discussed in the previous section, we must reconstruct $\phi^{n+1}$ into a signed distance function every time step, and for this we use the fifth-order accurate closest point method from \cite{saye:2014}.

Next, we make repeated use of \eqref{eqn:discretization} from Proposition \ref{prop:jump-splice} and discretize \eqref{eqn:timesplicedprojectionmethod-u1} as
\begin{align}
\label{eqn:ustar-discretized}
\left(I - \frac{\mu\Delta t}{\rho} \Delta^h\right)\bfu^* &= \bfu^n - \Delta t\, \text{JENO}(\bfu^n,\bfu^n) - \frac{\Delta t}{\rho}\biggl(\nabla^h p^n+ (\nabla^h v_p^n) H(\phi^n) - \nabla^h(v_p^n H(\phi^n)\biggr)\\
&+ \frac{\mu\Delta t}{\rho}\biggl( (\Delta^h \bfv_*) H(\phi^n) - \Delta^h (\bfv_* H(\phi^n))\biggr) \nonumber,
\end{align}
where $\Delta^h$ is the standard five-point Laplacian, $\text{JENO}$ refers to the second-order jump-spliced ENO method (see below), $\nabla^h$ is the node-to-cell-centered grid second-order finite difference gradient operator, and
\begin{equation}
\label{eqn:v-star-numeric}
\bfv_* = J_3^n(\bfu^*) = \bfv_\bfu^{n+1} + \frac{\Delta t}{\rho}\left( \nabla^h v_p^{n+1} - \nabla^h v_p^n\right),
\end{equation} 
as given by \eqref{eqn:jump-u-n} in the previous section. We enforce $\subbar{\bfu^*}{\partial\Omega} = 0$. All quantities on the right-hand side of \eqref{eqn:ustar-discretized} are known from data at time $t_n$, so a straightforward symmetric solve is all that is required to obtain $\bfu^*$. 

Because ENO is inherently nonlinear, we cannot appeal to \eqref{eqn:discretization} to obtain a jump-spliced adjustment. Instead, we calculate jump-spliced ENO (JENO) by applying standard second-order ENO, as given in \cite{osher:2002}, to $\bfu^n$, $\bfu^n - \bfv_\bfu^n H(\phi^n)$, and $\bfu^n + \bfv_\bfu^n (1 - H(\phi^n))$ at points $\bfx$ with $|\phi^n(\bfx)| > 2h$, $-2h < \phi^n(\bfx) < 0$, and $0 < \phi^n(\bfx) < 2h$, respectively, where $2h$ comes from the maximum stencil width of second-order ENO. In other words, we must apply ENO to the inner and outer splices directly, instead of being able to invoke \eqref{eqn:discretization}.

Next, we discretize \eqref{eqn:timesplicedprojectionsolve} as
\begin{align}
\label{eqn:pressureupdate-discretized}
\Delta^h \psi &= \frac{\rho}{\Delta t} \biggl(\nabla^h \cdot \bfu^* + (\nabla^h \cdot \bfv_*) H(\phi^n) - \nabla^h \cdot (\bfv_* H(\phi^n))\biggr) \\
&+ \rho (\nabla \cdot \bfv_\bfu^{n+1}) \left(\frac{H(\phi^{n+1}) - H(\phi^n)}{\Delta t}\right) \nonumber \\
&-\biggl((\Delta^h v_\psi) H(\phi^n) - \Delta^h (v_\psi H(\phi^n))\biggr) \nonumber,
\end{align}
where
\begin{equation}
\label{eqn:v-phi-numeric}
v_\psi = J_3^n(\psi) = v_p^{n+1} - v_p^n,
\end{equation}
as in \eqref{eqn:jump-phi-n}. Note that $\nabla^h\cdot$ is the cell-to-node-centered grid second-order finite difference divergence operator. Here we enforce the $\subbar{\nabla \psi \cdot \bfnu}{\partial \Omega} = 0$ boundary condition through the finite element formulation from \cite{almgren:1996}, which ensures the symmetry of $\Delta^h$. This is then a straightforward symmetric solve, and can be accomplished quickly with multigrid.

Finally, we determine $\bfu^{n+1}$ and $p^{n+1}$ with
\begin{equation}
\label{eqn:u-discretized}
\bfu^{n+1} = \bfu^* - \frac{\Delta t}{\rho}\biggl(\nabla^h \psi + (\nabla^h v_\psi)H(\phi^n) - \nabla^h(v_\psi H(\phi^n))\biggr) + \bfv_\bfu^{n+1}\left(H(\phi^{n+1}) - H(\phi^n)\right),
\end{equation} 
and
\begin{equation}
\label{eqn:p-discretized}
p^{n+1} = p^n + \psi + v_p^{n+1}\left(H(\phi^{n+1}) - H(\phi^n)\right).
\end{equation}
This method is straightforward to implement owing to the need for only standard symmetric positive-definite elliptic solvers, and is fully second-order accurate in space, as will be demonstrated numerically.

\subsection{Surface Tension}
\label{sec:ns:surface-tension}

Having developed fully second-order accurate discretizations of the singular force Navier-Stokes equations, we now restrict our attention to a particular type of singular forcing, namely surface tension. In this case, the singular force term takes the form
\[ \bff = -\sigma \kappa \bfn, \]
where $\sigma$ is the surface tension coefficient and $\kappa = \nabla \cdot \bfn$ is the mean curvature. In particular, we have $\bff_s = 0$ and $f_\bfn = -\sigma\kappa$. The jump conditions \eqref{eqn:nsjumpconditions1} and \eqref{eqn:nsjumpconditions2} become
\begin{equation}
\label{eqn:nsjumpconditions1-st}
\begin{aligned}[c]
\jump{\bfu} &= 0  \\
\jump{\partial_\bfn \bfu} &= 0
\end{aligned}
\qquad
\begin{aligned}[c]
\jump{p} &= f_\bfn\\
\jump{\partial_\bfn p} &= 0,
\end{aligned}
\end{equation}
and
\begin{equation}
\label{eqn:nsjumpconditions2-st}
\begin{aligned}[c]
\jump{\Delta \bfu } 
&= \frac{1}{\mu} \nabla_s f_\bfn  \\
\jump{\partial_\bfn \Delta \bfu} 
&= -\frac{1}{\mu}\biggl((\Delta_s f_\bfn)\bfn + \nabla_s f_\bfn \cdot \nabla \bfn\biggr) \\
\jump{\Delta p} &= 0 \\
\jump{\partial_\bfn \Delta p} &= -\frac{2\rho}{\mu}\biggl(\bfn \cdot \nabla\bfu \cdot \nabla_s f_\bfn\biggr)
\end{aligned}
\end{equation}

\subsection{Implementation of Singular Navier-Stokes for Surface Tension}
\label{sec:ns:implementation}

For the case of surface tension discussed in the previous section, we now describe the entire algorithm in full. We use a staggered grid, with $\bfu^n_{i,j}$ and $\phi^n_{i,j}$ defined on cell centers (cell-centered) and $p^n_{i,j}$ defined on cell nodes (node-centered). We will describe how these quantities at time $t_{n+1}$ are determined in a series of steps. Here we will write $\phi^n$ to denote a function with zero level set equal to $\Gamma^n$, but which may not be a signed distance function. We will write $\tilde \phi^n$ to denote the reconstruction of $\phi^n$ into a signed distance function. Furthermore, $\phi^n$ will in general only be defined in a band of width $b = 16h$ around $\Gamma^n$ for the sake of computational efficiency, as developed in \cite{adalsteinsson:1995}.

\begin{enumerate}
\item First, we use $\bfu^n$ to evolve the interface in accordance with \eqref{eqn:levelsetmethod-eno}, obtaining $\phi^{n+1}$. We do not yet reconstruct $\phi^{n+1}$ into a signed distance function.

\item Next, we form banded (width $b = 16h$) cell-centered signed distance functions $\tilde \phi^n$ and $\tilde \phi^{n+1}$ from $\phi^n$ and $\phi^{n+1}$, respectively, using the fifth-order closest point method from \cite{saye:2014}. At the same time, we also form node-centered signed distance functions $\tilde \phi^n_N$ and $\tilde \phi^{n+1}_N$, using the same technique. Achieving a high degree of fidelity in the signed distance function is essential to calculating $\kappa$ accurately, and fifth order accurate reconstruction is strictly necessary.

\item Because $\phi^{n+1}$ is defined on a band, it must be reconstructed frequently. Every 16 time steps, we overwrite $\phi^{n+1}$ with its corresponding signed distance function $\tilde\phi^{n+1}$. For more details on the choice of reconstruction frequency, see \cite{sethian:1999}.

\item Using $\tilde \phi^n_N$ and $\tilde \phi^{n+1}_N$, we calculate $\kappa^n$ and $\kappa^{n+1}$, both node-centered. Because we are using signed distance functions, we can simply compute
\[\kappa^n = \Delta^h_4 \tilde\phi^n_N, \]
and likewise for $\kappa^{n+1}$, recalling that $\Delta^h_4$ is the fourth-order accurate Laplacian defined in \eqref{eqn:laplacian-4}.

\item With curvature in hand, we form $f_\bfn^n = -\sigma\kappa^n$ and $f_\bfn^{n+1} = -\sigma \kappa^{n+1}$, again both defined on cell nodes.

\item We can now calculate the jumps in $\bfu$ and $p$. Using \eqref{eqn:nsjumpconditions1-st} and \eqref{eqn:nsjumpconditions2-st}, we have, for $\bfu^k$, where $k = n, n+1$,
\begin{equation}
\label{eqn:u-jumps-numeric}
\begin{aligned}
\bfg^{0}_{\bfu^k} &= 0 \\
\bfg^{1}_{\bfu^k} &= 0 \\
\bfg^{\Delta}_{\bfu^k} &= \frac{1}{\mu}\left(\nabla^h f_\bfn^k - \left(\nabla^h f_\bfn^k \cdot \bfn^k\right)\bfn^k\right) \\
\bfg^{\partial_\bfn \Delta}_{\bfu^k} &= - \nabla^h \cdot \bfg^{\Delta}_{\bfu^k} - \bfg^{\Delta}_{\bfu^k} \cdot \nabla^h \bfn^k,
\end{aligned}
\end{equation}
where here $\nabla^h$ denotes the appropriate (cell-cell or node-cell) second-order centered finite difference operator and $\bfn^k = \nabla^h \tilde \phi^k$ is defined at cell centers. Similarly, for $p^k$,
\begin{equation}
\label{eqn:p-jumps-numeric}
\begin{aligned}
g^{0}_{p^k} &= f_\bfn^k \\
g^{1}_{p^k} &= 0 \\
g^{\Delta}_{p^k} &= 0 \\
g^{\partial_\bfn\Delta}_{p^k} &= - 2\rho (\bfn_N^k \cdot \nabla^h \bfu^k \cdot \bfg^{\Delta}_{\bfu^k}),
\end{aligned}
\end{equation}
where $\bfn_N^k = \nabla^h \tilde \phi^k_N$ is now defined on cell nodes and $\nabla^h$ here represents the cell-node second-order finite difference operator. In \eqref{eqn:p-jumps-numeric}, $\bfg^{\Delta}_{\bfu^k}$ is calculated on cell nodes by interpolation from cell centers. 

With \eqref{eqn:u-jumps-numeric} and \eqref{eqn:p-jumps-numeric} in hand, we can now use the techniques from Section \ref{sec:splice} to compute $\bfv^k_\bfu = J_3^k(\bfu)$, defined at cell centers, and $v^k_p = J_3^k(p)$, defined at cell nodes, both for $k = n, n+1$.

\item Next, we need to construct $\bfv_* = J_3^n(\bfu^*)$ and $v_\psi = J_3^n(\psi)$. We do this by appealing to \eqref{eqn:v-star-numeric} and \eqref{eqn:v-phi-numeric}.

\item Finally, we can proceed with the jump-spliced approximate projection method. We solve \eqref{eqn:ustar-discretized} for $\bfu^*$ using either conjugate gradients or multigrid. Then we solve \eqref{eqn:pressureupdate-discretized} for $\psi$ using multigrid. Finally, we construct $\bfu^{n+1}$ and $p^{n+1}$ in accordance with \eqref{eqn:u-discretized} and \eqref{eqn:p-discretized}.
\end{enumerate}

\subsection{Results}
\label{sec:ns:results}

We have performed extensive analysis on the convergence behavior of the jump-spliced singular Navier-Stokes equations with surface tension, and two examples are presented below. In all of the following, we take our domain to be $\Omega = [0,1]^2$.

In the following examples, we look at four different metrics of convergence: velocity, pressure, interface, and volume convergence. We perform grid convergence in velocity and pressure and in the position of the interface as no exact solution is known for the examples below. 

To determine the errors in velocity and pressure, we evaluate
\begin{equation}
\label{eqn:u-error}
E_\bfu^h = \|\bfu^h - \bfu^{2h} \|_{\infty, \infty},
\end{equation}
and
\begin{equation}
\label{eqn:p-error}
E_p^h = \|p^h - p^{2h} \|_{\infty, \infty}, 
\end{equation}
where $\bfu^h$ and $p^h$ are the velocity and pressure with grid spacing $h$. Here $\|\cdot \|_{\infty, \infty}$ denotes the $L^\infty$ norm in both space and time. Because $\bfu$ is cell-centered, and cell-centered grids at different resolutions do not share points in common, we use second-order accurate interpolation to calculate \eqref{eqn:u-error}. This is justified in the case of surface tension, as $\jump{\bfu} = \jump{\partial_\bfn \bfu} = 0$, and thus $\bfu \in LC^1(\Omega)$. 

In the examples below, $p$ is discontinuous across the interface, which can result in spurious values of \eqref{eqn:p-error} when the interface lies on opposite sides of a grid point at two different grid resolutions. To account for this effect, if for a grid point $\bfx_{i,j}$ we have $\tilde \phi^h(\bfx_{i,j}) > 0$ and $\tilde \phi^{2h}(\bfx_{i,j}) < 0$ or vice-versa, we exclude the point $\bfx_{i,j}$ from the calculation \eqref{eqn:p-error}. This exclusion is necessary for only a small fraction of points within a distance $h$ of the interface, and thus our results still account for convergence behavior arbitrarily close to discontinuities.

For the error in the position of the interface, we evaluate
\begin{equation}
\label{eqn:phi-error}
E_{\phi}^h = \|\tilde \phi^{h} - \tilde \phi^{2h} \|_{\infty, \infty},
\end{equation}
where $\tilde \phi^h$ is the signed distance function calculated with grid spacing $h$. This metric is almost identical to \eqref{eqn:u-error} except that the difference $\tilde \phi^h - \tilde \phi^{2h}$ is only evaluated in the band on which $\tilde \phi$ is defined.

Finally, we calculate error in volume as
\begin{equation}
E_{\rm Vol}^h = \|\text{Vol}(\Gamma^h) - V_0 \|_\infty,
\end{equation}
where $V_0$ is the initial volume of $\Omega^+$ at time $t = 0$ and $\text{Vol}(\Gamma^h)$ is computed from $\tilde \phi^h$ to fourth-order accuracy at each time point using techniques from Section \ref{sec:integration}. Here $\|\cdot \|_\infty$ denotes the $L^\infty$ norm in time. Note that the fluid flow is incompressible, so volume should be conserved.

\begin{enumerate}[label=\bfseries Example \thesection.\arabic*.\ , align=left, leftmargin=0cm, itemindent=0cm, labelwidth=0cm, labelsep=0cm, ref=\thesection.\arabic*]

\item \label{ex:ns-1} We solve the Navier-Stokes equations with surface tension. We take the initial interface $\Gamma$ to be an ellipse centered at $(0.5, 0.5)$ with semi-principal axes $\mathbf{R} = (0.35, 0.15)$ and set $\rho = 1$, $\mu = 0.1$, and $\sigma = 1$. This gives $\text{Re} = 10$ for the Reynolds number. To show that the method is second-order in space, we employ a time step of $\Delta t = h^2$. The solution is computed to final time $T = 0.5$.

We use the jump splice methodology outlined in the previous section, and compare our results to the traditional approach of using smoothed $\delta$ functions to represent surface tension; see \cite{peskin:2002, brackbill:1992, sussman:1999}. More precisely, we compare to using the unspliced approximate projection method with bulk forcing term
\[\mathbf{st} = -\sigma\kappa \bfn \delta^\epsilon(\phi), \]
where $\kappa = \nabla\cdot\bfn$, and
\[\delta^\epsilon(\alpha) = \frac{1}{2\epsilon}\left(1 + \cos\left(\frac{\pi \alpha}{\epsilon}\right)\right), \]
is a smoothed approximation of the Dirac $\delta$. In the following tests, we take $\epsilon = 2h$, which is a standard choice.

% DATA: 02132015_1 and 02132015_2
\begin{table}[ht]
\centering
\begin{tabular}{rccccccccc}
  \toprule
& \multicolumn{4}{c}{$\delta^{2h}$} & \phantom{abc} & \multicolumn{4}{c}{Jump Splice} \\
  \cmidrule{2-5} \cmidrule{7-10}
$n$ & $E_\bfu$ & Rate & $E_p$ & Rate && $E_\bfu$ & Rate & $E_p$ & Rate \\
  \midrule
128     & 1.86\e{-2} &     & 2.79\e{-0} &      && 7.77\e{-2} &      & 6.15\e{-0}  &     \\
256     & 1.00\e{-2} & 0.9 & 2.79\e{-0} & 0.0  && 4.53\e{-3} & 4.1  & 2.10\e{-1} & 4.9  \\
512     & 3.81\e{-3} & 1.4 & 2.93\e{-0} & -0.1 && 1.27\e{-3} & 1.8  & 1.12\e{-1} & 0.9  \\
1024    & 2.05\e{-3} & 0.9 & 2.69\e{-0} & 0.1  && 3.45\e{-4} & 1.9  & 3.48\e{-2} & 1.7  \\
  \bottomrule
\end{tabular}
\caption{For Example \ref{ex:ns-1}, between-grid errors in the velocity ($E_\bfu$) and the pressure ($E_p$) for smoothed $\delta^{2h}$ as well as jump splice.}
\label{tab:ns-example1-grid}
\end{table}

% DATA: 02132015_1 and 02132015_2
\begin{table}[ht]
\centering
\begin{tabular}{rccccc}
  \toprule
& \multicolumn{2}{c}{$\delta^{2h}$} & \phantom{abc} & \multicolumn{2}{c}{Jump Splice} \\
  \cmidrule{2-3} \cmidrule{5-6}
$n$ & $E_{\phi}$ & Rate && $E_{\phi}$ & Rate \\
  \midrule
128     & 4.63\e{-4} &      && 4.09\e{-4} &      \\
256     & 9.31\e{-5} & 2.3  && 5.88\e{-5} & 2.8  \\
512     & 2.39\e{-5} & 2.0  && 1.49\e{-5} & 2.0  \\
1024    & 9.58\e{-6} & 1.3  && 3.70\e{-6} & 2.0  \\
  \bottomrule
\end{tabular}
\caption{For Example \ref{ex:ns-1}, between-grid errors in the interface ($E_\phi$) for smoothed $\delta^{2h}$ as well as jump splice.}
\label{tab:ns-example1-int}
\end{table}

% DATA: 02132015_1 and 02132015_2 (REGULARIZED)
\begin{table}[ht]
\centering
\begin{tabular}{rccccc}
  \toprule
& \multicolumn{2}{c}{$\delta^{2h}$} & \phantom{abc} & \multicolumn{2}{c}{Jump Splice} \\
  \cmidrule{2-3} \cmidrule{5-6}
$n$ & $E_{\rm Vol}$ & Rate && $E_{\rm Vol}$ & Rate \\
  \midrule
64      & 2.55\e{-4} &      && 2.42\e{-4} &      \\
128     & 7.85\e{-5} & 1.7  && 6.25\e{-5} & 2.0  \\
256     & 3.22\e{-5} & 1.3  && 1.61\e{-5} & 2.0  \\
512     & 1.42\e{-5} & 1.2  && 3.99\e{-6} & 2.0  \\
1024    & 6.61\e{-6} & 1.1  && 1.03\e{-6} & 2.0  \\
  \bottomrule
\end{tabular}
\caption{For Example \ref{ex:ns-1}, error in volume of the interface ($E_{\rm Vol}$) for smoothed $\delta^{2h}$ as well as jump splice.}
\label{tab:ns-example1-vol}
\end{table}

\begin{figure}[ht]
\centering
\begin{tabular}{ccc}
\includegraphics[width=0.3\linewidth]{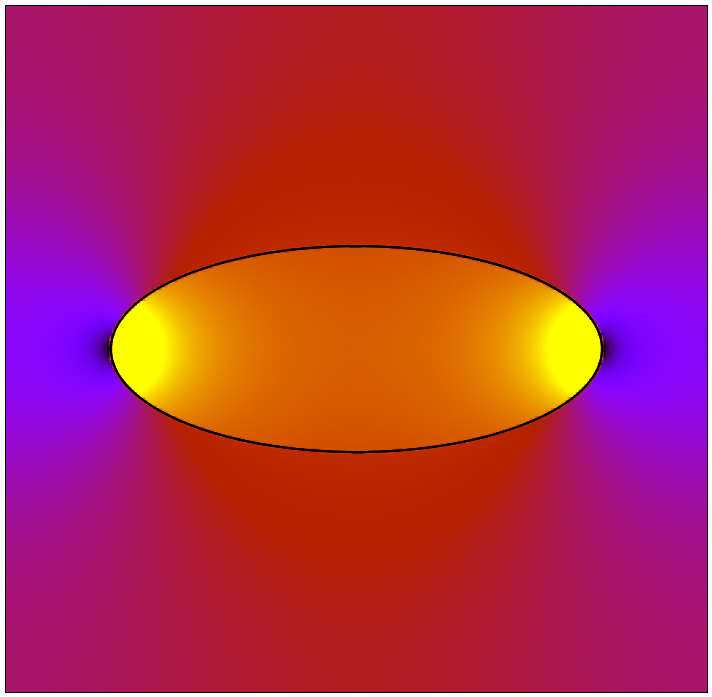} 
&
\includegraphics[width=0.3\linewidth]{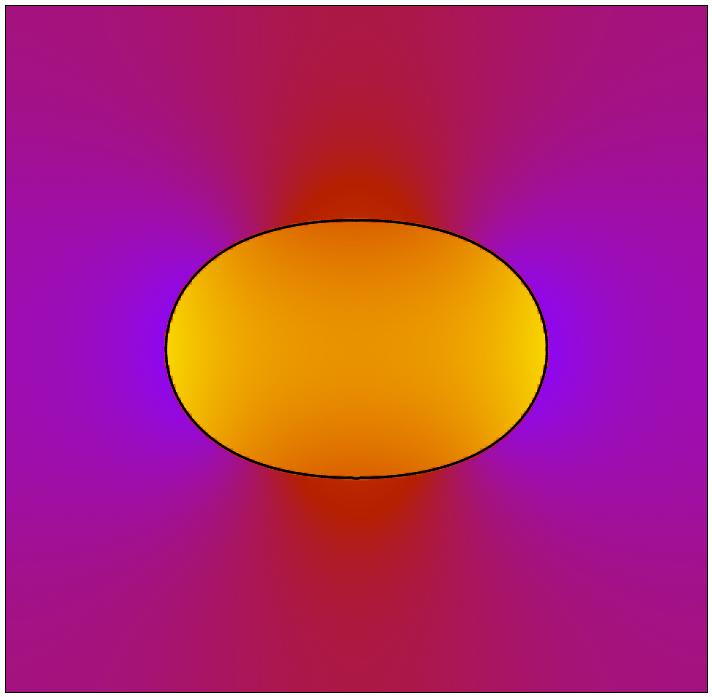} 
&
\includegraphics[width=0.3\linewidth]{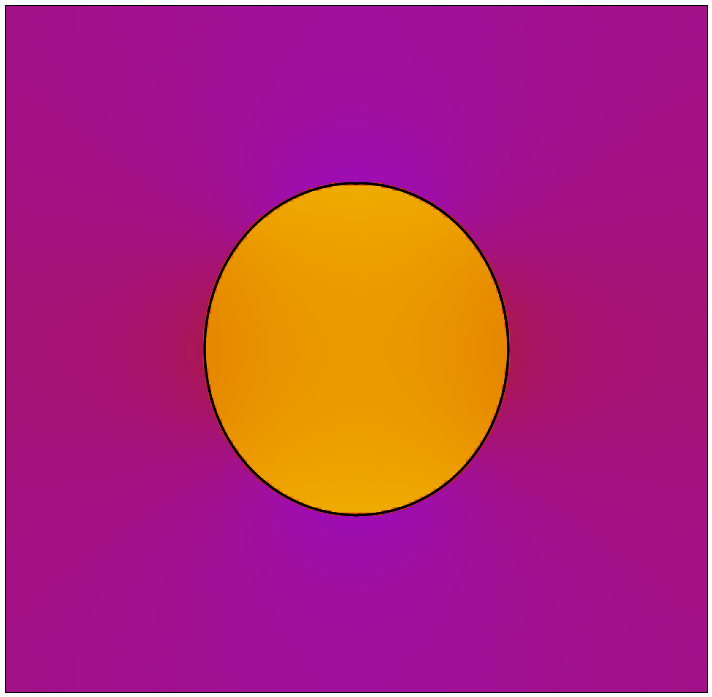} 
\\
$T = 0$ & $T = 0.125$ & $T = 0.25$ 
\end{tabular}
\begin{tabular}{cc}
\includegraphics[width=0.3\linewidth]{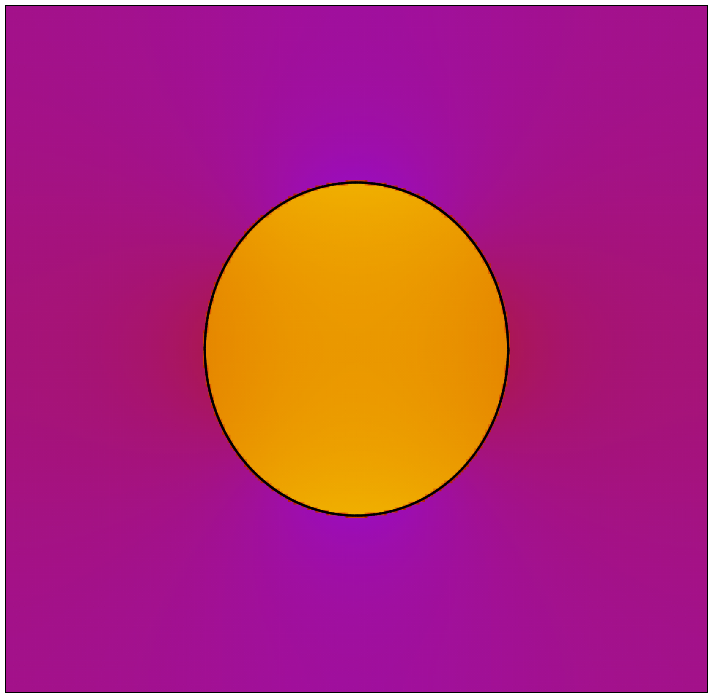} 
&
\includegraphics[width=0.3\linewidth]{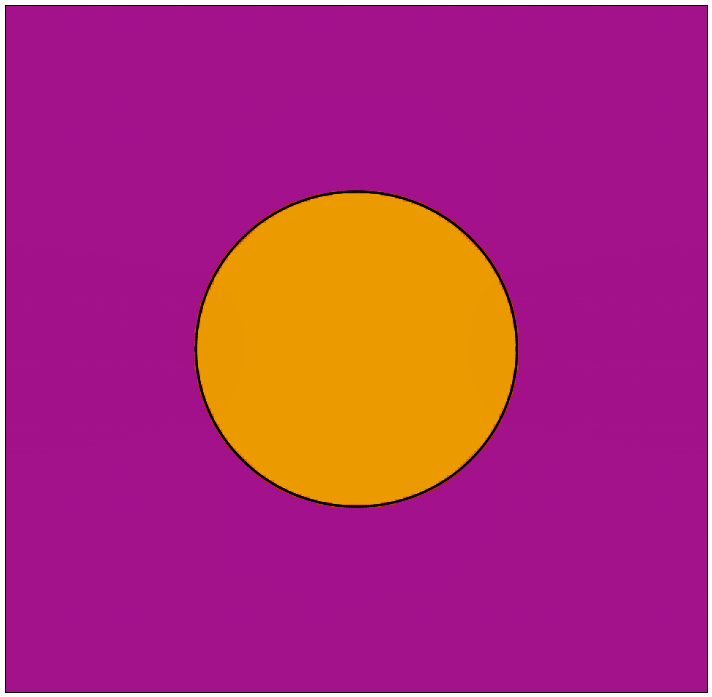}
\\
$T = 0.375$ & $T = 0.5$
\end{tabular}
\includegraphics[width=0.8\linewidth]{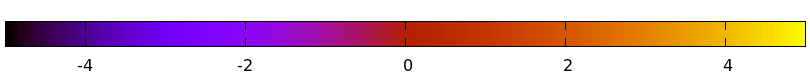}
\caption{Evolution of the interface $\Gamma$ (bold line) and the pressure $p$ in Example \ref{ex:ns-1} on a $256 \times 256$ grid. $\text{Re} = 10$.}
\label{fig:ns-example1}
\end{figure}

Convergence results are shown in Tables \ref{tab:ns-example1-grid}, \ref{tab:ns-example1-int}, and \ref{tab:ns-example1-vol} and Figure \ref{fig:ns-example1} shows the evolution of the interface overlaid on a visual representation of the pressure $p$.

\item \label{ex:ns-2} We repeat Example \ref{ex:ns-1} but with an order of magnitude less viscosity. Now $\mu = 0.01$ and thus $\text{Re} = 100$. Convergence results are shown in Tables \ref{tab:ns-example2-grid}, \ref{tab:ns-example2-int}, and \ref{tab:ns-example2-vol} and Figure \ref{fig:ns-example2} shows the evolution of $\Gamma$ and $p$.

% DATA: 02092015_3_REDO and 02112015_2_REDO
\begin{table}[ht]
\centering
\begin{tabular}{rccccccccc}
  \toprule
& \multicolumn{4}{c}{$\delta^{2h}$} & \phantom{abc} & \multicolumn{4}{c}{Jump Splice} \\
  \cmidrule{2-5} \cmidrule{7-10}
$n$ & $E_\bfu$ & Rate & $E_p$ & Rate && $E_\bfu$ & Rate & $E_p$ & Rate \\
  \midrule
128     & 9.31\e{-2} &      & 2.97\e{-0} &       && 9.96\e{-2} &      & 8.78\e{-0} &      \\
256     & 6.18\e{-2} & 0.6  & 2.83\e{-0} & 0.1   && 1.67\e{-2} & 2.6  & 2.65\e{-0} & 1.7  \\
512     & 3.24\e{-2} & 0.9  & 3.01\e{-0} & -0.1  && 4.22\e{-3} & 2.0  & 6.05\e{-1} & 2.1  \\
1024    & 1.69\e{-2} & 0.9  & 2.65\e{-0} & 0.2   && 1.15\e{-3} & 1.9  & 2.78\e{-1} & 1.1  \\
  \bottomrule
\end{tabular}
\caption{For Example \ref{ex:ns-2}, between-grid errors in the velocity ($E_\bfu$) and the pressure ($E_p$) for smoothed $\delta^{2h}$ as well as jump splice.}
\label{tab:ns-example2-grid}
\end{table}

% DATA: 02092015_3_REDO and 02112015_2
\begin{table}[ht]
\centering
\begin{tabular}{rccccc}
  \toprule
& \multicolumn{2}{c}{$\delta^{2h}$} & \phantom{abc} & \multicolumn{2}{c}{Jump Splice} \\
  \cmidrule{2-3} \cmidrule{5-6}
$n$ & $E_{\phi}$ & Rate && $E_{\phi}$ & Rate \\
  \midrule
128     & 2.65\e{-3} &      && 1.96\e{-3} &      \\
256     & 7.98\e{-4} & 1.7  && 5.01\e{-4} & 2.0  \\
512     & 2.50\e{-4} & 1.7  && 1.27\e{-4} & 2.0  \\
1024    & 8.16\e{-5} & 1.6  && 3.21\e{-5} & 2.0  \\
  \bottomrule
\end{tabular}
\caption{For Example \ref{ex:ns-2}, between-grid errors in the interface ($E_\phi$) for smoothed $\delta^{2h}$ as well as jump splice.}
\label{tab:ns-example2-int}
\end{table}

% DATA: 02092015_3_REDO and 02112015_2
\begin{table}[ht]
\centering
\begin{tabular}{rccccc}
  \toprule
& \multicolumn{2}{c}{$\delta^{2h}$} & \phantom{abc} & \multicolumn{2}{c}{Jump Splice} \\
  \cmidrule{2-3} \cmidrule{5-6}
$n$ & $E_{\rm Vol}$ & Rate && $E_{\rm Vol}$ & Rate \\
  \midrule
64      & 1.52\e{-3} &      && 1.82\e{-3} &      \\
128     & 6.71\e{-4} & 1.2  && 4.73\e{-4} & 2.0  \\
256     & 3.00\e{-4} & 1.2  && 1.20\e{-4} & 2.0  \\
512     & 1.39\e{-4} & 1.1  && 3.03\e{-5} & 2.0  \\
1024    & 6.65\e{-5} & 1.1  && 7.64\e{-6} & 2.0  \\
  \bottomrule
\end{tabular}
\caption{For Example \ref{ex:ns-2}, error in volume of the interface ($E_{\rm Vol}$) for smoothed $\delta^{2h}$ as well as jump splice.}
\label{tab:ns-example2-vol}
\end{table}

\begin{figure}[ht]
\centering
\begin{tabular}{ccc}
\includegraphics[width=0.3\linewidth]{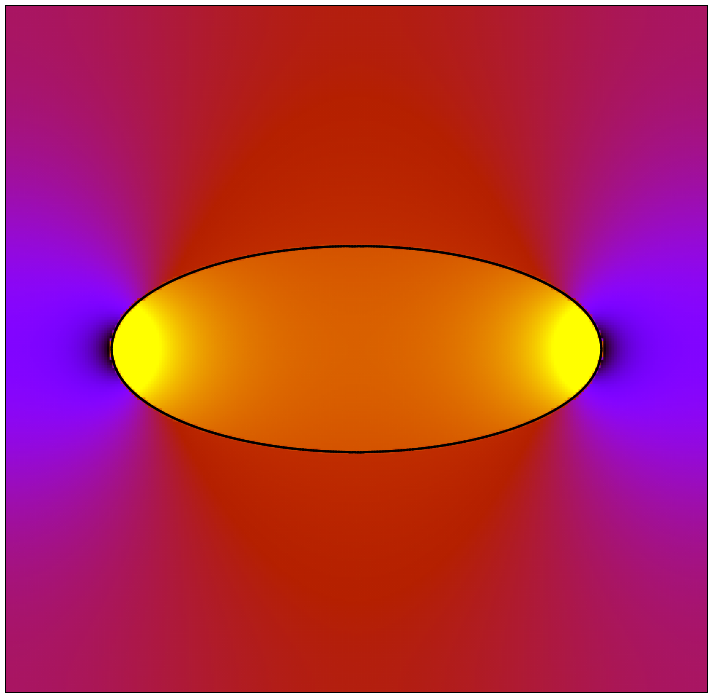} 
&
\includegraphics[width=0.3\linewidth]{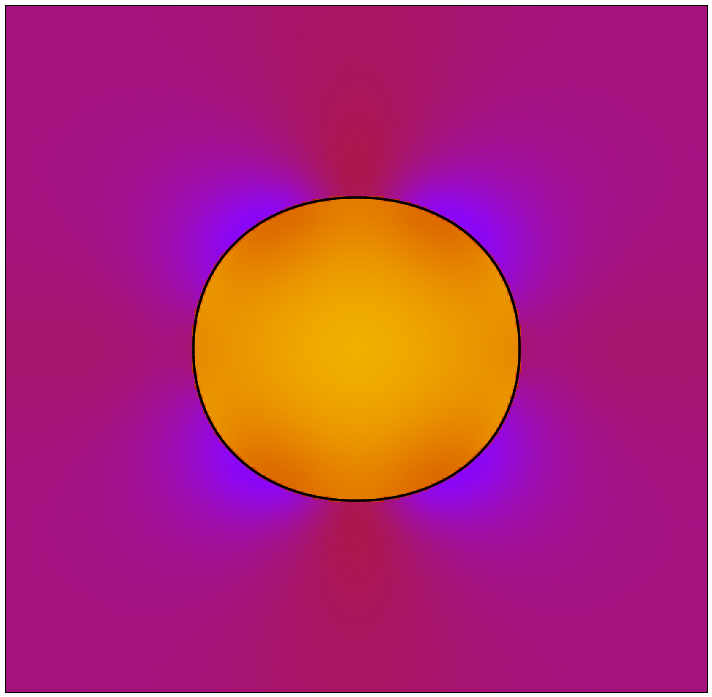} 
&
\includegraphics[width=0.3\linewidth]{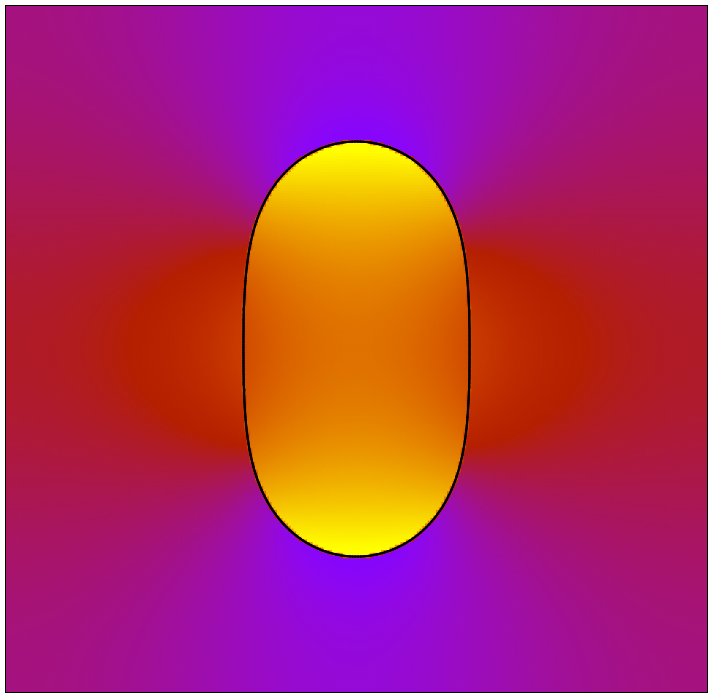} 
\\
$T = 0$ & $T = 0.125$ & $T = 0.25$ 
\end{tabular}
\\
\begin{tabular}{cc}
\includegraphics[width=0.3\linewidth]{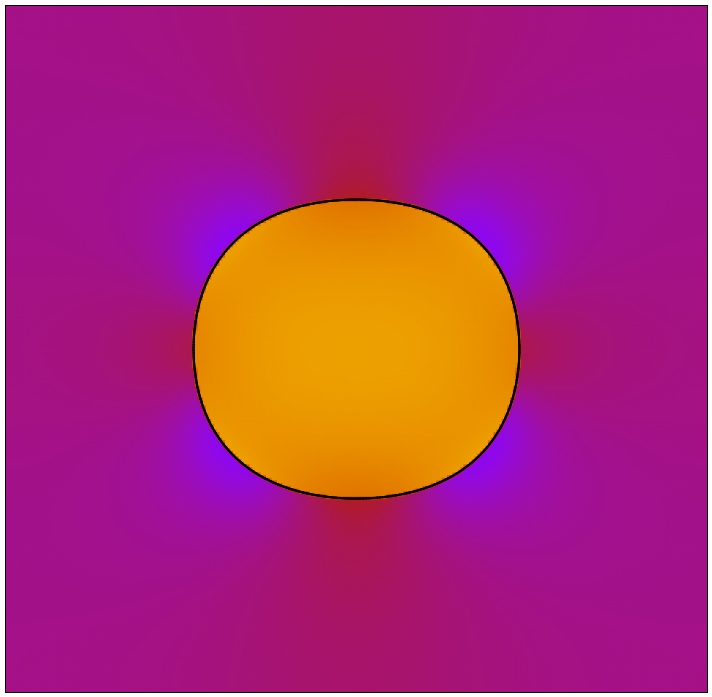} 
&
\includegraphics[width=0.3\linewidth]{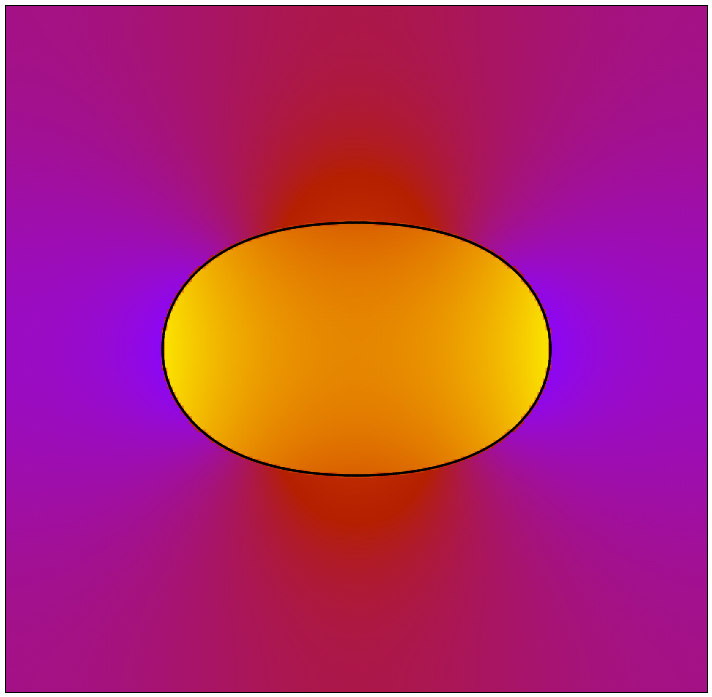}
\\
$T = 0.375$ & $T = 0.5$
\end{tabular}
\includegraphics[width=0.8\linewidth]{figures/bar.png}
\caption{Evolution of the interface $\Gamma$ (bold line) and the pressure $p$ in Example \ref{ex:ns-2} on a $256 \times 256$ grid. $\text{Re} = 100$.}
\label{fig:ns-example2}
\end{figure}

\end{enumerate}

\subsection{Discussion}
\label{sec:ns:discussion}

Examples 1 and 2 above clearly establish second-order convergence in space in velocity, interface position, and volume conservation, with evidence for order $1.5$ convergence in pressure. The traditional smoothed $\delta$ approach, by comparison, shows no convergence in pressure, at best first-order accuracy in velocity and volume, with ambiguously second-order convergence in the position of the interface. On the relatively coarse $256 \times 256$ grid, jump splice methods achieve errors that are 2--4 times smaller than those seen with $\delta^{2h}$.

\begin{figure}[ht]
\centering
\begin{tabular}{cc}
\includegraphics[width=0.4\linewidth]{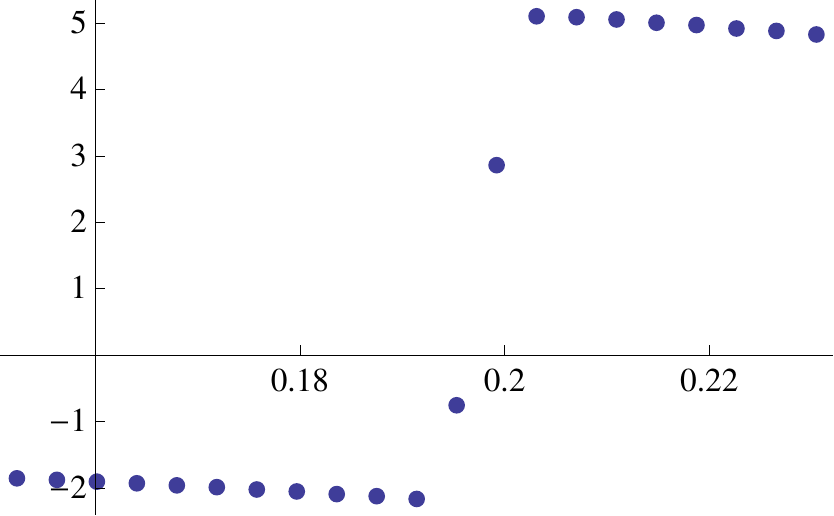} 
&
\includegraphics[width=0.4\linewidth]{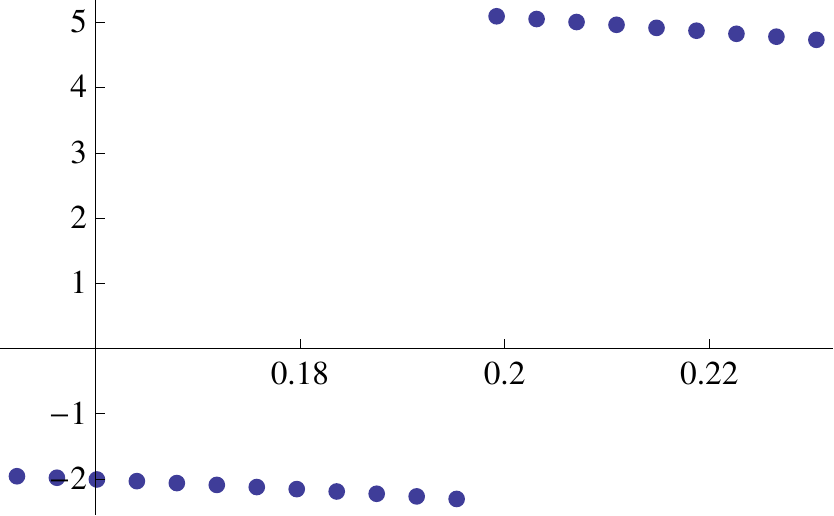} 
\\
$\delta^{2h}$ & Jump Splice
\end{tabular}
\caption{Visualization of pressure near the interface in Example \ref{ex:ns-2} at $T = 0.25$ along the line $x = 0.5$. Jump splicing accurately captures the sharp discontinuity in pressure, whereas use of $\delta^{2h}$ results in artificial smoothing. Results are from simulation on $256 \times 256$ grid.}
\label{fig:ns-pressure-jump}
\end{figure}

Beyond basic convergence properties, the jump splice achieves greater fidelity with respect to the physical formulation of the problem. Figure \ref{fig:ns-pressure-jump} shows $x = 0.5$ cross-sections of pressure near the interface at $T = 0.25$ from Example \ref{ex:ns-2} for both smoothed $\delta$ and jump splice approaches. The jump splice correctly captures a sharp discontinuity in pressure, whereas the $\delta^{2h}$ approach leads to artificial smoothing of the discontinuity.

\begin{figure}[ht]
\centering
\begin{tabular}{cc}
\includegraphics[width=0.4\linewidth]{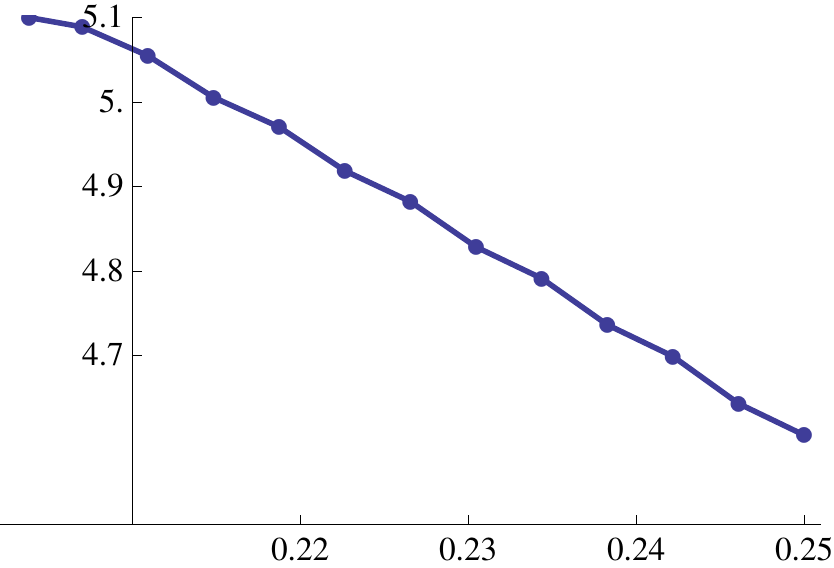} 
&
\includegraphics[width=0.4\linewidth]{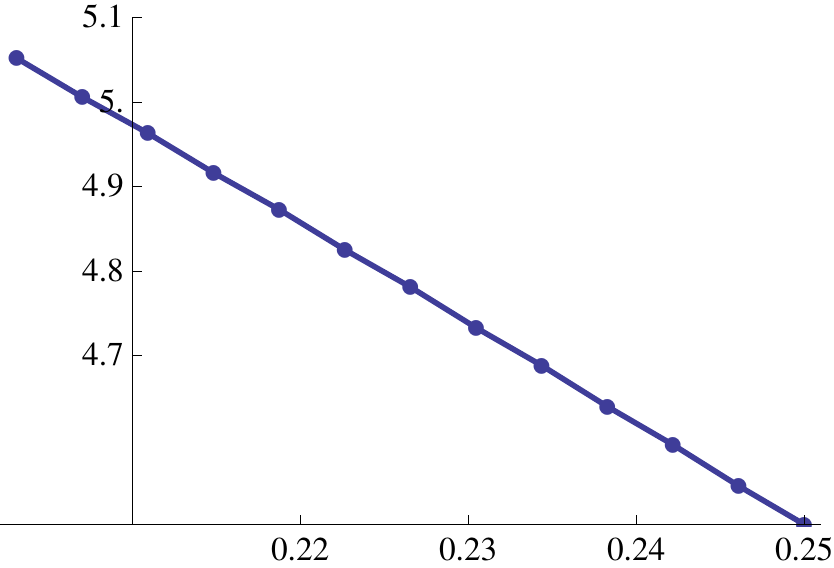} 
\\
$\delta^{2h}$ & Jump Splice
\end{tabular}
\caption{Visualization of pressure in the interior near $\Gamma$ in Example \ref{ex:ns-2} at $T = 0.25$ along the line $x = 0.5$. Note the high-frequency oscillations in the $\delta^{2h}$ result. Results are from simulation on $256 \times 256$ grid.}
\label{fig:ns-pressure-line}
\end{figure}

Use of smoothed $\delta$ functions also results in non-physical high frequency oscillations in pressure in the vicinity of the interface. Figure \ref{fig:ns-pressure-line} shows again an $x = 0.5$ cross-section of pressure from $T = 0.25$ in Example \ref{ex:ns-2}, but this time in the interior of $\Gamma$. Whereas the jump spliced pressure is smooth, the $\delta^{2h}$ pressure shows substantial oscillation with frequency scale $h^{-1}$. 

Finally, note that the techniques outlined in the previous sections work equally well to solve the incompressible Navier-Stokes equations in 3D. As with all jump splice applications, extension to 3D is as simple as changing the finite difference stencil. Indeed, using the 3D versions of $\Delta^h$, $\nabla^h$, and their fourth-order accurate counterparts in the the algorithm outlined in Section \ref{sec:ns:implementation} results in a second-order accurate algorithm in 3D.

\subsection{Summary}
\label{sec:ns:summary}

The jump splice naturally transforms an approximate projection method into a fully second-order in space method for handling strong discontinuities in both the velocity field and the pressure across the interface. In doing so, we achieve asymptotically optimal complexity of $\bigo{N}$ per time step, where $N$ is the number of grid points. The implementation is straightforward and requires solving no additional linear systems. Moreover, the results are significantly more accurate than the traditional smoothed $\delta$ approach, even on relatively coarse grids, and strong discontinuities are captured sharply.

\section*{Acknowledgements}

This work was supported in part by the Applied Mathematical Science subprogram of the Office of Energy Research, U.S. Department of Energy, under Contract Number DE-AC02-05CH11231, and by the Computational Mathematics Program of the National Science Foundation. Some computations used the resources of the National Energy Research Scientific Computing Center, which is supported by the Office of Science of the US Department of Energy under Contract No. DE-AC02-05CH11231. B.P. was also supported by the National Science Foundation Graduate Research Fellowship under Grant Number DGE 1106400.

\clearpage

\section{Appendix}
\label{sec:appendix}

First we show that for $(\Delta^h u)_{i,j}$ to be a second-order accurate approximation to $(\Delta u)(\bfx_{i,j})$, it is enough that $u \in LC^3(U)$ for some open set $U$ containing the cross of the stencil of $\Delta^hu$ at $\bfx_{i,j}$.

\begin{proposition}
\label{prop:laplacian-lipschitz}
Provided that $u \in LC^3(U)$, where $U$ is an open neighborhood of $C_{i,j}$, we have
\[(\Delta u)(\bfx_{i,j}) = (\Delta^h u)_{i,j} + \bigo{h^2}.\]
\end{proposition}

\begin{proof}
Let $\bfx_{i,j} = (x, y)$. Then using Taylor's theorem and that $u \in C^3(C_{i,j})$, we have
\begin{align*}
h^2 (\Delta^h u)_{i,j} - h^2 (\Delta u)(x,y)
&= u(x+h,y) + u(x-h,y) + u(x,y+h) + u(x,y-h) - 4 u(x,y) - h^2 (\Delta u)(x,y) \\
&= \frac{1}{3!}(\partial_x^3 u)(\xi_1,y) h^3 - \frac{1}{3!}(\partial_x^3 u)(\xi_2,y) h^3 + \frac{1}{3!}(\partial_y^3 u)(x,\xi_3) h^3  - \frac{1}{3!}(\partial_y^3 u)(x,\xi_4) h^3,
\end{align*}
where $|x-\xi_k| \leq h$ for $k = 1,2$ and $|y - \xi_k| \leq h$ for $k = 3,4$. Dividing by $h^2$ and using that $\partial_x^3 u$ and $\partial_y^3u$ are Lipschitz continuous with constants $K_x$ and $K_y$, we have
\begin{align*}
|(\Delta^h u)_{i,j} - (\Delta u)(x,y)| 
&\leq \frac{K_x}{3!}|\xi_1 - \xi_2|h + \frac{K_y}{3!}|\xi_3 - \xi_4|h \\
&\leq \left(\frac{2K_x}{3!} + \frac{2K_y}{3!}\right) h^2,
\end{align*}
and this establishes the claim.
\end{proof}

Note we have established Proposition \ref{prop:laplacian-lipschitz} in $\bbr^2$ in order to keep the notation simple; an identical result holds for the 7-point Laplacian in $\bbr^3$. Next, we show that a Lipschitz function defined on each side of the interface can be uniquely extended to a Lipschitz function defined on all of $\Omega$ provided it has zero jump across $\Gamma$.

\begin{proposition}
If $u \in LC(\Omega^+, \Omega^-)$ and $\jump{u} = 0$, then there exists a unique $\tilde u \in LC(\Omega)$ that extends $u$ in the sense that $\subbar{\tilde u}{\Omega^+\cup\Omega^-} = u$.
\label{prop:jump-lipschitz}
\end{proposition}

\begin{proof}
Lipschitz continuity implies uniform continuity, so $u$ is uniformly continuous in both $\Omega^+$ and $\Omega^-$. In particular, $\subbar{u}{\Omega^+}$ can be continuously extended to a function $u^+ \in LC(\overline{\Omega^+})$ and $\subbar{u}{\Omega^-}$ can be similarly extended to $u^- \in LC(\overline{\Omega^-})$. The condition $\jump{u} = 0$ says precisely that $u^+ = u^-$ on $\Gamma$.

Now, consider $\bfx \in \Omega^+$ and $\bfy \in \Omega^-$. Assume for now that the line segment $L = \{ t\bfx + (1-t)\bfy : 0 \leq t \leq 1\}$ intersects $\Gamma$ only once, and let $\bfz$ be the point of intersection. Then $u^+(\bfz) = u^-(\bfz)$, and
\begin{align*}
|u(\bfx) - u(\bfy)|
&= |u^+(\bfx) - u^+(\bfz) + u^-(\bfz) - u^-(\bfy)| \\
&\leq |u^+(\bfx) - u^+(\bfz)| + |u^-(\bfz) - u^-(\bfy)| \\
&\leq K^+ |\bfx - \bfz| + K^-|\bfz - \bfy| \\
&\leq \max\{K^+, K^-\} |\bfx - \bfy|,
\end{align*}
where the last step follows because $\bfz$ lies on the line $L$ between $\bfx$ and $\bfy$. In the case that $L$ intersects $\Gamma$ multiple times, we repeat this process for each point of intersection, and the result remains the same.

Finally, define $\tilde u$ to be equal to $u$ on $\Omega^+ \cup \Omega^-$ and equal to $u^+$ (equivalently, $u^-$) on $\Gamma$. The previous inequality shows that $\tilde u \in LC(\Omega)$ as stated.
\end{proof}

Proposition \ref{prop:jump-lipschitz} is needed to prove the more general result of Proposition \ref{prop:jump-lipschitz-iterated}, which was stated in Section \ref{sec:splice:splice}. In particular, we show that a function $u$ with Lipschitz derivatives up to order $k$ on each side of the interface can be extended to a function with the same property defined on all of $\Omega$.

\begin{proof}[\textbf{Proof of Proposition \ref{prop:jump-lipschitz-iterated}}]
Here we establish the result assuming that $\Gamma$ is $C^2$, and thus the signed distance function $\phi \in LC^2(\Gamma_\epsilon)$ for $\epsilon$ sufficiently small.

From Proposition \ref{prop:jump-lipschitz}, we obtain $\tilde u \in LC(\Omega)$. If $k = 0$, then we are done. Otherwise, recall that
\[ \nabla u = \nabla_s u + (\partial_\bfn u) \bfn, \]
and in particular that
\[\jump{\nabla u} = \nabla_s \jump{u} + \jump{\partial_\bfn u} \bfn. \]
Thus for $k \geq 1$, we have $\jump{\nabla u} = 0$, as $\jump{u} = \jump{\partial_\bfn u} = 0$, and we can apply Proposition \ref{prop:jump-lipschitz} again to $\nabla u$ to obtain $\widetilde{\nabla u} \in LC(\Omega)$. 

Fix $\bfx \in \Gamma$ and let $\bfh \in \bbr^d$. Define $\bfgamma(t) = \bfx + t\bfh$. Because $\phi \in LC^2(\Gamma_\epsilon)$, for $|\bfh|$ sufficiently small we have,
\[ \phi(\bfgamma(t)) = (\bfn(\bfx) \cdot \bfh)t + \frac{1}{2}(\bfh \cdot \nabla \bfn(x) \cdot \bfh)t^2 + \bigo{t^3|\bfh|^3},\]
where $\bfn = \nabla \phi$. Provided that either $\bfn(\bfx) \cdot \bfh \neq 0$ or $\bfh \cdot \nabla \bfn(x) \cdot \bfh \neq 0$, then for sufficiently small $|\bfh|$, we have $\phi(\bfgamma(t)) \neq 0$ for $t \in (0,1)$. It follows that $\bfgamma(t)$ lies entirely in either $\Omega^+$ or $\Omega^-$ for $t \in (0,1)$, and thus we can apply the mean value theorem to $\tilde u(\bfgamma(t))$ and obtain
\begin{align*}
\left|\tilde u(\bfx+\bfh) - \tilde u(\bfx) - \widetilde{\nabla u}(\bfx) \cdot \bfh\right|
&= \left|\nabla \tilde u(\bfxi) \cdot \bfh - \widetilde{\nabla u}(\bfx) \cdot \bfh\right| \\
&= \left|\widetilde{\nabla u}(\bfxi) \cdot \bfh - \widetilde{\nabla u}(\bfx) \cdot \bfh\right| \\
& \leq K|\bfh|^2,
\end{align*}
where $\bfxi = \bfx + t\bfh$ for some $t \in (0,1)$, $K$ is the Lipschitz constant for $\widetilde{\nabla u}$, and we have made use of the observation that $\nabla u(\bfxi) = \nabla \tilde u(\bfxi) = \widetilde{\nabla u}(\bfxi)$ as $\bfxi \notin \Gamma$. In the case that both $\bfn(\bfx) \cdot \bfh = 0$ and $\bfh \cdot \nabla \bfn(x) \cdot \bfh = 0$, we can instead apply the mean value theorem to $\bfgamma(t) + t|\bfh|^2\bfn(\bfx)$ and the conclusion remains the same, up to a constant, as $\tilde u$ is Lipschitz.

This calculation establishes that
\[ \nabla \tilde u(\bfx) = \widetilde{\nabla u}(\bfx), \]
for $\bfx \in \Gamma$, and thus $\nabla \tilde u = \widetilde{\nabla u}$ everywhere. In particular, we have $\nabla \tilde u \in LC(\Omega)$.

Iterating this process up to order $k$ establishes the proposition. Note that the unique $\tilde u$ furnished here is precisely the same as that provided by Proposition \ref{prop:jump-lipschitz}.
\end{proof}

Next, we prove Proposition \ref{prop:v-equivalence}, which was stated in Section \ref{sec:splice:extrapolation}.

\begin{proof}[\textbf{Proof of Proposition \ref{prop:v-equivalence}}]
Let $\zeta = v - \tilde v$. Then $\zeta \in LC^q(\Gamma_\epsilon)$ and $\subbar{\partial_\bfn^i \zeta}{\Gamma} = 0$ for $0\leq i\leq q$. Let $\bfx \in \Gamma_\epsilon$ be arbitrary, and let $\bfy = \bfx - \phi(\bfx)\bfn(\bfx)$ be the closest point to $\bfx$ on $\Gamma$. Then Taylor's theorem provides
\begin{align*}
\zeta(\bfx)
&= \zeta(\bfy + \phi(\bfx)\bfn(\bfx)) \\
&= \sum_{i=0}^{q-1} \frac{\partial_\bfn^i \zeta(\bfy)}{i!} \phi(\bfx)^i + \frac{\partial_\bfn^q \zeta(\bfxi)}{q!} \phi(\bfx)^q
\end{align*}
where $\bfxi = t\bfx + (1-t)\bfy$ for some $t \in (0,1)$ and we have used the fact that $\bfn(\bfx) = \bfn(\bfy)$. But $\bfy \in \Gamma$, so all terms but the last are zero. Moreover, because $\partial_\bfn^q \zeta$ is Lipschitz and $\partial_\bfn^q\zeta(\bfy) = 0$,
\begin{align*}
|\partial_\bfn^q\zeta(\bfxi)|
&= |\partial_\bfn^q\zeta(\bfxi) - \partial_\bfn^q\zeta(\bfy)| \\
&\leq K|\bfxi - \bfy| \\
&\leq K \phi(\bfx), 
\end{align*}
so that, in sum,
\[|\zeta(\bfx)| \leq \frac{K}{q!} \phi(\bfx)^{q+1}, \]
as desired.
\end{proof}

Next, we prove Proposition \ref{prop:modified-v-conditions}, which was stated in Section \ref{sec:splice:calculation}. 

\begin{proof}[\textbf{Proof of Proposition \ref{prop:modified-v-conditions}}]
Clearly the conditions for $g^0$ and $g^1$ are identical between \eqref{eqn:v-conditions} and \eqref{eqn:modified-v-conditions}. Note that, for arbitrary $u$, we can expand $\Delta u$ as
\begin{equation}
\label{eqn:surface-laplacian-expansion}
\Delta u = \Delta_s u + \kappa \partial_\bfn u + \partial_\bfn^2 u,
\end{equation}
where $\kappa = \nabla \cdot \bfn$. Applying this to $u$ and taking jumps, we have
\[ g^2 = g^\Delta - \kappa g^1 - \Delta_s g^0, \]
where we have used that $\jump{\Delta_s u} = \Delta_s g^0$. We can also apply \eqref{eqn:surface-laplacian-expansion} to $v$ and evaluate on $\Gamma$, obtaining
\[ \subbar{\partial_\bfn^2 v}{\Gamma} = g^\Delta - \kappa g^1 - \Delta_s g^0, \]
where we have made use of the fact that $v$ satisfies \eqref{eqn:modified-jump-conditions}. It immediately follows that $\subbar{\partial_\bfn^2 v}{\Gamma} = g^2$, as desired.

Next, note that we can expand $\partial_\bfn\Delta u$ as
\begin{equation}
\label{eqn:dn-laplacian-expansion}
\partial_\bfn\Delta u = \Delta (\partial_\bfn u) - 2\Tr(\nabla\nabla u \cdot \nabla \bfn) - \nabla u \cdot \Delta \bfn.
\end{equation}
Further expanding $\Delta (\partial_\bfn u)$ with \eqref{eqn:surface-laplacian-expansion} and evaluating jumps, we have
\[g^3 = g^{\partial_\bfn \Delta} - \kappa g^2 - \Delta_S g^1 + 2\Tr(\jump{\nabla\nabla u} \cdot \nabla \bfn) + \jump{\nabla u} \cdot \Delta \bfn. \]
Next, we apply \eqref{eqn:dn-laplacian-expansion} to $v$ and obtain
\[ \partial_\bfn^3 v = \partial_\bfn\Delta v - \kappa \partial_\bfn^2 v - \Delta_s (\partial_\bfn v) + 2\Tr(\nabla\nabla v \cdot \nabla \bfn) + \nabla v \cdot \Delta \bfn. \]
Now, because $\nabla\nabla v$ contains derivatives of at most second order and because we have already established that \subbar{\partial_\bfn^i v}{\Gamma} = \jump{\partial_\bfn^i u} for $0 \leq i \leq 2$, it follows that $\subbar{\nabla\nabla v}{\Gamma} = \jump{\nabla\nabla u}$, and similarly that $\subbar{\nabla v}{\Gamma} = \jump{\nabla u}$. Thus we have
\[ \subbar{\partial_\bfn^3 v}{\Gamma} = g^{\partial_\bfn\Delta} - \kappa g^2 - \Delta_s g^1 + \Tr(\jump{\nabla\nabla u} \cdot \nabla \bfn) + \jump{\nabla u} \cdot \Delta \bfn.\]
Comparing these expressions, it immediately follows that $\subbar{\partial_\bfn^3 v}{\Gamma} = g^3$.
\end{proof}

\bibliography{refs}
\bibliographystyle{unsrt}

\end{document}